\documentclass[reqno]{amsart}
\usepackage[left=1in,right=1in,top=1in,bottom=1in]{geometry}
\usepackage{tikz}
\usetikzlibrary{shapes,snakes,calc,arrows}
\usepackage{amsthm}
\usepackage{amscd}
\usepackage{amsfonts}
\usepackage{amsmath}
\usepackage{amssymb}
\usepackage{amsrefs}
\usepackage{mathrsfs}
\usepackage{multirow}
\usepackage{verbatim}
\usepackage{url}
\usepackage{graphicx}
\usepackage{cite}
\usepackage{yfonts}

\begin{document}

%%%%%%%%%%%%%%%%%%%%%%%%
%                       Custom Commands                           %
%%%%%%%%%%%%%%%%%%%%%%%%

%\renewcommand{\theenumii}{\roman{enumii}}
%\renewcommand{\labelenumii}{(\theenumii)}
\newcommand{\supp}{\text{supp}}
\newcommand{\Aut}{\text{Aut}}
\newcommand{\Gal}{\text{Gal}}
\newcommand{\Inn}{\text{Inn}}
\newcommand{\Irr}{\text{Irr}}
\newcommand{\Ker}{\text{Ker}}
\newcommand{\N}{\mathbb{N}}
\newcommand{\Z}{\mathbb{Z}}
\newcommand{\Q}{\mathbb{Q}}
\newcommand{\R}{\mathbb{R}}
\newcommand{\C}{\mathbb{C}}
\renewcommand{\H}{\mathcal{H}}
\newcommand{\B}{\mathcal{B}}
\newcommand{\A}{\mathcal{A}}
\newcommand{\K}{\mathcal{K}}
\newcommand{\M}{\mathcal{M}}
\newcommand{\vphi}{\varphi}

\newcommand{\J}{\mathscr{J}}
\newcommand{\D}{\mathscr{D}}
\renewcommand{\l}{\ell}
\newcommand{\Tr}{\text{Tr}}

\renewcommand{\P}{\mathcal{P}}

\newcommand{\ul}[1]{\underline{#1}}

\newcommand{\I}{\rm{I}}
\newcommand{\II}{\rm{II}}
\newcommand{\III}{\rm{III}}

\newcommand{\<}{\left\langle}
\renewcommand{\>}{\right\rangle}
\renewcommand{\Re}[1]{\text{Re}\ #1}
\renewcommand{\Im}[1]{\text{Im}\ #1}
\newcommand{\dom}[1]{\text{dom}\,#1}
\renewcommand{\i}{\text{i}}
\renewcommand{\mod}[1]{(\operatorname{mod}#1)}
\newcommand{\mb}[1]{\mathbb{#1}}
\newcommand{\mc}[1]{\mathcal{#1}}
\newcommand{\mf}[1]{\mathfrak{#1}}
\newcommand{\im}{\operatorname{im}}

\newcommand{\TODO}[1]{{\color{red}\textbf{TODO: }{#1}}}

%%%%%%%%%%%%%%%%%%%%%%%%
%                      Theorem Environments                      %
%%%%%%%%%%%%%%%%%%%%%%%%

\newtheorem{thm}{Theorem}[section]
\newtheorem{prop}[thm]{Proposition}
\newtheorem{lem}[thm]{Lemma}
\newtheorem{cor}[thm]{Corollary}
\newtheorem{innercthm}{Theorem}
\newenvironment{cthm}[1]
  {\renewcommand\theinnercthm{#1}\innercthm}
  {\endinnercthm}
\newtheorem{innerclem}{Lemma}
\newenvironment{clem}[1]
  {\renewcommand\theinnerclem{#1}\innerclem}
  {\endinnerclem}

\theoremstyle{definition}
\newtheorem{defi}[thm]{Definition}
\newtheorem{ex}[thm]{Example}
\newtheorem*{exs}{Examples}
\newtheorem{rem}[thm]{Remark}
\newtheorem{innercdefi}{Definition}
\newenvironment{cdefi}[1]
  {\renewcommand\theinnercdefi{#1}\innercdefi}
  {\endinnercdefi}

%for proofs with equivlances where you want to assert which implication is being shown
\newtheorem*{imp}{}
%%%%%%%%%%%%%%%%%%%%%%%%
%                        	  Title Page                                  %
%%%%%%%%%%%%%%%%%%%%%%%%

\title{On finite free Fisher information for eigenvectors of a modular operator}
\author{Brent Nelson}
\address{Department of Mathematics, University of California, Berkeley}
\email{brent@math.berkeley.edu}

\thanks{Research supported by NSF grants DMS-1502822 and DMS-0838680}

\begin{abstract}
Suppose $M$ is a von Neumann algebra equipped with a faithful normal state $\varphi$ and generated by a finite set $G=G^*$, $|G|\geq 2$. We show that if $G$ consists of eigenvectors of the modular operator $\Delta_\varphi$ with finite free Fisher information, then the centralizer $M^\varphi$ is a $\II_1$ factor and $M$ is either a type $\II_1$ factor or a type $\III_\lambda$ factor, $0<\lambda\leq 1$, depending on the eigenvalues of $G$. Furthermore, $(M^\vphi)'\cap M=\C$, $M^\vphi$ does not have property $\Gamma$, and $M$ is full provided it is type $\III_\lambda$, $0<\lambda<1$.
\end{abstract}

\maketitle

%%%%%%%%%%%%%%%%%%%%%%%%%%%%%%%%%%%%%%%%%%%%%%%%%%%%%%%%%%%%%%%%%%%%%%%%%%%%%
%                   			Introduction   		                       %
%%%%%%%%%%%%%%%%%%%%%%%%%%%%%%%%%%%%%%%%%%%%%%%%%%%%%%%%%%%%%%%%%%%%%%%%%%%%%

\section*{Introduction}

Given random variables $x_1,\ldots, x_n$ in a non-commutative probability space $(M,\varphi)$, it is natural to ask what information about the distribution of a polynomial $p\in \C\<x_1,\ldots, x_n\>$ can be gleaned from the distributions of $x_1,\ldots, x_n$. If $p=x_1+x_2$ or $p=x_1x_2$ with $x_1$ freely independent from $x_2$, the theory of free additive and multiplicative convolutions tells us everything about the distribution of $p$, but (until recently) without the strict \emph{regularity condition} of free independence little could be deduced about the distribution of a general polynomial.

Shlyakhtenko and Skoufranis studied the distributions of matrices of polynomials in freely independent random variables $x_1,\ldots, x_n$ and their adjoints, and in particular showed that if $x_1,\ldots, x_n$ were semicircular random variables, then any self-adjoint polynomial has diffuse spectrum \cite{SS15}. Mai, Speicher, and Weber later improved upon this result by showing that if $x_1,\ldots, x_n$ are self-adjoint random variables, not necessarily freely independent or having semicircular distributions but instead having finite free Fisher information, then $x_1,\ldots, x_n$ are algebraically free, any non-constant self-adjoint polynomial $p\in \C\<x_1,\ldots, x_n\>$ has diffuse spectrum, and $W^*(x_1,\ldots, x_n)$ contains no zero divisors for $\C\<x_1,\ldots, x_n\>$ \cite{MSW15}. Charlesworth and Shlyakhtenko further improved on this result by weakening the assumption of finite free Fisher information to having full free entropy dimension, and showed that under stronger assumptions on $x_1,\ldots, x_n$ one can assert that the spectral measure of $p\in\C\<x_1,\ldots, x_n\>$ is non-singular \cite{CS15}. These techniques have since been applied by Hartglass to show that certain elements in $C^*$-algebras associated to weighted graphs have diffuse spectrum \cite{Har15}. In this paper, these techniques are brought to bear on non-tracial von Neumann algebras.

We consider a von Neumann algebra $M$ with a faithful normal state $\varphi$, and a finite generating set $G$. We will further assume that $G$ has finite free Fisher information with respect to the state $\varphi$, and that each $y\in G$ is an ``eigenoperator''; that is, scaled by the modular automorphism group: $\sigma_t^\varphi(y)=\lambda_y^{it} y$ for some $\lambda_y>0$. Under these assumptions, we obtain a criterion for when polynomials $\C\<G\>$ in the centralizer $M^\vphi$ have diffuse spectrum (\emph{cf.} Corollary \ref{diffuse_centralizer_elements}). Our context is inspired by Shlyakhtenko's free Araki-Woods factors, which are non-tracial von Neumann algebras generated by \emph{generalized circular elements} (operators scaled by the action of the modular automorphism group, \emph{cf.} \cite[Section 4]{Shl97}).

Regularity conditions on $x_1,\ldots, x_n$ can also have consequences on the von Neumann algebra generated by these operators. Indeed, Dabrowski \cite{Dab10} showed that if $x_1,\ldots, x_n$ in a tracial non-commutative probability space have finite free Fisher information, then these operators generate a factor without property $\Gamma$. The non-tracial analogue of this result, which considers the centralizer $M^\vphi$ as well as $M$, is the content of the two main results of this paper. The first is concerned with factoriality:

\begin{cthm}{A}\label{A}
Let $M$ be a von Neumann algebra with a faithful normal state $\varphi$. Suppose $M$ is generated by a finite set $G=G^*$, $|G|\geq 2$, of eigenoperators of $\sigma^\varphi$ with finite free Fisher information. Then $(M^\vphi)'\cap M=\C$. In particular, $M^\varphi$ is a $\II_1$ factor and if $H< \R_+^\times$ is the closed subgroup generated by the eigenvalues of $G$ then
	\begin{align*}
		\text{$M$ is a factor of type} \left\{\begin{array}{cl} \rm{III}_1 & \text{if  }\ H=\R_+^\times\\
													\rm{III}_\lambda	& \text{if }\ H=\lambda^\Z,\ 0<\lambda<1\\
													\rm{II}_1 & \text{if }\ H=\{1\}. \end{array}\right.
	\end{align*}
\end{cthm}

Lacking property $\Gamma$ is, for tracial von Neumann algebras, equivalent to the more general property of a von Neumann algebra being ``full'' (\emph{cf.} Subsection \ref{full}). Consequently, the following theorem is the other half of the non-tracial analogue to Dabrowski's result:

\begin{cthm}{B}\label{B}
Let $M$ be a von Neumann algebra with a faithful normal state $\varphi$. Suppose $M$ generated by a finite set $G=G^*$, $|G|\geq 2$, of eigenoperators of $\sigma^\varphi$ with finite free Fisher information. Then $M^\varphi$ does not have property $\Gamma$. Furthermore, if $M$ is a type $\III_\lambda$ factor, $0<\lambda<1$, then $M$ is full.
\end{cthm}

The structure of the paper is as follows. In Section \ref{prelim} we recall various notions relevant to the study of non-tracial von Neumann algebras. We also recall the definition of Dirichlet and completely Dirichlet forms. The context of our results is established in Section \ref{Generators}, wherein eigenoperators are defined and studied. In Section \ref{diff_calc} we analyze derivations on the non-tracial von Neumann algebra $M$, conjugate variables, and free Fisher information. Of particular interest are ``$\mu$-modular'' derivations, namely those derivations that interact nicely with the modular automorphism group $\sigma^\vphi$. Section \ref{Closable} is dedicated to the study of closable $\mu$-modular derivations. In order to show that these closures are still derivations when restricted to the centralizer $M^\vphi$, we study Dirichlet forms arising from $\mu$-modular derivations. This is also used to establish a type of Kaplansky's density theorem for operators in the domain of a $\mu$-modular derivation. Contraction resolvents associated to these derivations are also considered here. In Section \ref{diffuse} we produce a criterion for when polynomials $\C\<G\>$ in $M^\vphi$ are diffuse, and deduce when monomials in $M^\vphi$ have an atom at zero and of what size. Section \ref{polar_deriv} combines the analysis of the previous two sections to show that derivations associated to $y\in G$ give rise to derivations (enjoying many of the same properties) that are related to the polar decomposition of $y$. Furthermore, the derivations associated to $|y|$ for $y\in G$ are in some sense tracial derivations, which we exploit in the proofs of our main theorems in Section \ref{main_results}. Theorem \ref{A} is proven using a contraction resolvent argument similar to the one used in \cite{Dab10}. The type classifications of these von Neumann algebras are deduced using the well-known invariants recalled in Section \ref{prelim}. Theorem \ref{B} is proven by using the derivations associated to $|y|$ for $y\in G$ to appeal to a tracial result of Curran, Dabrowski, and Shlyakhtenko from \cite{CDS14}.

\subsection*{Acknowledgments} I would like to thank Dimitri Shlyakhtenko for the initial idea of this paper and his guidance in the early stages. I am indebted to Yoann Dabrowski, who greatly accelerated the progress of this paper with discussions we had while at Mathematisches Forschungsinstitut Oberwolfach, and suggested proofs that significantly improved Theorem B. Fabio Cipriani also helped immensely by guiding me through the literature on Dirichlet forms. Thanks to Ian Charlesworth, Michael Hartglass, and Benjamin Hayes who all provided an abundance of support and advice. Finally, I would also like to thank the anonymous referee for their helpful comments and suggestions.

%%%%%%%%%%%%%%%%%%%%%%%%%%%%%%%%%%%%%%%%%%%%%%%%%%%%%%%%%%%%%%%%%%%%%%%%%%%%%
%                    			 Preliminaries                  	       %
%%%%%%%%%%%%%%%%%%%%%%%%%%%%%%%%%%%%%%%%%%%%%%%%%%%%%%%%%%%%%%%%%%%%%%%%%%%%%

\section{Preliminaries}\label{prelim}

Throughout the paper, $M$ will denote a von Neumann algebra with a faithful normal state $\vphi$.

%%%%%%%%%%%%%%%%%%%%%%%%%%%%%%%%%%%%%%%%%%%%%%%%%%%%%%%%%%%%%%%%%%%%%%%%%%%%%

\subsection{Arveson spectrum and Connes's $S(M)$ invariant}\label{arveson_connes}

Suppose $M$ is a factor. We recall some invariants for later use in establishing the type classification of $M$ in Theorem \ref{A}. The following exposition can be found in greater generality in \cite[Chapter XI]{Tak03}.

Identify $\R_+^\times$ as the dual group of $\R$ via the pairing
	\begin{align*}
		\R\otimes \R_+^\times \ni (t,\lambda)\mapsto \lambda^{it},
	\end{align*}
so that the \emph{Fourier transform} $\mc{F}$ on $L^1(\R)$ is defined
	\begin{align*}
		(\mc{F} f)(\lambda) = \int_\R \lambda^{-it} f(t)\ dt,\qquad f\in L^1(\R).
	\end{align*}
Denote by $A(\R_+^\times):=\mc{F}(L^1(\R))$, and for $f\in A(\R_+^\times)$ let $\check{f}$ denote its inverse image under the Fourier transform. When $f$ is integrable this can be computed by the inverse Fourier transform:
	\begin{align*}
		\check{f}(t)= \int_{\R_+^\times} \lambda^{it} f(\lambda)\ d\lambda\qquad f\in A(\R_+^\times)\cap L^1(\R_+^\times).
	\end{align*}

The following definitions are due to Arveson (\emph{cf.} \cite{Arv74}), but we use the notation from \cite[Chapter XI]{Tak03}. For all $x\in M$ and $f\in A(\R_+^\times)$, denote
	\begin{align*}
		\sigma_f(x) = \int_\R \check{f}(-t) \sigma_t^\varphi(x)\ dt.
	\end{align*}
Define for each $x\in M$
	\begin{align*}
		I(x) = \{f\in A(\R_+^\times)\colon \sigma_f^\varphi(x)=0\}.
	\end{align*}
The \emph{$\sigma^\varphi$-spectrum of $x$} is defined
	\begin{align*}
		\text{Sp}_{\sigma^\varphi}(x):=\{\lambda\in \R_+^\times\colon f(\lambda)=0,\ f\in I(x)\},
	\end{align*}
and the \emph{Arveson spectrum of $\sigma^\varphi$} is defined
	\begin{align*}
		\text{Sp}(\sigma^\varphi) := \left\{ \lambda \in \R_+^\times\colon f(\lambda)=0,\ f\in \bigcap_{x\in M} I(x)\right\}.
	\end{align*}

\begin{lem}[{\cite[Lemma XI.1.3.(v)]{Tak03}}]\label{M_0_Takesaki_Lemma}
For an open subset $U\subset \R_+^\times$ define $M^{\sigma^\varphi}_0(U)$ as the weak$^*$ closed subspace of $M$ spanned by
	\begin{align*}
		\left\{\sigma_f^\varphi(x)\colon x\in M,\ f\in A(\R_+^\times)\text{ with }\text{supp}(f)\subset U\right\}.
	\end{align*}
Then $\lambda\in \R_+^\times$ belongs to $\text{Sp}(\sigma^\varphi)$ if and only if $M^{\sigma^\varphi}_0(U)\neq \{0\}$ for every open neighborhood $U$ of $\lambda$.
\end{lem}

Given a projection $p\in M^\varphi$, the restriction of $\sigma^\varphi$ to $pMp$ is denoted $(\sigma^\varphi)^p$. The \emph{Connes spectrum of $\sigma^\varphi$} \cite[Definition 2.2.1]{Con73} is defined
	\begin{align*}
		\Gamma(\sigma^\varphi) := \bigcap_{p} \text{Sp}( (\sigma^\varphi)^p),
	\end{align*}
where the intersection is over non-zero projections $p\in M^\vphi$.

\begin{lem}[{\cite[Proposition 2.2.2.(c)]{Con73}}]\label{Connes_second_lemma}
If $M$ is a factor with faithful normal state $\varphi$ such that $M^\varphi$ is a factor, then $\Gamma(\sigma^\varphi)=\text{Sp}(\sigma^\varphi)$.
\end{lem}

For a faithful semi-finite normal weight $\psi$ on $M$, let $S_\psi$ denote the closure of the map $M\ni x\mapsto x^*$ when viewed as a densely defined operator on $L^2(M,\psi)$. The polar decomposition $S_\psi=J_\psi \Delta_\psi^{1/2}$ yields an anti-linear isometry $J_\psi$ and the modular operator $\Delta_\psi$, which is a positive non-singular (unbounded) operator densely defined on $L^2(M,\psi)$. Suppose $M$ is a factor. Denoting the set of all faithful semi-finite normal weights on $M$ by $\mc{W}_0$, the \emph{modular spectrum of $M$} \cite[Definition 3.1.1]{Con73} is defined
	\begin{align*}
		S(M):=\bigcap_{\psi\in \mc{W}_0} \text{spectrum}(\Delta_\psi).
	\end{align*}
If $S(M)=\{1\}$, then $M$ is semi-finite. Otherwise $M$ is a type $\rm{III}$ factor and $S(M)$ determines its type classification:
	\begin{align*}
		\text{$M$ is of type}\left\{\begin{array}{cl}
				\rm{III}_0 & \text{if }S(M)=\{0,1\}\\
				\rm{III}_\lambda & \text{if }S(M)=\{0\}\cup\lambda^\Z \text{ for } 0<\lambda<1\\
				\rm{III}_1 & \text{if }S(M)=[0,+\infty).
			\end{array}\right.
	\end{align*}
	
\begin{lem}[{\cite[Theorem 3.2.1, Corollary 3.2.7]{Con73}}]\label{Connes_first_lemma}
If $M$ is a factor with faithful normal state $\varphi$ then $S(M)\cap \R_+^\times = \Gamma(\sigma^\varphi)$, and if $M^\varphi$ is a factor then $S(M)=\text{spectrum}(\Delta_\varphi)$.
\end{lem}

%%%%%%%%%%%%%%%%%%%%%%%%%%%%%%%%%%%%%%%%%%%%%%%%%%%%%%%%%%%%%%%%%%%%%%%%%%%%%

\subsection{Full von Neumann algebras}\label{full}

Let $\text{Aut}(M)$ denote the group of automorphisms on $M$, and let $\text{Int}(M)$ denote the group of inner automorphisms (i.e. those automorphisms implemented via conjugation by a unitary $u\in M$). On $\text{Aut}(M)$ we consider the topology of point-wise norm convergence in $M_*$: a net $\{\alpha_\iota\}\subset \text{Aut}(M)$ converges to $\alpha\in \text{Aut}(M)$ if for every $\phi\in M_*$ we have
	\begin{align*}
		0=\lim_{\iota} \| \phi\circ(\alpha_\iota-\alpha)\|_{M_*} = \lim_{\iota} \sup_{x\in M} \frac{|\phi(\alpha_\iota(x) - \alpha(x))|}{\|x\|}.
	\end{align*}

\begin{defi}[{\cite[Definition 3.5]{Con74}}]
A von Neumann algebra $M$ with separable predual $M_*$ is \emph{full} when $\text{Int}(M)$ is closed in $\text{Aut}(M)$ with respect to the above topology.
\end{defi}

Given $\phi\in M_*$ and $x\in M$ we define $[x,\phi]\in M_*$ by $[x,\phi](y):=\phi([y,x])$.

\begin{defi}
For $\phi\in M_*$, an operator-norm-bounded sequence $(z_j)_{j\in\N}$ in $M$ is said to be in the \emph{asymptotic centralizer with respect to $\phi$} if $\|[z_j,\phi]\|_{M_*}\to 0$.
\end{defi}

Theorem 3.1 of \cite{Con74} tells us the following two conditions are equivalent for a von Neumann algebra $M$ with separable predual:
	\begin{enumerate}
		\item[(i)] $M$ is full.
		\item[(ii)] Whenever a bounded sequence $(z_j)_{j\in \N}$ in $M$ is in the asymptotic centralizer with respect to $\phi$ for all $\phi\in M_*$, there exists a bounded sequence of scalars $(c_j)_{j\in \N}$ so that $z_j-c_j\to 0$ $*$-strongly.
	\end{enumerate}
For any faithful normal state $\psi$, we define a norm on $M$ by
	\[
		\|x\|_\psi^\sharp :=\sqrt{ \|x\|_\psi^2+\|x^*\|_\psi^2}\qquad x\in M,
	\]
and recall that for uniformly bounded sequences in $M$ convergence with respect to this norm coincides with $*$-strong convergence. It is an easy exercise to see that the sequence of scalars in condition (ii) can be replaced with $(\psi(z_j))_{j\in \N}$ for any faithful normal state $\psi$ on $M$. Moreover, $z_j$ can be replaced with $z_j - \psi(z_j)$ so that we need only consider $\psi$-centered sequences. Finally, when $M$ is a $\rm{II}_1$ factor this condition is equivalent to not having property $\Gamma$.

%%%%%%%%%%%%%%%%%%%%%%%%%%%%%%%%%%%%%%%%%%%%%%%%%%%%%%%%%%%%%%%%%%%%%%%%%%%%%

\subsection{Entire elements and a bimodule structure for $L^2(M,\varphi)$}

While $L^2(M,\varphi)$ is a left $M$-module, unlike in the tracial case, the right action of $M$ on itself does not in general extend to a bounded action on $L^2(M,\varphi)$. However, there is a $*$-subalgebra of $M$ containing the centralizer $M^\varphi$ for which the right action extends to a bounded action on $L^2(M,\varphi)$.

Recall that a map from a complex domain valued in a Banach space is analytic if it can locally be expressed as a norm-convergent power series with coefficients in the Banach space. Such a map is entire if the domain is $\C$.

\begin{defi}[{\cite[Definition VIII.2.2]{Tak03}}]
An element $x\in M$ is said to be \emph{entire} if the $M$-valued map $\R\ni t\mapsto \sigma_t^\varphi(x)$ has an extension to an entire function, denoted by $f_x\colon \C\to M$. The set of entire elements will be denoted $M_\infty$.
\end{defi}

Given $x\in M_\infty$ and $a\in M$, we compute
	\begin{align*}
		\| ax \|_\varphi = \| J_\varphi ax \|_\varphi = \| \Delta_\varphi^{1/2} x^* a^*\|_\varphi \leq \| f_{x^*}(-i/2)\| \|\Delta_\varphi^{1/2} a^*\|_\varphi = \|f_{x^*}(-i/2)\| \| a\|_\varphi.
	\end{align*}
Thus this right action of $M_\infty$ extends to $L^2(M,\varphi)$ and $L^2(M,\varphi)$ is an $M$-$M_\infty$ bimodule.

\begin{rem}\label{right_action} We note that for $\xi\in L^2(M,\varphi)$ and $x\in M_\infty$ the right action we are considering is equivalent to
	\begin{align*}
		\xi\cdot x = (J_\varphi \sigma_{i/2}^\varphi(x)^* J_\varphi) \xi= S_\vphi x^* S_\vphi \xi,
	\end{align*}
whereas the usual right action of $M$ on $L^2(M,\vphi)$ is given by $J_\vphi x^* J_\vphi \xi$. We make this adjustment for the sake of the derivations; that is, so that when $\xi\in M$, the action of $x$ is simply right-multiplication. In Section \ref{technical_estimate_section} we will make use of the usual right action.
\end{rem}

%%%%%%%%%%%%%%%%%%%%%%%%%%%%%%%%%%%%%%%%%%%%%%%%%%%%%%%%%%%%%%%%%%%%%%%%%%%%%

\subsection{Dirichlet forms}

We refer the reader to \cite[Section 4]{Cip97} and \cite[Section 4]{Cip08} for further details. We begin with a discussion of standard forms of a von Neumann algebra. Dirichlet forms are defined in terms of a certain monotonicity condition on Hilbert spaces, and hence require a choice of positive cone. In our context, the Hilbert space will be $L^2(M,\varphi)$ and a standard form offers a convenient choice of positive cone.

Since $M\subset \dom(S_\varphi)=\dom(\Delta_\vphi^{1/2})$, we have $M\subset \dom(\Delta_\vphi^{1/4})$. Let $L^2_+(M,\varphi)=\overline{\Delta_\varphi^{1/4}M_+}$, then  $(M,L^2(M,\varphi),L^2_+(M,\varphi),J_\varphi)$ is a standard form of $M$. A vector $\xi\in L^2(M,\varphi)$ is \emph{$J_\varphi$-real} if $J_\varphi \xi=\xi$; we let $L^2_{s.a.}(M,\varphi)$ denote the closed subspace of $J_\vphi$-real vectors. We note that
	\[
		L^2_{s.a.}(M,\varphi)=\overline{\Delta_\varphi^{1/4} M_{s.a.}}.
	\]

We will consider complex-valued, sesquilinear forms (linear in the right entry)
	\[
		\mathscr{E}\colon \dom(\mathscr{E})\times \dom(\mathscr{E})\to \C,
	\]
defined on a dense subspace $\dom(\mathscr{E})\subset L^2(M,\varphi)$, along with the associated quadratic form $\mathscr{E}[\xi] := \mathscr{E}(\xi,\xi)$, $\xi\in\dom(\mathscr{E})$. The quadratic form is \emph{$J_\varphi$-real} if $J_\varphi \dom(\mathscr{E})\subset \dom(\mathscr{E})$ and  $\mathscr{E}[J_\varphi \xi]= \overline{\mathscr{E}[\xi]}$ for all $\xi\in \dom(\mathscr{E})$. We will restrict our attention to the case when $\mathscr{E}[\cdot]$ is valued in $[0,\infty)$. Such forms are \emph{closed} if $\dom(\mathscr{E})$ is a Hilbert space with the inner product
	\[
		\<\xi,\eta\>_\mathscr{E}:=\<\xi,\eta\>_\varphi + \mathscr{E}(\xi,\eta),
	\]
and are \emph{closable} if the identity map from $\dom(\mathscr{E})$ to $L^2(M,\varphi)$ extends injectively to the Hilbert space completion of $\dom(\mathscr{E})$ with respect to $\|\cdot\|_{\mathscr{E}}$. Equivalently, $\mathscr{E}$ is closed if whenever $\{\xi_n\}_{n\in\N}\subset \dom(\mathscr{E})$ is a Cauchy sequence with respect to $\|\cdot\|_\mathscr{E}$ and $\xi_n \to \xi$ with respect to $\|\cdot\|_\varphi$, then $\xi\in \dom(\mathscr{E})$ and $\xi_n \to \xi$ with respect to $\|\cdot\|_{\mathscr{E}}$. $\mathscr{E}$ is closable if whenever  $\{\xi_n\}_{n\in\N}\subset \dom(\mathscr{E})$ is a Cauchy sequence with respect to $\|\cdot\|_{\mathscr{E}}$ and $\xi_n\to 0$ with respect to $\|\cdot\|_\varphi$, then $\lim_n \mathscr{E}[\xi_n]=0$. In particular, if $\mathscr{E}[\xi]=\| T\xi\|_{\varphi}^2$ for some operator $T$ with $\dom(T)=\dom(\mathscr{E})$, then $\mathscr{E}$ is closed (resp. closable) if and only if $T$ is closed (resp. closable). 

If one extends a closed form $\mathscr{E}$ to all of $L^2(M,\varphi)$ by letting $\mathscr{E}\equiv +\infty$ outside of $\dom(\mathscr{E})$, then $\mathscr{E}$ is lower semicontinuous: whenever $\{\xi_n\}_{n\in\N}\subset \dom(\mathscr{E})$ converges to $\xi\in L^2(M,\varphi)$ we have
	\[
		\mathscr{E}[\xi]\leq \liminf_{n\to\infty} \mathscr{E}[\xi_n].
	\]
In this case, to show $\xi\in \dom(\mathscr{E})$ (for the unextended form) it suffices to show $\mathscr{E}[\xi]<+\infty$.

Consider the following closed, convex set:
	\[
		C=\overline{\{\Delta_\varphi^{1/4}x \colon x\in M_{s.a.},\ x\leq 1\}}^{\|\cdot\|_\vphi}
	\]
For $\xi\in L^2_{s.a.}(M,\varphi)$, we let $\xi\wedge 1$ denote the projection of $\xi$ onto $C$.

A quadratic form ($\mathscr{E},\dom(\mathscr{E}))$ is \emph{Markovian} if for every $J_\varphi$-real $\xi\in \dom(\mathscr{E})$, $\xi\wedge 1\in\dom(\mathscr{E})$ with
	\[
		\mathscr{E}[\xi\wedge 1]\leq \mathscr{E}[\xi].
	\]
A \emph{Dirichlet form} is a closed Markovian form.

For any $n\in \N$, one has a canonical extension of $\mathscr{E}$ to $L^2(M_n(M),\frac{1}{n}\varphi\circ \Tr)$ in terms of the quadratic form $\mathscr{E}^{(n)}$ defined on the subspace $\{ \left[\xi_{i,j}\right]_{i,j=1}^n\colon \xi_{i,j}\in \dom(\mathscr{E}),\ 1\leq i,j\leq n\}$  by
	\[
		\mathscr{E}^{(n)}[ \left[\xi_{i,j}\right]_{i,j=1}^n ] = \sum_{i,j=1}^n \mathscr{E}[\xi_{i,j}].
	\]
A form is \emph{completely Markovian} (resp. \emph{completely Dirichlet}) if $\mathscr{E}^{(n)}$ is Markovian (resp. Dirichlet) for every $n\geq 1$.

%%%%%%%%%%%%%%%%%%%%%%%%%%%%%%%%%%%%%%%%%%%%%%%%%%%%%%%%%%%%%%%%%%%%%%%%%%%%%
%              			Regarding the Generators  	      		 	 %
%%%%%%%%%%%%%%%%%%%%%%%%%%%%%%%%%%%%%%%%%%%%%%%%%%%%%%%%%%%%%%%%%%%%%%%%%%%%%
\section{Eigenoperators}\label{Generators}

We will assume that $M$ is reasonably well-behaved under the action of the the modular automorphism group $\sigma^\vphi$; that is, $M$ is generated by operators for whom the action of $\sigma^\vphi$ is merely multiplication by a scalar. Such operators will be known as ``eigenoperators'' of $\sigma^\vphi$. We begin with some equivalent conditions for having such generators and the analytic implications of their existence.

%%%%%%%%%%%%%%%%%%%%%%%%%%%%%%%%%%%%%%%%%%%%%%%%%%%%%%%%%%%%%%%%%%%%%%%%%%%%%
\subsection{Regarding the generators}\label{Regarding}

We begin with a proposition that, given a finitely generated von Neumann algebra $M$ with faithful normal state $\varphi$ and minimal assumptions on the generators, will allow us to freely switch to generators with more convenient behavior under the modular operator $\Delta_\varphi$.

\begin{prop}\label{equivalent_forms_of_generators}
Let $M$ be a von Neumann algebra with a faithful normal state $\varphi$. Then the following are equivalent:
	\begin{enumerate}
		\item[(i)] $M$ is generated by $\{e_1,\ldots, e_n\}\subset M_{s.a.}$ such that $\Delta_\varphi e_j\in \text{span}\{e_1,\ldots, e_n\}$ for each $j=1,\ldots, n$.
		
		\item[(ii)] $M$ is generated by $\{x_1,\ldots, x_n\}\subset M_{s.a.}$ such that $\Delta_\varphi x_j= \sum_{k=1}^n [A]_{jk} x_k$  for each $j=1,\ldots, n$ where $A\in M_n(\C)$ is a positive definite matrix of the form $A=\text{diag}(A_1,\ldots, A_k,1,\ldots, 1)$ where for each $j=1,\ldots, k$
			\begin{align*}
				A_j=\frac{1}{2}\left(\begin{array}{cc} \lambda_j + \lambda_j^{-1} & -i(\lambda_j - \lambda_j^{-1})\\	
											i(\lambda_j-\lambda_j^{-1}) & \lambda_j + \lambda_j^{1} \end{array}\right)
			\end{align*}
		for some $\lambda_j \in (0,1)$. Furthermore, the covariance of these generators is given by
			\begin{align*}
				\varphi(x_kx_j) = \left[\frac{2}{1+A}\right]_{jk}.
			\end{align*}
		
		\item[(iii)] $M$ is generated by $\{c_1,c_1^*,\ldots, c_k,c_k^*,z_{2k+1},\ldots, z_n\}$ where $\Delta_\varphi c_j = \lambda_j c_j$ for some $\lambda_j\in (0,1)$, $j=1,\ldots, k$, and $z_j$ are self-adjoint elements satisfying $\Delta_\varphi z_j=z_j$, $j=2k+1,\ldots, n$.
	\end{enumerate}
Moreover, elements of $GL(n,\C)$ linearly relate the three sets of generators in the following sense. If $V=(v_1,\ldots, v_n)$ and $W=(w_1,\ldots,w_n)$ are $n$-tuples formed by two sets of the above generators, then there is a $Q\in GL(n,\C)$ such that $V=QW$.

Consequently the modular automorphism group $\{\sigma_t^\varphi\}_{t\in\R}$ of $\varphi$ can be extended to $\sigma_z^\varphi$, $z\in\C$, on
	\begin{align*}
		\C\<e_1,\ldots, e_n\>=\C\<x_1,\ldots, x_n\>=\C\<c_1,c_1^*,\ldots,c_k,c_k^*,z_{2k+1},\ldots,z_n\>.
	\end{align*}
In particular, for any $z\in \C$ we have
	\begin{align*}
		\begin{array}{ll}
		\sigma_{z}^\varphi(x_j)=\sum_{k=1}^n [A^{iz}]_{jk} x_k	&\forall j\in\{1,\ldots, n\},\\
		\sigma_z^\varphi(c_j) = \lambda_j^{iz} c_j	&\forall j\in\{1,\ldots, k\},\ \text{and}\\
		\sigma_z^\varphi(z_j) = z_j	&\forall j\in\{2k+1,\ldots,n\}.
		\end{array}
	\end{align*}
\end{prop}
\begin{proof}
We first note that (iii)$\Rightarrow$(i) is clear.

\begin{imp}[i)$\Rightarrow$(ii]
Since $\varphi$ is faithful, we assume without loss of generality that $\{e_1,\ldots, e_n\}\subset L^2(M,\varphi)$ is a linearly independent set. Furthermore, $\Re\<\cdot,\cdot\>_\varphi$ is a positive definite symmetric bilinear form and so by a Gram-Schmidt process we may assume $\<e_j,e_k\>_\varphi \in i\R $ when $j\neq k$ and $\|e_j\|_\varphi=1$ while still maintaining $e_j=e_j^*$ for each $j=1,\ldots,n$.

Now, the modular operator $\Delta_\varphi$ is a positive, non-singular operator on $L^2(M,\varphi)$. The condition on $\{e_1,\ldots,e_n\}$ implies that $\text{span}\{e_1,\ldots, e_n\}$ is a $\Delta_\varphi$-invariant subspace. Define $A\in M_n(\C)$ by
	\begin{align*}
		\Delta_\varphi e_j = \sum_{k=1}^n [A]_{jk} e_k\qquad \forall j=1,\ldots, n.
	\end{align*}
We first claim that $A=A^*$. Let $\Lambda\in M_n(\C)$ be the covariance matrix: $[\Lambda]_{jk}:=\<e_k,e_j\>_\varphi$ for all $j,k=1,\ldots, n$. Note that $\Lambda$ is positive definite since $\{e_1,\ldots, e_n\}$ is a linearly independent set. Also,
	\begin{align*}
		[\Lambda^T\Lambda]_{jk} = \sum_{l=1}^n [\Lambda]_{lj} [\Lambda]_{lk} &= \sum_{l=1}^n \<e_j, e_l\>_\varphi\<e_k, e_l\>_\varphi\\
			&= \sum_{l=1}^n \overline{\<e_l, e_j\>_\varphi}\overline{\<e_l, e_k\>_\varphi }\\
			&= \sum_{l=1}^n -\<e_l, e_j\>_\varphi\cdot -\<e_l, e_k\>_\varphi = [\Lambda\Lambda^T]_{jk},
	\end{align*}
which implies $\Lambda^T\Lambda^{-1}=\Lambda^{-1}\Lambda^T$. Now, recall $S_\varphi = J_\varphi \Delta_\varphi^{1/2}=\Delta_\varphi^{-1/2}J_\varphi$, where $J_\varphi$ is an anti-linear isometry. We have for each $j,k\in\{1,\ldots, n\}$
	\begin{align*}
		[A\Lambda]_{jk} &= \sum_{l=1}^n [A]_{jl} \<e_k, e_l\>_\varphi= \< e_k, \sum_{l=1}^n [A]_{jl} e_l\>_\varphi = \< e_k, \Delta_\varphi e_j\>_\varphi\\
			&= \< J_\varphi \Delta_\varphi e_j, J_\varphi e_k\>_\varphi= \<\Delta_\varphi^{-1/2}J_\varphi \Delta_\varphi^{1/2} e_j, J_\vphi e_k\>_\varphi= \<S_\varphi e_j, S_\varphi e_k\>_\varphi\\
			&=\<e_j, e_k\>_\varphi = [\Lambda^T]_{jk},
	\end{align*}
so that $A=\Lambda^T\Lambda^{-1}=\Lambda^{-1}\Lambda^T$. By a similar computation we have $\Lambda A^*=\Lambda^T$ or $A^*=\Lambda^{-1}\Lambda^T = A$. We also note that this computation also shows
	\begin{align*}
		A & = \Lambda^{-1/2} \Lambda^T \Lambda^{-1/2}\\
		A^T & = (\Lambda^{-1})^T \Lambda = (\Lambda^{-1} \Lambda^T)^{-1} = A^{-1},
	\end{align*}
hence $A$ is positive definite with inverse $A^{-1}=A^T$.

Thus for each $t\in\R$, $U_t:= A^{it}$ is a unitary matrix because $A>0$ and an orthogonal matrix because
	\begin{align*}
		U_t^T = (A^{it})^T = A^{-it} = U_t^{-1}.
	\end{align*}
Hence, the entries of $U_t$ are real and $t\mapsto U_t$ is an orthogonal representation of $\R$ on $\R^n$. It follows that up to conjugating by some orthogonal matrix with real entries, $\forall t\in \R$ we have $U_t=\text{diag}(R_1(t),\ldots, R_k(t),1,\ldots, 1)$ where for each $j=1,\ldots, k$
	\begin{align*}
		R_j(t)= \left(\begin{array}{cc} \cos(t\log{\lambda_j}	) & -\sin(t\log{\lambda_j})\\	
											\sin(t\log{\lambda_j}) & \cos(t\log{\lambda_j}) \end{array}\right)
	\end{align*}
for some $\lambda_j\in (0,1]$. Using the formula
	\begin{align*}
		i\log(A) t v= \lim_{t\to 0} \frac{1}{t}(U_t - 1)v\qquad \forall v\in \C^n,
	\end{align*}
we see that $A$ has the desired form up to conjugation by an orthogonal matrix with real entries. Redefine $A$ to be the desired conjugated form and let $x_1,\ldots, x_n\in M_{s.a.}$ be the image of $\{e_1,\ldots, e_n\}$ under this orthogonal change of basis. Then clearly $\{x_1,\ldots, x_n\}$ generate $M$ and we have $\Delta_\varphi x_j= \sum_{k=1}^n [A]_{jk}x_k$ for each $j=1,\ldots, n$.

Redefine $\Lambda$ to be the covariance matrix of $\{x_1,\ldots, x_n\}$. Note that $\{x_1,\ldots, x_n\}$ is orthonormal with respect to $\Re\<\cdot,\cdot\>_\varphi$ since these elements are obtained from $\{e_1,\ldots, e_n\}$ by an orthogonal change of basis. Hence
	\begin{align*}
		\left[\Lambda + \Lambda^T\right]_{jk} = 2 \Re\< x_k, x_j\>_\varphi = [2]_{jk},
	\end{align*}
so that $1+\Lambda^T\Lambda^{-1} = 2\Lambda^{-1}$. By the same computation as above we have that $\Lambda^T\Lambda^{-1}=A$. It follows that
	\begin{align*}
		\Lambda = \frac{2}{1+A}.
	\end{align*}
\end{imp}

\begin{imp}[ii)$\Rightarrow$(iii]
Define $Q\in M_n(\C)$ by $Q=\text{diag}(Q_1,\ldots,Q_k,1,\ldots,1)$ where for each $j=1,\ldots, n$
	\begin{align*}
		Q_j = \frac{1}{\sqrt{2}}\left(\begin{array}{cc} 1 & -i\\
								1 & i \end{array}\right).
	\end{align*}
Then it is easy to check that $Q A Q^* = \text{diag}(\lambda_1,\lambda_1^{-1},\ldots, \lambda_k,\lambda_k^{-1},1,\ldots, 1)$. Write $X=(x_1,\ldots, x_n)$ and define the new generators as the entries of the $n$-tuple $QX$. If we write
	\begin{align*}
		(c_1,b_1,\ldots, c_k,b_k,z_{2k+1},\ldots, z_n)=QX,
	\end{align*}
then because $x_1,\ldots,x_n$ are self-adjoint it is clear from the definition of $Q$ that $b_j=c_j^*$, $j=1,\ldots,k$, and $z_j=z_j^*$, $j=2k+1,\ldots, n$. Furthermore,
	\begin{align*}
		\Delta_{\varphi} c_j = \sum_{k=1}^n \Delta_\varphi [Q]_{jk} x_k = \sum_{k,l=1}^n [Q]_{jk} [A]_{kl} x_l = \sum_{l=1}^n \lambda_j [Q]_{jl} x_l = \lambda_j c_j,
	\end{align*}
for each $j=1,\ldots, k$. A similar computation yields $\Delta_\varphi z_j = z_j$ for each $j=2k+1,\ldots, n$.
\end{imp}

It is clear from their constructions that the various sets of generators are linearly related by invertible matrices.

Finally, the extension of the modular automorphism group is given by
	\begin{align*}
		\sigma_z^\varphi(x):= \Delta_\varphi^{iz} x \Delta_\varphi^{-iz}.
	\end{align*}
The action of $\Delta_\varphi$ on the vectors $x_j$, $c_j$, and $z_j$ then implies the claimed formulas.
\end{proof}

\begin{rem}
The relationship between the sets of generators in (ii) and (iii) is precisely the relationship between quasi-free semicircular random variables and generalized circular elements (\emph{cf.} \cite[Section 4]{Shl97} and \cite[Section 3]{Nel15b}).
\end{rem}

\begin{rem}
If $k=0$ in either condition (ii) or (iii), then all of the generators are fixed points of $\Delta_\varphi$ and hence $\varphi$ is a trace.
\end{rem}

\begin{rem}\label{no_worries_about_adjoints}
Suppose $y\in M\subset L^2(M,\vphi)$ is an eigenvector of $\Delta_\varphi$ with eigenvalue $\lambda>0$. Then $y^*$ is also an eigenvector with eigenvalue $\lambda^{-1}$ because $\Delta_\varphi S_\varphi = S_\varphi \Delta_{\varphi}^{-1}$. Therefore if $y=y^*$ then $\lambda=1$. Hence if $M$ is generated by a finite set $G=G^*$ of eigenvectors of $\Delta_\varphi$, then $G$ is of the form in condition (iii).
\end{rem}

\begin{defi}
If $y\in M\subset L^2(M,\varphi)$ is an eigenvector of $\Delta_\varphi$ with eigenvalue $\lambda>0$, we say that $y$ is an \emph{ eigenoperator of the modular automorphism group $\sigma^\varphi=\{\sigma_t^\varphi\}_{t\in\R}$ with eigenvalue $\lambda$}.
\end{defi}

Observe that every element of $M^\vphi$ is an eigenoperator with eigenvalue $1$. In particular, if $\vphi=\tau$ is a trace, then every element of $M$ is an eigenoperator.

\begin{prop}\label{prop:eigenop_polar_decomp}
Let $M$ be a von Neumann algebra with faithful normal state $\vphi$. Suppose $y\in M$ is a eigenoperator of $\sigma^\vphi$ with eigenvalue $\lambda$. If $y=v|y|$ is the polar decomposition, then $|y|\in M^\vphi$ and $v$ is an eigenoperator with eigenvalue $\lambda$.
\end{prop}
\begin{proof}
We note that
	\[
		\sigma_t^\vphi(y^*y) = \sigma_t^\vphi(y)^* \sigma_t^\vphi(y) = \lambda^{-it} y^* \lambda^{it} y = y^*y.
	\]
Thus $y^*y\in M^\vphi$, and consequently $|y|=\sqrt{y^*y}\in M^\vphi$. It follows that $y=\lambda^{-it}\sigma_t^\vphi(v)|y|$, and so by uniqueness of the polar decomposition we have $\sigma_t^\vphi(v)=\lambda^{it} v$.
\end{proof}

Henceforth $M$ will be a von Neumann algebra with a faithful normal state $\varphi$, which is generated by a finite set $G=G^*$ of eigenoperators of $\sigma^\varphi$. By Remark \ref{no_worries_about_adjoints} we may use Proposition \ref{equivalent_forms_of_generators} to call on generators $\{x_1,\ldots, x_n\}$ of the form in (ii). We will often switch between these two generating sets depending on convenience, but will always denote generators of the form in (ii) by $x_1,\ldots, x_n$ whereas elements of $G$ will generally be denoted by $y$. Denote
	\begin{align*}
		\P:=\C\<G\>=\C\<x_1,\ldots, x_n\>.
	\end{align*}
We also note that $M$ has a separable predual since it can be faithfully represented on $L^2(M,\varphi)$ and is finitely generated.

\begin{rem}
The formulas at the end of Proposition \ref{equivalent_forms_of_generators} imply $\P\subset M_\infty$ with $f_p(z)=\Delta_\varphi^{iz}p\Delta_{\varphi}^{-iz}=\sigma_z^\vphi(p)$ for $p\in \P$. Furthermore, using the density of $\P$ in $L^2(M,\varphi)$ and its invariance under $\Delta_\varphi^{iz}$ for every $z\in \C$, one can deduce from the proof of \cite[Lemma VI.2.3]{Tak03} that $f_x(z)=\Delta_\varphi^{iz}x\Delta_\varphi^{-iz}$ for every $x\in M_\infty$.
\end{rem}

%%%%%%%%%%%%%%%%%%%%%%%%%%%%%%%%%%%%%%%%%%%%%%%%%%%%%%%%%%%%%%%%

\subsection{Implications of eigenoperators as generators}\label{Implications}

Observe that since $\sigma_{-i}^\varphi$ is a homomorphism, any product of eigenoperators is again an eigenoperator. In particular, any monomial in $\C\<G\>$ is an eigenvector of $\Delta_\varphi$ and so the modular operator $\Delta_\varphi$ can be written
	\begin{align*}
		\Delta_\varphi = \sum_{\lambda \in \R_+^\times} \lambda \pi_\lambda,
	\end{align*}
for pairwise orthogonal projections $\{\pi_\lambda\}_{\lambda\in \R_+^\times}$ on $L^2(M,\varphi)$; that is, $\varphi$ is almost periodic (\emph{cf.} \cite{Con72}). Moreover, $\sum_{\lambda\in \R_+^\times} \pi_\lambda$ converges strongly to the identity. For each $\lambda\in \R_+^\times$, let $E_\lambda$ denote $\pi_\lambda L^2(M,\varphi)$, the eigenspace of $\Delta_\varphi$ corresponding to the eigenvalue $\lambda$.

For each $\lambda\in \R_+^\times$, define a map $\mc{E}_\lambda$ on $\mc{B}(L^2(M,\varphi))$ by
	\begin{align*}
		\mc{E}_\lambda(T) = \sum_{\mu\in \R_+^\times} \pi_{\lambda\mu} T\pi_\mu.
	\end{align*}
Using the orthogonality of the projections $\pi_\lambda$, it is easy to see that $\mc{E}_\lambda(T)$ is bounded with $\|\mc{E}_\lambda(T)\|\leq \|T\|$ and that $\mc{E}_\lambda\circ\mc{E}_\lambda=\mc{E}_\lambda$. Moreover, since each $\pi_\lambda \in (M^\varphi)'\cap \mc{B}(L^2(M,\varphi))$, $\mc{E}_\lambda$ is left and right $M^\varphi$-linear.

\begin{lem}
Let $\lambda\in \R_+^\times$. If $p\in M$ is an eigenoperator of $\sigma^\varphi$ with eigenvalue $\mu\in\R_+^\times$, then $\mc{E}_\lambda(p)=\delta_{\lambda=\mu} p$. Furthermore, $\mc{E}_\lambda(M)\subset M$.
\end{lem}
\begin{proof}
If $\xi\in E_\nu$, then clearly $p\xi \in E_{\mu\nu}$ and therefore $\pi_{\lambda\nu} p\xi =\delta_{\lambda=\mu} p\xi$, establishing the first claim. Now, given $x\in M$, we can use Kaplansky's density theorem to find $x_n\in \C\<G\>$ converging strongly to $x$ and satisfying $\|x_n\|\leq \|x\|$; in particular, $x_n$ converges $\sigma$-strongly to $x$. We claim that $\mc{E}_\lambda(x_n)$ converges strongly to $\mc{E}_\lambda(x)$. Indeed, we fix $\xi\in L^2(M,\varphi)$ and compute
	\begin{align*}
		\| \mc{E}_\lambda (x - x_n) \xi\|_\varphi^2 = \sum_{\mu\in\R_+^\times} \| \pi_{\lambda\mu}(x- x_n)\pi_{\mu} \xi\|_\varphi^2 \leq \sum_{\mu\in \R_+^\times} \| (x - x_n)\pi_\mu \xi\|_\varphi^2.
	\end{align*}
This tends to zero by the $\sigma$-strong convergence of $x_n$ to $x$ since
	\begin{align*}
		\sum_{\mu\in\R_+^\times} \|\pi_\mu \xi\|_\varphi^2 = \|\xi\|_\varphi^2<\infty.
	\end{align*}
Since each $x_n$ is a sum of eigenoperators of $\sigma^\varphi$, $\mc{E}_\lambda(x_n)\in \C\<G\>$ by the above discussion. Hence $\mc{E}_\lambda(x)\in M$ as the strong limit of polynomials.
\end{proof}

From this lemma we can deduce that $\mc{E}_\lambda(M)$ consists of all eigenoperators in $M$ with eigenvalue $\lambda$ and that $\varphi\circ\mc{E}_\lambda= \delta_{\lambda=1}\varphi$ on $M$. We think of $\mc{E}_\lambda$ as a ``conditional expectation'' onto $\mc{E}_\lambda(M)$, but recognize that $\mc{E}_\lambda(M)$ is not an algebra for $\lambda\neq 1$. However, $\mc{E}_1=\mc{E}_\varphi$, which is the actual conditional expectation onto the centralizer $M^\varphi$.

\begin{cor}\label{cor:bicentralizer}
$(M^\vphi)'\cap M = (M^\vphi)'\cap M^\vphi$.
\end{cor}
\begin{proof}
Let $x\in (M^\vphi)'\cap M$. Without loss of generality, we may assume $x$ is self-adjoint. Since
	\[
		\| x - \mc{E}_\vphi(x)\|_\vphi^2 = \sum_{\lambda\neq 1} \|\mc{E}_\lambda(x)\|_\vphi^2,
	\]
it suffices to show $\mc{E}_\lambda(x)=0$ for all $\lambda\neq 1$. Moreover, because $\mc{E}_\lambda(x)^*=\mc{E}_{\lambda^{-1}}(x)$, we need only consider $\lambda>1$. Fix such a $\lambda$, and note that $\mc{E}_\lambda(x)\in (M^\vphi)'\cap M$ since $\mc{E}_\lambda$ is left and right $M^\vphi$-linear. Let $\mc{E}_\lambda(x)=v|\mc{E}_\lambda(x)|$ be the polar decomposition, so that $v$ is an eigenoperator with eigenvalue $\lambda$ by Proposition \ref{prop:eigenop_polar_decomp}. Observe that for $a\in M^\vphi$,
	\[
		v^* a v|\mc{E}_\lambda(x)| = v^* a \mc{E}_\lambda(x) = v^* \mc{E}_\lambda(x) a  = |\mc{E}_\lambda(x)| a.
	\]
Thus, if we define $\theta\colon M^\vphi\to M^\vphi$ by $\theta(a):=v^*av$, then
	\[
		\vphi(|\mc{E}_\lambda(x)| a ) = \vphi( \theta(a) |\mc{E}_\lambda(x)|) = \vphi( |\mc{E}_\lambda(x)| \theta(a)) = \cdots =\vphi(  |\mc{E}_\lambda(x)| \theta^k(a))
	\]
for any $k\in \N$. Consequently,
	\begin{align*}
		\| \mc{E}_\lambda(z) \|_\vphi^2 &= \vphi( |\mc{E}_\lambda(x)| |\mc{E}_\lambda(x)|)= \vphi\left( |\mc{E}_\lambda(x)| \theta^k(|\mc{E}_\lambda(x)|)\right)\\
			&= \vphi\left( \sigma_i^\vphi(v^k) |\mc{E}_\lambda(x)| (v^*)^k |\mc{E}_\lambda(x)|\right)= \lambda^{-k} \vphi( v^k |\mc{E}_\lambda(x)| (v^*)^k |\mc{E}_\lambda(x)|)\\
			&= \lambda^{-k} \< v^k |\mc{E}_\lambda(x)| (v^*)^k, |\mc{E}_\lambda(x)|\>_\vphi\leq \lambda^{-k} \|v^k |\mc{E}_\lambda(x)| (v^*)^k\| \| \mc{E}_\lambda(x) \|_\vphi\\
			&\leq \lambda^{-k} \| \mc{E}_\lambda(x) \| \|\mc{E}_\lambda(x)\|_\vphi.
	\end{align*}
Letting $k\to \infty$ we see $\|\mc{E}_\lambda(z)\|_\vphi =0$. Thus $x=\mc{E}_\vphi(x)\in M^\vphi$.
\end{proof}

For a subset $I\subset \R_+^\times$, let
	\begin{align*}
		\pi_I:=\sum_{\lambda\in I} \pi_\lambda
	\end{align*}
and
	\begin{align*}
		E_I=\overline{\text{span}}\{E_\lambda\colon \lambda\in I\} = \pi_I L^2(M,\varphi).
	\end{align*}
We define
	\[
		M_0:=\{x\in M\colon x \in E_{[1/n,n]}\text{ for some }n\in\N\},
	\]
which is easily seen to be a $*$-algebra containing $M^\varphi$ and $\P$, and hence is strongly dense in $M$. Clearly $M_0\subset M_\infty$.

We will use the following lemma frequently to simplify the analysis of the modular operator.

\begin{lem}\label{P_is_core}
Let $\mc{D}\subset M$ be an algebra generated by eigenoperators, and assume $\mc{D}$ is dense in $L^2(M,\vphi)$. Then for every $z\in \C$, $\mc{D}$ is a core of $\Delta_\varphi^z$.
\end{lem}
\begin{proof}
Since $\Delta_\varphi^z$ is closed, its graph $\mc{G}(\Delta_\varphi^z)$ has the following orthogonal decomposition:
	\[
		\mc{G}(\Delta_\varphi^z) = \overline{\mc{G}(\Delta_\varphi^z\mid_\mc{D})}\oplus \mc{G}(\Delta_\varphi^z\mid_\mc{D})^\perp.
	\]
Thus, it suffices to show $\mc{G}(\Delta_\varphi^z\mid_\mc{D})^\perp=0$. If $\xi\in \dom(\Delta_\varphi^z)$ satisfies
	\[
		\<\xi,x\>_\varphi + \<\Delta_\varphi^z \xi, \Delta_\varphi^z x\>_\varphi=0,\qquad \forall x\in \mc{D},
	\]
then $\<\xi, \left(1+\Delta_\varphi^{2\Re(z)}\right) x\>_\varphi =0$ for all $x\in \mc{D}$. Noting that $\left(1+\Delta_\varphi^{2\Re(z)}\right)\mc{D}=\mc{D}$ (since an eigenoperator with eigenvalue $\lambda$ is simply scaled by $1+\lambda^{2\Re(z)}$), the density of $\mc{D}$ in $L^2(M,\varphi)$ implies $\xi=0$.
\end{proof}

\begin{rem}\label{mod_op_at_one_quarter_convergence}
An easy consequence of this lemma when $z=\frac14$ is that $\overline{\mc{D}}^{\|\cdot\|_\varphi^\sharp}\subset \dom(\Delta_\varphi^{1/4})$. In fact,
	\[
		\|\Delta_\vphi^{1/4}  x\|_\vphi^2 =\<  x, \Delta_\vphi^{1/2}  x\>_\vphi \leq \| x\|_\vphi \| x\|
	\]
implies that $\Delta_\vphi^{1/4}  x_n \to \Delta_\vphi^{1/4}  x$ so long as $ x_n\to  x\in L^2(M,\vphi)$ and $\{x_n\}_{n\in\N}$ is uniformly bounded. 
\end{rem}

From this remark we can somewhat simplify our later analysis of Dirichlet forms. Recall that for $\xi\in L^2_{s.a.}(M,\vphi)$, $\xi\wedge 1$ is the projection of $\xi$ onto the closed, convex set
	\[
		C=\overline{\{\Delta_\varphi^{1/4}x \colon x\in M_{s.a.},\ x\leq 1\}}^{\|\cdot\|_\vphi}.
	\]

\begin{lem}\label{convex_proj_for_M}
For $x\in M^\vphi$ be self-adjoint, $[\Delta_\vphi^{1/4} x]\wedge 1= \Delta_\vphi^{1/4} f(x)$, where $f(t)=\min\{t,1\}$ for $t\in\R$.
\end{lem}
\begin{proof}
Since $x\in M^\vphi$, we have $\Delta_\vphi^{1/4}x=x$ and $\Delta_\vphi^{1/4} f(x)=f(x)$ in $L^2_{s.a.}(M,\vphi)$. Note that $f(x)\in C$. By \cite[Theorem 5.2]{Bre11}, the projection $x\wedge 1$ is characterized by
	\begin{align}\label{proj_char}
		\<x - x\wedge 1, x\wedge 1-\xi\>\geq 0\qquad \forall \xi\in C.
	\end{align}
Note that by the functional calculus we have $x-f(x)\in L^2_+(M,\vphi)$ and that
	\[
		\<x-f(x), f(x)\>_{\vphi}= \<x-f(x), 1\>_\vphi.
	\]
So for $a\in M_{s.a.}$ with $a\leq 1$, we have $\Delta_\vphi^{1/4}(1-a)\in L_+^2(M,\vphi)$ and hence
	\[
		\<x-f(x), f(x) - \Delta_\vphi^{1/4}a\>_\vphi = \<x-f(x), 1- a\>_\vphi \geq 0.
	\]
Since such elements $a$ are dense in $C$, this proves that $f(x)$ satisfies (\ref{proj_char}).
\end{proof}

%%%%%%%%%%%%%%%%%%%%%%%%%%%%%%%%%%%%%%%%%%%%%%%%%%%%%%%%%%%%%%%%%%%%%%%%%%%%%
%              		   Non-tracial differential calculus   	       	 %
%%%%%%%%%%%%%%%%%%%%%%%%%%%%%%%%%%%%%%%%%%%%%%%%%%%%%%%%%%%%%%%%%%%%%%%%%%%%%

\section{Non-tracial Differential Calculus}\label{diff_calc}

In the tracial case, derivations on a von Neumann algebra have proven to be powerful tools in both Popa's deformation/rigidity theory and free probability. Provided the modular automorphism group interacts with the derivation in a nice way, one should expect similar results in the non-tracial case. Here we study, in particular, the notion of ``$\mu$-modularity,'' and give several examples of derivations exhibiting this behavior. We will also examine derivations previously considered in \cite{Nel15}, which will be used to prove an $L^2$-homology type result in Subsection \ref{technical_estimate_section}.

%%%%%%%%%%%%%%%%%%%%%%%%%%%%%%%%%%%%%%%%%%%%%%%%%%%%%%%%%%%%%%%%%%%%%%%%%%%%%

\subsection{Non-tracial derivatives, $\mu$-modularity, conjugate variables, and free Fisher information}\label{Definitions}

All of the derivations we consider will be valued in $M\otimes M^{op}$ or $L^2(M\bar\otimes M^{op}, \vphi\otimes\vphi^{op})$, so we begin with some conventions on these spaces. First, we shall usually denote the latter space simply by $L^2(M\bar\otimes M^{op})$, as the only GNS representation of $M\bar\otimes M^{op}$ we will consider is the one with respect to $\varphi\otimes\vphi^{op}$. For elements $x^\circ\in M^{op}$, we will also usually suppress the ``$\circ$'' notation, and use the notation $\#$ for the natural multiplication on this space:
	\begin{align*}
		(a\otimes b)\#(c\otimes d)= (ac)\otimes (db).
	\end{align*}
On $M\otimes M^{op}$ we consider three involutions:
	\begin{align*}
		(a\otimes b)^* &:=a^*\otimes b^*\\
		(a\otimes b)^\dagger &:= b^*\otimes a^*\\
		(a\otimes b)^\diamond &:= b\otimes a.
	\end{align*}
The first involution is used in the definition of $\<\ \cdot\,,\cdot\ \>_{\text{HS}}$:
	\begin{align*}
		\<a\otimes b,c\otimes d\>_{\text{HS}}:=\varphi\otimes \varphi^{op}( (a\otimes b)^*\# c\otimes d).
	\end{align*}
The second involution corresponds to the adjoint on $\text{HS}(L^2(M,\varphi))$ when we identify it with $L^2(M\bar{\otimes }M^{op})$ via the map
	\begin{align*}
		\Psi(a\otimes b):= aP_1 b,
	\end{align*}
where $P_1\in \B(L^2(M,\varphi))$ is the projection onto the cyclic vector $1\in L^2(M,\vphi)$. The third involution arises as the composition of the first two.

Note that
	\begin{align*}
		x\cdot (a\otimes b) &:= (xa)\otimes b\\
		(a\otimes b)\cdot y &:= a\otimes (by)
	\end{align*}
extend to operator-norm-bounded left and right actions of $M$ on $L^2(M\bar\otimes M^{op})$.

\begin{defi}
Let $B\subset M$ be a $*$-subalgebra. A \emph{derivation} is a $\C$-linear map $\delta\colon B\to M\otimes M^{op}$ satisfying the Leibniz rule:
	\[
		\delta(ab)=\delta(a)\cdot b + a\cdot \delta(b).
	\]
We call $B$ the \emph{domain} of $\delta$ and write $\dom(\delta):=B$. The \emph{conjugate derivation to $\delta$}, denoted by $\hat\delta$, is a derivation with $\dom(\hat\delta)=\dom(\delta)$ defined by $\hat\delta(x)=\delta(x^*)^\dagger$ for $x\in \dom(\delta)$.
\end{defi}

Given the required codomain, this definition is more restrictive than the usual definition of a derivation, but is sufficiently general for our present work. Note that if $\C\subset\dom(\delta)$, the Leibniz rule implies $\delta(z)=0$ for all $z\in \C$.

Most of the derivations we consider will, by virtue of the regularity conditions imposed on the generators of $M$, interact nicely with the modular automorphism group. We give this property the following name.

\begin{defi}
Suppose $B\subset M_\infty$, and that $\delta\colon B\to M_\infty\otimes M_\infty^{op}$ is a derivation. For $\mu>0$, we say that $\delta$ is \emph{$\mu$-modular} if it satisfies
	\begin{align}\label{deriv_mod_op_scaling}
		\delta\circ \sigma_z^\vphi(x) = \mu^{iz} (\sigma_z^\vphi\otimes\sigma_z^\vphi)\circ\delta(x)\qquad \forall z\in \C,\ x\in B.
	\end{align}
\end{defi}

\begin{rem}
Note that if $\delta$ is $\mu$-modular, then its conjugate derivation $\hat\delta$ is $\mu^{-1}$-modular.
\end{rem}

We will assume initially that $G$ is an algebraically free set: no non-trivial polynomial is zero. However, we will see that this assumption is in fact redundant in the context of our main results (\emph{cf}. Remark \ref{alg_free}). Note that the linear relation between $G$ and $\{x_1,\ldots, x_n\}$ implies that $x_1,\ldots, x_n$ are algebraically free as well. For each $y\in G$, we define
	\[
		\delta_y(y')=\delta_{y=y'}1\otimes 1,
	\]
then extend $\delta_y$ to $\P$ by the Leibniz rule and linearity. Then $\delta_y\colon \P\to\P\otimes\P^{op}$ is a derivation. If $\lambda_y>0$ is the eigenvalue of $y$, then using the formulas at the end of Proposition \ref{equivalent_forms_of_generators} it is easy to see that $\delta_y$ is $\lambda_y$-modular. Also, the conjugate derivation $\hat\delta_y$ is simply $\delta_{y^*}$.

When $\vphi$ is a trace (and consequently $y=y^*$ for each $y\in G$), the derivations $\delta_y$ are examples of Voiculescu's free difference quotients (\emph{cf.} \cite{VoiII}). When $\vphi$ is not a trace, we still consider them to be free difference quotients, but under the following slightly more general definition (the main difference being that the defining variable $a$ need not be self-adjoint).

\begin{defi}
Given a $*$-subalgebra $B\subset M$, let $a\in M$ be algebraically free from $B$ and denote by $B[a]$ the $*$-algebra generated by $B$ and $a$. If $a$ is not self-adjoint further assume $a^*$ is algebraically free from $a$. The \emph{free difference quotient with respect to $a$} is a derivation
	\begin{align*}
		\delta_a\colon B[a]\to B[a]\otimes B[a]^{op}
	\end{align*}
defined by
	\begin{align*}
		\delta_a(a)&=1\otimes 1,\\
		\delta_a(b)&=0 \ \ \ \forall b\in B,
	\end{align*}
and the Leibniz rule. If $a$ is not self-adjoint, we also set $\delta_a(a^*)=0$.
\end{defi}

When $a=a^*$ we have $\hat\delta_a=\delta_a$, and when $a$ is not self-adjoint we have $\hat\delta_a=\delta_{a^*}$, provided the latter derivation is well-defined. If $B\subset M_\infty$, then $\delta_a$ is $\mu$-modular if $a$ is an eigenoperator with eigenvalue $\mu$.

In the spirit of \cite[Definitions 3.1 and 6.1]{VoiV}, we make the following definitions.

\begin{defi}\label{conjugate_variables_definition}
For a derivation $\delta$, the \emph{conjugate variable to $\delta$ with respect to $\vphi$} is defined to be a vector $\xi\in L^2(W^*(\dom(\delta)),\vphi)$ such that
	\begin{align}\label{conjugate_variable_defining_property}
		\<\xi, x\>_\vphi = \<1\otimes 1, \delta(x)\>_{\vphi\otimes\vphi^{op}}\qquad \forall x\in \dom(\delta),
	\end{align}
provided such a vector exists. Note that the conjugate variable is uniquely determined by (\ref{conjugate_variable_defining_property}). Also, observe that $\delta^*(1\otimes 1)$ is the conjugate variable to $\delta$ provided $1\otimes 1\in\dom(\delta^*)$ when
	\[
		\delta\colon L^2(W^*(\dom(\delta)),\vphi)\to L^2(M\bar\otimes M^{op})
	\]
is viewed as a densely defined operator. For a free difference quotient $\delta_a$, $a\in M$, with $\dom(\delta_a)=B[a]$, we denote the conjugate variable by $J_\vphi(a\colon B)$.
\end{defi}

\begin{defi}
For $a\in M$ algebraically free from a subalgebra $B\subset M$, the \emph{free Fisher information for $a$ over $B$ with respect to $\vphi$} is defined as
	\begin{align*}
		\Phi_\vphi^*(a\colon B):= \|J_\vphi(a\colon B)\|_\vphi^2
	\end{align*}
when $J_\vphi(a\colon B)$ exists, and as $+\infty$ otherwise. For multiple elements $a_1,\ldots, a_n\in M$ that are algebraically free over $B$, let $B[a_k\colon k\neq j]$ denote the $*$-algebra generated by $B$ and the elements $\{a_k\}_{k\neq j}$. Then the \emph{free Fisher information for $a_1,\ldots, a_n$ over $B$} is defined as
	\begin{align*}
		\Phi_\vphi^*(a_1,\ldots, a_n\colon B):=\sum_{j=1}^n \Phi_\vphi^*(a_j\colon B[a_k\colon k\neq j])
	\end{align*}
when $J_\vphi(a_j\colon B[a_k\colon k\neq j])$ exists for each $j=1,\ldots, n$, and as $+\infty$ otherwise. When $a_1,\ldots, a_n$ are just algebraically free (over $\C$), then the \emph{free Fisher information for $a_1,\ldots, a_n$} is defined
	\begin{align*}
		\Phi_\vphi^*(a_1,\ldots, a_n):=\Phi_\vphi^*(a_1,\ldots, a_n\colon \C).
	\end{align*}
\end{defi}

\begin{rem}
Notions of conjugate variables and free Fisher information for non-tracial von Neumann algebras were previously considered by Shlyakhtenko in \cite{Shl03}. We will see in Section \ref{technical_estimate_section} that our current definitions are strongly related (\emph{cf.} Remark \ref{conjugate_variables_are_same_as_dima}).
\end{rem}

\begin{rem}\label{alg_free}
When $\vphi=\tau$ is a trace, \cite[Theorem 2.5]{MSW15} implies that finite free Fisher information for $a_1,\ldots, a_n$ (or equivalently the existence of the conjugate variables to their free difference quotients) guarantees that $a_1,\ldots, a_n$ are algebraically free. In fact, the proof of this theorem never invokes the trace condition on $\tau$ and therefore holds when $\vphi$ is merely a faithful normal state, as in our context. Thus we will not require in our results that $G$ be algebraically free, as this will be a consequence of $\Phi_\vphi^*(G)<\infty$.
\end{rem}

For each $y\in G$, we let $\xi_y$ denote the conjugate variable to $\delta_y$ with respect to $\vphi$, provided it exists. One immediate consequence of the existence of conjugate variables is that the derivation is closable, as we shall see in the following lemma. In fact, when we assume that the free Fisher information is finite, it is usually to make use of this property.

\begin{lem}\label{adjoint_formula}
Let $\delta\colon \dom(\delta)\to M_\infty\otimes M_\infty^{op}$ be a derivation with $\dom(\delta)\subset M_\infty$ dense in $L^2(M,\vphi)$. If $\eta\in \dom(\delta^*)$ when viewed as a densely defined map $\delta\colon L^2(M,\varphi)\to L^2(M\bar{\otimes} M^{op})$, then for $a,b\in \dom(\delta)$ we have $a\cdot \eta, \eta\cdot b\in \dom(\delta^*)$ with
	\begin{align*}
		\delta^*(a\cdot \eta)&=a\cdot \delta^*(\eta) - (\varphi\otimes\sigma_{-i}^\varphi)\left[\eta\#(\sigma_{-i}^\varphi\otimes \sigma_{i}^\varphi) \left(\hat\delta(a)^\diamond\right)\right]\\
		\delta^*(\eta\cdot b)&=\delta^*(\eta)\cdot \sigma_{-i}^\varphi(b) - (1\otimes\varphi)\left[\eta\#(\sigma_{-i}^\varphi\otimes \sigma_{i}^\varphi)\left(\hat\delta(b)^\diamond\right)\right].
	\end{align*}
In particular, if the conjugate variable $\xi$ to $\delta$ exists then $\dom(\delta)\otimes\dom(\delta)^{op}\subset \dom(\delta^*)$ with
	\begin{align*}
		\delta^*(a\otimes b)=a\cdot \xi\cdot \sigma_{-i}^\varphi(b) - m(1\otimes \varphi\otimes \sigma_{-i}^\varphi)( 1\otimes \hat\delta + \hat\delta\otimes 1)( a\otimes b)
	\end{align*}
(where $m(a\otimes b)=ab$), and $\delta$ is closable.
\end{lem}
\begin{proof}
Let $\eta\in \dom(\delta^*)$, $b\in \dom(\delta)$, and $x\in \dom(\delta)$. We compute
	\begin{align*}
		\<\eta\cdot b, \delta(x)\>_{\text{HS}}&= \<\eta, (1\otimes b^*)\# \delta(x)\>_{\text{HS}}\\
			&= \<\eta, \delta(xb^*) - x\cdot \delta(b^*)\>_{\text{HS}}\\
			&= \<\delta^*(\eta), xb^*\>_\varphi - \< \eta\# (\sigma_{-i}^\varphi\otimes\sigma_i^\varphi)(\delta(b^*)^*), x\otimes 1\>_{\text{HS}}\\
			&= \< \delta^*(\eta)\cdot \sigma_{-i}^\vphi(b) - (1\otimes\varphi)\left[\eta \# (\sigma_{-i}^\varphi\otimes \sigma_i^\varphi)(\hat\delta(b)^\diamond)\right],x\>_\varphi.
	\end{align*}
A similar computation yields the formula for $\delta^*(a\cdot \eta)$.

Now, if the conjugate variable $\xi$ to $\delta$ exists, recall that $\xi= \delta^*(1\otimes 1)$. So applying the previous two formulas we have for $a,b\in \dom(\delta)$:
	\begin{align*}
		\delta^*(a\otimes b) &= a\delta^*(1\otimes b) - (\sigma_{-i}^\varphi\otimes\varphi)\left[(\sigma_{i}^\varphi\otimes \sigma_{-i}^\varphi)\circ\hat\delta(a)\#b\otimes 1\right]\\
			&= a\left[ \xi\cdot \sigma_{-i}^\varphi(b) - (\varphi\otimes 1)\left[(\sigma_{i}^\varphi\otimes \sigma_{-i}^\varphi)\circ\hat\delta(b)\right]\right] - (1\otimes \varphi)\left[ \hat\delta(a)\# \sigma_{-i}^\varphi(b)\otimes 1\right]\\
			&= a\cdot \xi\cdot \sigma_{-i}^\varphi(b) - (\varphi\otimes \sigma_{-i}^\varphi)\left[a\otimes 1\# \hat\delta(b)\right] - (1\otimes \varphi)\left[ \hat\delta(a)\# \sigma_{-i}^\varphi(b)\otimes 1\right]\\
			&= a\cdot \xi\cdot \sigma_{-i}^\varphi(b) - m(1\otimes \varphi\otimes \sigma_{-i}^\varphi)( 1\otimes \hat\delta + \hat\delta\otimes 1)( a\otimes b).
	\end{align*}
Thus $\delta$ is closable since $\dom(\delta^*)$ contains the dense set $\dom(\delta)\otimes\dom(\delta)^{op}$.
\end{proof}

We conclude this subsection by noting that $\mu$-modularity forces the conjugate variable to a derivation to be well-behaved under the modular operator. In particular, this will apply to $\xi_y$, $y\in G$, when they exist.

\begin{lem}\label{conj_var_entire}
Let $\delta$ be a $\mu$-modular derivation for some $\mu>0$, with $\dom(\delta)$ a $*$-algebra generated by eigenoperators that is dense in $L^2(M,\vphi)$. If the conjugate variable $\xi$ to $\delta$ exists, then $\xi\in \dom(\Delta_\vphi^z)$ for all $z\in \C$ with
	\begin{align}\label{modular_operator_of_conjugate_variables_to_free_difference_quotient}
		\Delta_\vphi^z \xi = \mu^z \xi.
	\end{align}
Furthermore, $\xi\in \dom(S_\vphi)$ with
	\begin{align}\label{adjoint_of_conjugate_variables_to_free_difference_quotient}
		S_\vphi \xi = \mu \hat\xi,
	\end{align}
where $\hat\xi$ is the conjugate variable to $\hat\delta$, which exists if and only if $\xi$ does. In particular,
	\begin{align}\label{F_of_conjugate_variable_to_free_difference_quotient}
		\Delta_\vphi S_\vphi \xi = \hat\xi.
	\end{align}
\end{lem}
\begin{proof}
For $x\in \dom(\delta)$ we compute
	\begin{align*}
		\<\xi, \Delta_\varphi^{\bar{z}}  x \>_\varphi &=\< 1\otimes 1, \delta\circ\sigma_{\bar{-iz}}^\vphi( x)\>_{\text{HS}}\\
			& = \< 1\otimes 1, \mu^{\bar{z}} (\sigma_{\bar{-iz}}^\vphi\otimes\sigma_{\bar{-iz}}^\vphi)\circ \delta(x)\>_{\text{HS}}\\
			& = \< \mu^z 1\otimes 1, \delta(x)\>_{\text{HS}}\\
			& = \< \mu^z \xi,  x\>_\varphi.
	\end{align*}
This computation suffices since $\dom(\delta)$ is a core of $\Delta_\varphi^{\bar{z}}$ by Lemma \ref{P_is_core}.
	
Since $\dom(S_\varphi)=\dom(\Delta_\varphi^{1/2})$, the previous argument implies $\xi\in \dom(S_\varphi)$. So we compute for $x\in\dom(\delta)$:
	\begin{align*}
		\<S_\varphi \xi,  x\>_\varphi &= \<  J_\varphi \Delta^{1/2}_\varphi  x, \Delta_\varphi\xi\>_\varphi\\
						&= \<x^*,\mu \xi\>_\varphi\\
						&=\mu \< \delta(x^*),1\otimes 1\>_{\text{HS}}\\
						&=\mu \varphi\otimes\varphi^{op}(\hat\delta(x)^\diamond)\\
						&=\mu \varphi\otimes\varphi^{op}(\hat\delta(x))\\
						&=\<\mu 1\otimes 1, \hat\delta( x)\>_{\text{HS}},
	\end{align*}
which shows $\hat\xi$ exists and equals $\mu^{-1}S_\varphi \xi$. Finally, (\ref{F_of_conjugate_variable_to_free_difference_quotient}) follows from combining (\ref{modular_operator_of_conjugate_variables_to_free_difference_quotient}) for $z=1$ and (\ref{adjoint_of_conjugate_variables_to_free_difference_quotient}).\end{proof}

%%%%%%%%%%%%%%%%%%%%%%%%%%%%%%%%%%%%%%%%%%%%%%%%%%%%%%%%%%%%%%%%%%%%%%%%%%%%%

\subsection{Quasi-free difference quotients}\label{Qausi-free}

Let $A$ be the matrix from Proposition \ref{equivalent_forms_of_generators}. For each $j=1,\ldots, n$ we define
	\[
		\partial_j(x_k)=\left[\frac{2}{1+A}\right]_{kj} 1\otimes 1,
	\] 
and then extend $\partial_j$ to $\P$ by the Leibniz rule and linearity. Then $\partial_j$ is a derivation with conjugate derivation $\hat\partial_j$ determined by
	\[
		\hat\partial_j(x_k)=\left[\frac{2}{1+A}\right]_{jk} 1\otimes 1.
	\]
Furthermore, it follows from $\left(\frac{2}{1+A}\right)^T = \frac{2}{1+A^{-1}}$ that for $p\in\P$
	\begin{align}\label{modularity_of_quasi-free_diff_quotients}
		(\sigma_i^\vphi\otimes\sigma_i^{\vphi})\circ \partial_j\circ \sigma_{-i}^\vphi( p) = \hat{\partial}_j(p).
	\end{align}
If we let $\delta_j$ denote the free difference quotient with respect to $x_j$, $j=1,\ldots, n$, then the $\{\partial_j\}_{j=1}^n$ and $\{\delta_j\}_{j=1}^n$ are linearly related as follows:
	\[
		\partial_j = \sum_{k=1}^n \left[\frac{2}{1+A}\right]_{kj} \delta_k \qquad{\text{and}}\qquad \hat\partial_j = \sum_{k=1}^n \left[\frac{2}{1+A}\right]_{jk} \delta_k.
	\]
Such derivations have been previously considered in \cite{Nel15}, and we formally define them here.

\begin{defi}
Suppose $a_1,\ldots, a_n\in M_\infty$ are self-adjoint and that there is an $n\times n$ matrix $A>0$ which determines the covariance and action of the modular operator:
	\begin{align*}
		\vphi(a_ka_j) &= \left[\frac{2}{1+A}\right]_{jk}\\
		\sigma_{-i}^\vphi(a_j) &= \sum_{k=1}^n [A]_{jk} a_k.
	\end{align*}
If $a_1,\ldots, a_n$ are algebraically free then the \emph{quasi-free difference quotients} are defined on $\C\<a_1,\ldots, a_n\>$ as
	\begin{align*}
		\partial_{a_j} := \sum_{k=1}^n \left[\frac{2}{1+A}\right]_{kj} \delta_{a_k}.
	\end{align*}
The conjugate variables to $\partial_{a_j}$ with respect to $\vphi$ are defined as in Definition \ref{conjugate_variables_definition} and denoted by $J_\vphi^A(a_j\colon \C[a_k\colon k\neq j])$, provided they exist.
\end{defi}

The conjugate variables to $\partial_1,\ldots, \partial_n$ will be denoted $\xi_1,\ldots, \xi_n$, respectively. All quasi-free difference quotients satisfy the corresponding version of (\ref{modularity_of_quasi-free_diff_quotients}). One consequence of this is the following lemma, analogous to part of Lemma \ref{conj_var_entire}.

\begin{lem}\label{conj_var_to_quasi-free_diff_quot_is_self-adjoint}
If $\xi\in L^2(M,\varphi)$ is the conjugate variable to a quasi-free difference quotient $\partial$, $\dom(\partial)=\P$, then $\xi\in \dom(S_\varphi)$ with $S_\varphi \xi= \xi$.
\end{lem}	
\begin{proof}
Recall that $S_\varphi =J_\varphi \Delta_\varphi^{1/2}= \Delta_\varphi^{-1/2} J_\varphi$. We compute for $p\in \P$
	\begin{align*}
		\<\Delta_\varphi^{1/2}J_\varphi  p, \xi\>_\varphi &= \< \Delta_\varphi J_\varphi \Delta_\varphi^{1/2}  p, \xi\>_\varphi\\
			&= \<\sigma_{-i}^\vphi(p^*), \xi\>_\varphi\\
			&= \<\partial\circ\sigma_{-i}^\vphi(p^*), 1\otimes 1\>_{\text{HS}}\\
			&= \< (\sigma_{-i}^\vphi\otimes\sigma_{-i}^\vphi) \circ \hat{\partial}(p^*),1\otimes 1\>_{\text{HS}}\\
			&= \< \hat{\partial}(p^*),1\otimes 1\>_{\text{HS}}\\
			&= \varphi\otimes\varphi^{op}( \partial(p)^\diamond)\\
			&= \varphi\otimes\varphi^{op}( \partial(p))= \<\xi,  p\>_\varphi.
	\end{align*}
Now, the adjoint of $S_\varphi$, $J_\varphi \Delta_\varphi^{-1/2}$, has domain $\dom(\Delta_\varphi^{-1/2})$. Recall from Lemma \ref{P_is_core} that $\P$ is a core of $\Delta_\varphi^{-1/2}$ (and hence of $J_\varphi \Delta_\varphi^{-1/2}$). Therefore the above computation shows $\xi\in \dom(S_\varphi)$ with the claimed formula.
\end{proof}

\begin{rem}
Recall from Proposition \ref{equivalent_forms_of_generators} that there exists $Q\in GL(n,\C)$ such that
	\begin{align*}
		(c_1,c_1^*,\ldots, c_k,c_k^*,z_{2k+1},\ldots, z_n) = Q\cdot (x_1,\ldots, x_n).
	\end{align*}
It follows that the free difference quotients are related by $(Q^{-1})^T$:
	\begin{align*}
		(\delta_{c_1},\delta_{c_1^*},\ldots, \delta_{c_k},\delta_{c_k^*},\delta_{z_{2k+1}},\ldots, \delta_{z_n}) = (Q^{-1})^T\cdot (\delta_{x_1},\ldots, \delta_{x_n}).
	\end{align*}
Moreover, the conjugate variables $\{J_\varphi(x_j\colon \C[x_k\colon k\neq j])\}_{j=1}^n$ exist if and only if the conjugate variables $\{\xi_y\}_{y\in G}$ exist, and are related by $(Q^{-1})^*$:
	\begin{align*}
		(\xi_{c_1},\xi_{c_1^*},\ldots, \xi_{c_k},\xi_{c_k^*},\xi_{z_{2k+1}},\ldots, \xi_{z_n}) = (Q^{-1})^*\cdot \left(J_\varphi(x_1\colon \C[x_k\colon k\neq 1]),\ldots, J_\varphi(x_n\colon \C[x_k\colon k\neq n])\right).
	\end{align*}
Consequently $\Phi_\varphi^*(x_1,\ldots, x_n)$ is finite if and only if $\Phi_\varphi^*(G)$ is finite. Similarly, the linear relation between the quasi-free difference quotients $\partial_j$ and the free difference quotients $\delta_j$ implies $\{\xi_j\}_{j=1}^n$ exist if and only if $\{J_\varphi(x_j\colon \C[x_k\colon k\neq j])\}_{j=1}^n$ exist, and are related by
	\begin{align}\label{free_to_quasi-free_conjugate_variables}
		\xi_j = \sum_{k =1}^n \left[\frac{2}{1+A}\right]_{jk} J_\varphi(x_k\colon \C[x_\l\colon \l\neq k]).
	\end{align}
\end{rem}

%%%%%%%%%%%%%%%%%%%%%%%%%%%%%%%%%%%%%%%%%%%%%%%%%%%%%%%%%%%%%%%%%%%%%%%%%%%%%
%					The Leibniz Rule						 %
%%%%%%%%%%%%%%%%%%%%%%%%%%%%%%%%%%%%%%%%%%%%%%%%%%%%%%%%%%%%%%%%%%%%%%%%%%%%%

\section{Closable $\mu$-Modular Derivations}\label{Closable}

In order to gain insights into the von Nuemann algebra $M$ (as opposed to just the $*$-algebra $\P$) we must necessarily consider $\mu$-modular derivations $\delta$ which are closable, such as when the conjugate variable exists. By employing the theory of Dirichlet forms, we will see that the restriction of $\bar\delta$ to $M^\vphi$ satisfies the Leibniz rule. This result will allow us to establish a type of Kaplansky's density theorem (\emph{cf.} Theorem \ref{Kaplansky}), as well as some bounds for $\bar\delta$ (when the conjugate variable exists) that imply the domain of $\delta^*$ is in fact quite large.

%%%%%%%%%%%%%%%%%%%%%%%%%%%%%%%%%%%%%%%%%%%%%%%%%%%%%%%%%%%%%%%%%%%%%%%%%%%%%

\subsection{Some preliminary observations}\label{Prelim_Obs}

Any closable operator is assumed to have dense domain. In order to simplify the exposition, for any (unbounded) closed operator $T\colon \H_1\to \H_2$, we denote
	\[
		\|\xi\|_T = \sqrt{\|\xi\|_{\H_1}^2 + \|T\xi\|_{\H_2}^2} \qquad \xi\in \dom(T).
	\]
We first observe that $\mu$-modularity for a closable derivation has implications for how the closure restricts to the eigenspaces of $\Delta_\vphi$.

\begin{lem}\label{derivation_on_eigenspaces}
Let $\delta\colon L^2(M,\vphi)\to L^2(M\bar\otimes M^{op})$ be a closable $\mu$-modular derivation for some $\mu>0$, with $\dom(\delta)$ generated by eigenoperators of $\sigma^\vphi$. If we denote the closure by $\bar\delta$, then for $I,J\subset \R_+^\times$ disjoint subsets
	\begin{itemize}
	\item[(i)] $\pi_I\dom(\bar{\delta})\subset \dom(\bar{\delta})$;
	\item[(ii)] $\dom(\delta)\cap E_I$ is a core of $\bar{\delta}\mid_{\pi_I\dom(\bar{\delta})}$; and
	\item[(iii)] $\bar{\delta}(\pi_I\dom(\bar{\delta}))\perp \bar{\delta}(\pi_J\dom(\bar{\delta}))$.
	\end{itemize}
Furthermore, if the conjugate derivation $\hat\delta$ is closable with closure $\bar{\hat\delta}$, then $S_\vphi(\pi_I\dom(\bar{\delta}) ) \subset \dom(\bar{\hat\delta})$ with $\bar{\hat\delta}\circ S_\vphi\circ\pi_I(\cdot) = [\bar{\delta}\circ\pi_I(\cdot)]^\dagger$ for any subset $I\subset \R_+^\times$ bounded above.
\end{lem}
\begin{proof}
Since $\dom(\delta)$ is generated by eigenoperators, $\pi_I\dom(\delta)=\dom(\delta)\cap E_I$ for any $I\subset \R_+^\times$. The $\mu$-modularity implies that if $p\in \dom(\delta)\cap E_I$ and $q\in \dom(\delta)\cap E_J$ for disjoint sets $I$ and $J$, then $\<\delta(p),\delta(q)\>_{\text{HS}}=0$. Now, given $\xi\in \dom(\bar{\delta})$, let $\{p_n\}_{n\in\N}\subset \dom(\delta)$ approximate $\xi$ in the $\|\cdot\|_{\bar{\delta}}$-norm. Then 
	\begin{align*}
		\| \xi -  p_n \|_\varphi^2 = \| \pi_I(\xi -  p_n)\|_\varphi^2 + \| (1-\pi_I)(\xi - p_n)\|_\varphi^2
	\end{align*}
implies $\pi_I  p_n$ converges to $\pi_I \xi$ in $L^2(M,\varphi)$. Also, since $1-\pi_I=\pi_J$ for $J=R_+^\times\setminus I$, $\mu$-modularity implies
	\begin{align*}
		\| \delta( p_n -  p_m)\|_\varphi^2 = \| \delta(\pi_I( p_n -  p_m))\|_\varphi^2 + \|\delta((1-\pi_I)( p_n-  p_m))\|_\varphi^2.
	\end{align*}
So $\delta(\pi_I p_n)$ is a Cauchy sequence and must converges to some $\eta\in L^2(M\bar{\otimes}M^{op})$; that is, $\pi_I \xi\in \dom(\bar{\delta})$. This establishes (i), and our proof establishes (ii) and (iii).

Finally, let $I\subset \R_+^\times$ be bounded above by $\lambda>0$. Then clearly $E_I\subset\dom(\Delta_\varphi^{1/2})=\dom(S_\varphi)$. Given $\xi\in \pi_I\dom(\bar{\delta})$, let $\{ p_n\}_{n\in\N}\in \pi_I\dom(\delta)$ approximate $\xi$ in the $\|\cdot\|_{\bar{\delta}}$-norm. Since $\Delta_\vphi$ is bounded on $E_I$ (and hence so is $S_\vphi=J_\vphi \Delta_\vphi^{1/2}$), $S_\varphi  p_n$ converges to $S_\varphi \xi$. Since $\dagger$ is an isometry we have that $\hat{\delta}(p_n^*)= \delta( p_n)^\dagger$ converges to $\bar{\delta}(\xi)^\dagger$.
\end{proof}

From the above lemma, we obtain an immediate (albeit partial) extension of the Leibniz rule, which we shall use later to obtain a more robust result. Recall that $M_\vphi$ acts boundedly on $L^2(M,\varphi)$ as $S_\vphi M_\infty S_\vphi$.

\begin{lem}\label{extension_of_Leibniz_rule}
Let $\delta\colon L^2(M,\vphi)\to L^2(M\bar\otimes M^{op})$ be a closable $\mu$-modular derivation, for some $\mu>0$, with $\dom(\delta)$ generated by eigenoperators of $\sigma^\vphi$. Then $\bar\delta$ is defined on the products $\dom(\bar\delta)\cdot \dom(\delta)$ and $\dom(\delta)\cdot\left(\pi_I\dom(\bar\delta)\right)$, $I\subset \R_+^\times$ bounded above, and satisfies the Leibniz rule on these products. Moreover, $\bar\delta$ is defined on $\ker(\delta)\cdot \dom(\bar\delta)\cdot \ker(\delta)$ and satisfies the Leibniz rule on this product.
\end{lem}
\begin{proof}
For $\xi\in \dom(\bar{\delta})$, let $\{x_n\}_{n\in\N}\subset \dom(\delta)$ approximate $\xi$ in the $\|\cdot\|_{\bar{\delta}}$-norm. Then for $p\in \dom(\delta)$ we have
	\begin{align*}
		\| \xi\cdot p - x_n p\|_\varphi \leq \| \sigma_{i/2}^\varphi (p)\|\|\xi - x_n\|_\varphi \to 0.
	\end{align*}
Similarly, say $\delta(p)=\sum_j a_j\otimes b_j \in M_\infty\otimes M_\infty^{op}$ (a finite sum) then
	\begin{align*}
		\| (\xi-x_n)\cdot \delta(p)\|_{\text{HS}} \leq \sum_j \| [(\xi - x_n)\cdot a_j]\otimes b_j\|_{\text{HS}} \leq \sum_j \|\sigma_{i/2}^\varphi(a_j)\|\|\xi - x_n\|_\varphi \|b_j^*\|_\varphi\to 0.
	\end{align*}
Therefore it is clear that
	\begin{align*}
		\| \bar{\delta}(\xi)\cdot p + \xi\cdot \delta(p) - \delta(x_np)\|_{\text{HS}} =\| \bar{\delta}(\xi)\cdot p + \xi\cdot \delta(p) - \delta(x_n)\cdot p - x_n\cdot \delta(p)\|_{\text{HS}} \to 0,
	\end{align*}
so that $\xi\cdot p\in \dom(\bar{\delta})$ with $\bar{\delta}(\xi\cdot p)= \bar{\delta}(\xi)\cdot p + \xi\cdot \delta(p)$.

For $\xi\in \pi_I\dom(\bar{\delta})$, $I\subset R_+^\times$ bounded above, the proof of the final assertion in Lemma \ref{derivation_on_eigenspaces} shows that if $\{x_n\}_{n\in\N}\subset \dom(\delta)\cap E_I$ approximates $\xi$ in the $\|\cdot\|_{\bar{\delta}}$-norm, then in fact $\|\xi - x_n \|_\varphi^\#\to 0$. Thus if $p\in \dom(\delta)$ and $\delta(p)=\sum_j a_j\otimes b_j$, then
	\begin{align*}
		\| \delta(p)\cdot (\xi - x_n)\|_{\text{HS}} \leq \sum_j \|a_j\|_\varphi \| S_\varphi b_j\cdot(\xi - x_n)\|_\varphi \leq \sum_j \|a_j\|_\varphi \|\sigma_{-i/2}^\varphi(b_j)\| \| S_\varphi (\xi - x_n)\|_\varphi \to 0.
	\end{align*}
We of course also have
	\begin{align*}
		\| p\cdot \xi - px_n\|_\varphi\to 0,\qquad\text{and}\qquad\| p\cdot \bar\delta(\xi) - p\cdot \delta(x_n)\|_{\text{HS}}\to 0,
	\end{align*}
so that $p\cdot \xi\in \dom(\bar{\delta})$ with $\bar{\delta}(p\cdot \xi) = \delta(p)\cdot \xi + p\cdot \bar{\delta}(\xi)$.

Finally, for $p,q\in \ker(\delta)\subset \dom(\delta)$ and $\xi\in \dom(\bar{\delta})$ $\|\cdot\|_{\bar{\delta}}$-approximated by $\{x_n\}_{n\in\N} \subset \dom(\delta)$, we easily observe that $px_nq$ $\|\cdot\|_{\bar{\delta}}$-approximates $p\cdot \xi\cdot q$ since $\delta(px_nq)=p\cdot \delta(x_n)\cdot q$.
\end{proof}

We will also frequently make use of the following lemma, which is a variation of \cite[Lemma 7.2]{CS03}. We first establish some notation. Let $a\in M$ be self-adjoint with spectrum contained in a compact interval $I\subset \R$. We consider a representation $LR_a$ of $C(I)\otimes C(I)=C(I\times I)$ in $C^*(a)\otimes C^*(a)^{op}$ defined by
	\[
		LR_a(f\otimes g)=f(a)\otimes g(a)
	\]
for $f,g\in C(I)$. In particular, if $h\in C(I\times I)$ factors as $h(s,t)=f(s)g(t)$, then $LR_a(h)=f(a)\otimes g(a)$. For $f\in C^1(I)$, denote
	\[
		\tilde f(s,t):=\left\{\begin{array}{cl}	f'(t) & \text{if }t=s\\
						\frac{f(t)-f(s)}{t-s} & \text{otherwise}\end{array}\right. \in C(I\times I).
	\]

\begin{lem}\label{C^1_functional_calculus}
Let $\delta\colon L^2(M,\vphi)\to L^2(M\bar\otimes M^{op})$ be a closable derivation with closure $\bar\delta$ and $\dom(\delta)$ a unital $*$-algebra. If $a=a^*\in \dom(\delta)$ has spectrum contained in a compact interval $I\subset \R$, then for every $f\in C^1(I)$ we have $f(a)\in \dom(\bar{\delta})$ with $\bar\delta(f(a)) = LR_a(\tilde f)\# \delta(a)$. Moreover, if $g\in C(I)$ is Lipschitz with constant $C$, then $g(a)\in \dom(\bar\delta)$ with $\|\bar\delta (g(a))\|_{\text{HS}}\leq C \|\delta(a)\|_{\text{HS}}$.

\end{lem}
\begin{proof}
The proof of the first part is identical to that in \cite[Lemma 7.2]{CS03}, but we note that $\delta(1)=0$ allows us to consider $f\in C^1(I)$ with $f(0)\neq 0$. For Lipschitz functions, approximate by functions in $C^1(I)$.
\end{proof}

%%%%%%%%%%%%%%%%%%%%%%%%%%%%%%%%%%%%%%%%%%%%%%%%%%%%%%%%%%%%%%%%%%%%%%%%%%%%%

\subsection{Dirichlet forms arising from non-tracial derivations}\label{Dirichlet}

Fix a $\mu$-modular derivation $\delta$ with $\dom(\delta)$ a $*$-algebra generated by eigenoperators of $\sigma^\vphi$. We consider the following quadratic form on $L^2(M^\vphi,\vphi)$:
	\[
		\mc{E}[\ \cdot\ ]:=\|\delta(\ \cdot\ )\|_{\text{HS}}^2 + \|\hat\delta(\ \cdot\ )\|_{\text{HS}}^2,
	\]
with $\dom(\mc{E})=M^\vphi\cap \dom(\delta)$. If $\delta=\hat\delta$, $\mc{E}$ is defined by the above formula divided by a factor of two. While $(M^\vphi, \vphi)$ is a tracial von Neumann algebra, the results of this section are not entirely subsumed by the tracial case since $\delta$ may be valued outside of $M^\vphi\otimes (M^\vphi)^{op}$.

We claim that $\mc{E}$ is $J_\vphi$-real. Indeed, recall that $J_\vphi\mid_{M^\vphi}=S_\vphi\mid_{M^\vphi}$ and that $\dagger$ is an isometry for $\|\cdot\|_{\text{HS}}$. So, for $p\in \dom(\mathscr{E})$ we have
	\begin{align*}
		\mathscr{E}[J_\varphi p] &= \|\delta(p^*)\|_{\text{HS}}^2 + \|\hat\delta(p^*)\|_{\text{HS}}^2\\
			&= \| \hat\delta( p)^\dagger\|_{\text{HS}}^2 + \|\delta( p)^\dagger\|_{\text{HS}}^2 = \mathscr{E}[p].
	\end{align*}

We also note that if the conjugate variables to either $\delta$ or $\hat\delta$ exist (and hence both exist by Lemma \ref{conj_var_entire}), then clearly $\mathscr{E}$ is closable, say with closure $\bar{\mathscr{E}}$. One has that $\xi \in \dom(\bar{\mathscr{E}})$ if and only if $\xi\in L^2(M^\vphi,\vphi)\cap \dom(\bar{\delta})\cap \dom(\bar{\hat\delta})$ and there exists a sequence $(p_n)_{n\in\N}\subset M^\vphi\cap \dom(\delta)$ such that $p_n\to \xi$ in $L^2(M^\vphi,\varphi)$ and simultaneously $\delta(p_n)\to \bar\delta(\xi)$ and $\hat\delta(p_n)\to \bar{\hat\delta}(\xi)$ in $L^2(M\bar\otimes M^{op})$.

\begin{prop}\label{completely_Dirichlet}
Let $\delta\colon \dom(\delta)\to M_\infty\otimes M_\infty^{op}$ be a $\mu$-modular derivation for some $\mu>0$, with $\dom(\delta)$ generated by eigenoperators of $\sigma^\vphi$. Assume the conjugate variable to $\delta$ exists so that $\mathscr{E}$ is closable with closure $\bar{\mathscr{E}}$. Then $\bar{\mathscr{E}}$ is a completely Dirichlet form on $L^2(M^\vphi,\vphi)$.
\end{prop}
\begin{proof}
Let $\eta=J_\varphi \eta\in \dom(\bar{\mathscr{E}})$, and let $(p_n)_{n\in \N}\subset \dom(\mathscr{E})$ be a sequence converging to $\eta$ with respect to $\|\cdot\|_{\bar{\mathscr{E}}}$. By replacing $p_n$ with $p_n + J_\varphi p_n \in \dom(\mathscr{E})$, we may assume each $p_n$ is $J_\varphi$-real. Since $J_\vphi\mid_{M^\vphi}=S\mid_{M^\vphi}$, each $p_n$ is self-adjoint, and so by Lemma \ref{convex_proj_for_M} $p_n\wedge 1 = f(p_n)$ where $f(t)=\min\{t,1\}$. We first claim $p_n\wedge 1\in \dom(\bar{\mathscr{E}})$ with
	\[
		\bar{\mathscr{E}}[p_n\wedge 1] \leq \mathscr{E}[p_n].
	\]
Indeed, let $I\subset \R$ be an interval containing $1$ and the spectrum of $p_n$. Then there exists a sequence $\{g_k\}_{k\in \N}$ of polynomials with real coefficients such that $g_k$ approximate $f$ uniformly on $I$, and $g_k'$ are uniformly bounded by say $1+\frac{1}{k}$ on $I$. Then $\{g_k(p_n)\}_{k\in \N}$ is a sequence of self-adjoint operators in $\dom(\mathscr{E})$ that by Lemma \ref{C^1_functional_calculus} satisfy
	\[
		\mathscr{E}[ g_k(p_n)] \leq \left(1+\frac{1}{k}\right)^2\mathscr{E}[p_n].
	\]
Moreover, $g_k(p_n)\to f(p_n)$ in $L^2(M^\vphi,\varphi)$. Extending $\bar{\mathscr{E}}$ to all of $L^2(M^\vphi,\varphi)$ by letting $\bar{\mathscr{E}}\equiv +\infty$ outside of $\dom(\bar{\mathscr{E}})$, we have
	\[
		\bar{\mathscr{E}}[p_n\wedge 1] = \bar{\mathscr{E}}[f(p_n)] \leq \liminf_{k\to\infty} \mathscr{E}[g_k(p_n)] \leq \mathscr{E}[p_n]
	\]
from the lower semicontinuity guaranteed by $\bar{\mathscr{E}}$ being closed. In particular, $\bar{\mathscr{E}}[p_n\wedge 1]<+\infty$ so that $p_n\wedge 1\in \dom(\bar{\mathscr{E}})$.

Now, since $(\cdot)\wedge 1$ is a projection onto a closed convex set, we know $(p_n\wedge 1)_{n\in \N}$ converges to $\eta\wedge 1$ with respect to $\|\cdot\|_\varphi$. Thus, using lower semicontinuity again we have
	\[
		\bar{\mathscr{E}}[\eta\wedge 1] \leq \liminf_{n\to\infty} \bar{\mathscr{E}}[p_n\wedge 1] \leq \liminf_{n\to\infty} \mathscr{E}[p_n]=\bar{\mathscr{E}}[\eta].
	\]
Thus, $\bar{\mathscr{E}}$ is Markovian and hence Dirichlet.

Given $n\in\N$, we note that the canonical extension $\bar{\mathscr{E}}^{(n)}$ is defined for $T\in M_n(\dom(\bar{\mathscr{E}}))$ by
	\begin{align*}
		\bar{\mathscr{E}}^{(n)}\left[ T\right] &= \left\| \left(\bar\delta\otimes I_n\right)\left(T\right)\right\|_{\frac{1}{n}(\varphi\otimes\varphi^{op})\circ\Tr}^2 + \left\| \left( \bar{\hat\delta}\otimes I_n\right)\left(T\right)\right\|_{\frac{1}{n}(\varphi\otimes\varphi^{op})\circ\Tr}^2.
	\end{align*}
If $\xi$ is the conjugate variable to $\delta$, let $\hat\xi$ denote the conjugate variable to $\hat\delta$, which exists by Lemma \ref{conj_var_entire}. It is easy to see that $\delta\otimes I_n$ and $\hat\delta\otimes I_n$ are derivations on $M_n(\dom(\mathscr{E}))$ with conjugate variables $\xi\otimes I_n$ and $\hat\xi\otimes I_n$, respectively. Consequently, by the same argument preceding the proposition, we see that $\bar{\mathscr{E}}^{(n)}$ is closed. The argument showing that $\bar{\mathscr{E}}$ is a Dirichlet form relied only on the functional calculus, hence we repeat the argument to see that $\bar{\mathscr{E}}^{(n)}$ is also a Dirichlet form. Thus $\bar{\mathscr{E}}$ is completely Dirichlet.
\end{proof}

Using the proof of \cite[Proposition 4.7]{Cip08}, we obtain the following as an immediate corollary.

\begin{cor}\label{Dirichlet_alg}
Let $\delta\colon \dom(\delta)\to M_\infty\otimes M_\infty^{op}$ be a $\mu$-modular derivation for some $\mu>0$, with $\dom(\delta)$ generated by eigenoperators of $\sigma^\vphi$. Assume the conjugate variable to $\delta$ exists so that $\mathscr{E}$ is closable with closure $\bar{\mathscr{E}}$. Then the set $M^\vphi\cap \dom(\bar{\mathscr{E}})$ is a $*$-algebra.
\end{cor}

For the remainder of this section, we will assume that the conjugate variable to $\delta$ exists, so that the conjugate variable to $\hat\delta$ also exists and both $\delta$ and $\hat\delta$ are closable with closures $\bar\delta$ and $\bar{\hat\delta}$, respectively. We define $\overline{\dom{}}(\delta\oplus\hat\delta)$ to be the subspace of elements in $\dom(\bar\delta)\cap\dom(\bar{\hat\delta})$ that can be approximated by elements of $\dom(\delta)$ simultaneously in the $\|\cdot\|_{\bar\delta}$ and $\|\cdot\|_{\bar{\hat\delta}}$ norms. When $\delta=\hat\delta$, this set is simply $\dom(\bar\delta)$.

\begin{prop}\label{closure_is_der}
Let $\delta\colon \dom(\delta)\to M_\infty\otimes M_\infty^{op}$ be a $\mu$-modular derivation for some $\mu>0$, with $\dom(\delta)$ generated by eigenoperators of $\sigma^\vphi$. Assume the conjugate variable to $\delta$ exists. Then the set $M^\vphi\cap \overline{\dom{}}(\delta\oplus \hat\delta)$ is a $*$-algebra on which both $\bar\delta$ and $\bar{\hat\delta}$ satisfy the Leibniz rule.
\end{prop}
\begin{proof}
Since $M^\vphi\cap \overline{\dom{}}(\delta\oplus\hat\delta)=M^\vphi\cap \dom(\bar{\mathscr{E}})$ by the remarks preceding Proposition \ref{completely_Dirichlet}, we see this set is a $*$-algebra by Corollary \ref{Dirichlet_alg}.

Let $a,b\in M^\vphi\cap \overline{\dom{}}(\delta\oplus\hat\delta)$. We will show $\bar\delta(ab)=\bar\delta(a)\cdot b + a\cdot \bar\delta(b)$ (the proof for $\bar{\hat\delta}$ being similar). Let $(p_n)_{n\in\N}\subset M^\vphi\cap\dom(\delta)$ approximate $a$ in the $\|\cdot\|_{\bar\delta}$-norm. Since $L^2(M,\varphi)$ and $L^2(M\bar\otimes M^{op})$ admit bounded right actions of $M_\infty$ we have
	\[
		\|p_n b - ab\|_\vphi\to 0\qquad \text{and}\qquad \|\delta(p_n)\cdot b - \bar\delta(a)\cdot b\|_{\text{HS}}\to 0.
	\]
Also, since $\|p_n^* - a^*\|_\vphi=\|p_n - a\|_\vphi\to 0$, $p_n\cdot \bar\delta(b)$ converges to $a\cdot\bar\delta(b)$ weakly against $M_\infty\otimes M_\infty^{op}\subset L^2(M\bar\otimes M^{op})$. Thus $\delta(p_n)\cdot b + p_n\cdot\bar\delta(b)$ converges weakly to $\bar\delta(a)\cdot b + a\cdot\bar\delta(b)$ against $M_\infty\otimes M_\infty^{op}$. On the other hand, from Lemma \ref{extension_of_Leibniz_rule} we know $p_n b\in \dom(\bar\delta)$ with $\bar\delta(p_n b) = \delta(p_n)\cdot b + p_n\cdot\bar\delta(b)$. Consequently, for any $\eta\in \dom(\delta)\otimes \dom(\delta)^{op}$ (which also lies in $\dom(\delta^*)$ by Lemma \ref{adjoint_formula}) we have
	\begin{align*}
		\<\bar\delta(a)\cdot b + a \cdot\bar\delta(b), \eta\>_{\text{HS}} &= \lim_{n\to\infty} \< \delta(p_n)\cdot b + p_n\cdot \bar\delta(b), \eta\>_{\text{HS}}\\
			&= \lim_{n\to\infty} \< \bar\delta(p_n b), \eta\>_{\text{HS}}= \lim_{n\to\infty} \< p_n b, \delta^*(\eta)\>_{\vphi}\\
			&= \< ab, \delta^*(\eta)\>_{\vphi}= \<\bar\delta(ab), \eta\>_{\text{HS}}.
	\end{align*}
Since $\dom(\delta)\otimes\dom(\delta)^{op}$ is dense, we have $\bar\delta(ab)=\bar\delta(a)\cdot b + a\cdot\bar\delta(b)$.
\end{proof}

One particular consequence of this proposition is that for any $x\in M^\vphi\cap \overline{\dom{}}(\delta\oplus \hat\delta)$, we have $p(x)\in M^\vphi\cap \overline{\dom{}}(\delta\oplus \hat\delta)$ for any polynomial $p$. By the same argument as in Lemma \ref{C^1_functional_calculus} we obtain the following corollary.

\begin{cor}\label{extC1}
Let $\delta\colon \dom(\delta)\to M_\infty\otimes M_\infty^{op}$ be a $\mu$-modular derivation for some $\mu>0$, with $\dom(\delta)$ generated by eigenoperators of $\sigma^\vphi$. Assume the conjugate variable to $\delta$ exists. Let $a=a^*\in M^\vphi\cap \overline{\dom{}}(\delta\oplus \hat\delta)$ with spectrum contained in a compact interval $I\subset \R$. Then for any $f\in C^1(I)$, $f(a)\in M^\vphi\cap\overline{\dom{}}(\delta\oplus \hat\delta)$ with $\bar\delta( f(a)) = LR_a(\tilde f)\# \bar\delta(a)$ (and similarly for $\bar{\hat\delta}$). Moreover, if $g\in C(I)$ is Lipschitz with constant $C$, then $g(a)\in M^\vphi\cap \overline{\dom{}}(\delta\oplus \hat\delta)$ with $\|\bar\delta(g(a))\|_{\text{HS}}\leq C \|\bar\delta(a)\|_{\text{HS}}$ (and similarly for $\bar{\hat\delta}$).
\end{cor}

We conclude with the following analogue of \cite[Proposition 6]{Dab10}, which we obtain via the same proof. This result is a version of Kaplansky's density theorem for elements in the domain of a closed derivation. 

\begin{thm}\label{Kaplansky}
Let $\delta\colon \dom(\delta)\to M_\infty\otimes M_\infty^{op}$ be a $\mu$-modular derivation for some $\mu>0$, with $\dom(\delta)$ generated by eigenoperators of $\sigma^\vphi$. Assume the conjugate variable to $\delta$ exists so that $\delta$ and $\hat\delta$ are closable with closures $\bar\delta$ and $\bar{\hat\delta}$, respectively. For any $x\in M^\vphi\cap \overline{\dom{}}(\delta\oplus \hat\delta)$, $\lambda>0$, there exists a sequence $(p_n)_{n\in\N}\subset M^\vphi\cap\dom(\delta)$ converging $*$-strongly to $x$ and approximating $x$ simultaneously in the $\|\cdot\|_{\bar\delta}$ and $\|\cdot\|_{\bar{\hat\delta}}$ norms such that $\|p_n\|\leq \|x\|$ for all $n\in\N$.
\end{thm}

%%%%%%%%%%%%%%%%%%%%%%%%%%%%%%%%%%%%%%%%%%%%%%%%%%%%%%%%%%%%%%%%%%%%%%%%%%%%%

\subsection{Norm boundedness of $\delta^*$}\label{Boundedness}

By exhibiting some boundedness conditions for a derivation $\delta$, we will be able to see that the domain of $\delta^*$ is in fact quite large. The following proposition is the non-tracial analogue of \cite[Lemma 12]{Dab10}.

\begin{prop}\label{one_component_of_derivative_is_bounded}
Let $\delta\colon \dom(\delta)\to M_\infty\otimes M_\infty^{op}$ be a $\mu$-modular derivation for some $\mu>0$, with $\dom(\delta)$ generated by eigenoperators of $\sigma^\vphi$. Assume the conjugate variable $\xi$ to $\delta$ exists so that $\delta$ and $\hat\delta$ are closable with closures $\bar\delta$ and $\bar{\hat\delta}$, respectively. Then for $u\in M^\vphi\cap\overline{\dom{}}(\delta\oplus\hat\delta)$ unitary 
	\begin{align*}
		\| u\cdot\xi- (1\otimes \varphi)(\bar{\hat\delta}(u)) \|_\varphi &= \|\xi\|_\varphi,\text{ and}\\
		\| \xi\cdot u - (\varphi\otimes \sigma_{-i}^\varphi)(\bar{\hat\delta}(u)) \|_\varphi &=\|\xi\|_\varphi.
	\end{align*}
For $a\in M^\vphi\cap \overline{\dom{}}(\delta\oplus\hat\delta)$ self-adjoint we have
	\begin{align*}
		\| a\cdot\xi - (1\otimes \varphi)(\bar{\hat\delta}(a)) \|_\varphi &\leq \|a\| \|\xi\|_\varphi,\text{ and}\\
		\| \xi\cdot a - (\varphi\otimes \sigma_{-i}^\varphi)(\bar{\hat\delta}(a)) \|_\varphi &\leq \|a\| \|\xi\|_\varphi.
	\end{align*}
For any $x\in M^\vphi\cap \overline{\dom{}}(\delta\oplus\hat\delta)$ we have
	\begin{align*}
		\| x\cdot\xi - (1\otimes \varphi)(\bar{\hat\delta}(x)) \|_\varphi &\leq 2\|x\| \|\xi\|_\varphi,\text{ and}\\
		\| \xi\cdot x - (\varphi\otimes \sigma_{-i}^\varphi)(\bar{\hat\delta}(x)) \|_\varphi &\leq 2\|x\| \|\xi\|_\varphi.
	\end{align*}
Consequently, for $x\in M^\vphi\cap \overline{\dom{}}(\delta\oplus\hat\delta)$ we have
	\begin{align}\label{partial_derivative_bounds}
		\| (1\otimes \varphi)(\bar{\hat\delta}(x)\|_\varphi&\leq 3\|x\| \|\xi\|_\varphi,\text{ and}\\
		\| (\varphi\otimes\sigma_{-i}^\varphi(\bar{\hat\delta}(x))\|_\varphi & \leq 3\| x\| \|\xi\|_\varphi.\nonumber
	\end{align}
Moreover, for any $\lambda\in\R_+^\times$ and any $p\in \mc{E}_{\lambda}(\dom(\delta))$ we have
	\begin{align*}
		\| (1\otimes \varphi)(\bar{\hat\delta}(p)\|_\varphi&\leq 3\|p\| \|\xi\|_\varphi,\text{ and}\\
		\| (\varphi\otimes\sigma_{-i}^\varphi(\bar{\hat\delta}(p))\|_\varphi & \leq 3\lambda^{1/2}\| p\| \|\xi\|_\varphi.
	\end{align*}
\end{prop}
\begin{proof}
It is clear that the estimates in (\ref{partial_derivative_bounds}) follow immediately from the previous ones (along with the well known inequality $\| wx\|_\varphi \leq \| \sigma_{-i/2}^\varphi(x^*)\| \|w\|_\varphi$), which we now prove.

For notational simplicity, we write $b(x):=(1\otimes \varphi)\left[\bar{\hat\delta}(x)\right]$ for $x\in M^\vphi\cap \overline{\dom{}}(\delta\oplus\hat\delta)$. We first establish the inequalities which involve $b$. We let $\hat\xi$ denote the conjugate variable to $\hat\delta$, which exists by Lemma \ref{conj_var_entire}. First, we consider $p\in \dom(\delta)$. Using Lemma \ref{adjoint_formula} we have
	\begin{align*}
		\|b(p)\|_\varphi^2 &= \< b(p)\otimes 1, \hat\delta(p)\>_{\text{HS}}\\
						&= \< b(p)\hat\xi - (1\otimes \varphi)(\delta(b(p))), p\>_\varphi\\
						&= \<b(p)\hat\xi - (1\otimes \varphi\otimes \varphi)( \delta\otimes 1)(\hat\delta(p)),p\>_\varphi\\
						&= \<b(p) \hat\xi - (1\otimes \varphi\otimes \varphi)(1\otimes \hat\delta)(\delta(p)), p\>_\varphi\\
						&= \<b(p) \hat\xi - (1\otimes \varphi)(\delta(p)\# 1\otimes S_\vphi(\hat\xi) ), p\>_\varphi\\
						&= \<b(p) \hat\xi - (1\otimes \varphi) (\delta(p)\cdot \xi), p\>_\varphi\\
						&=\<b(p), p\xi\>_\varphi - \<\delta(p), p\otimes (S_\varphi \xi)\>_{\text{HS}},
	\end{align*}
where we have used (\ref{F_of_conjugate_variable_to_free_difference_quotient}) in the final two equalities. We can obtain the equality of the first and last expressions for any $x\in M^\vphi\cap\overline{\dom{}}(\delta\oplus\hat\delta)$ by applying Theorem \ref{Kaplansky} to approximate $x$ by polynomials $p$ with $\|p\|\leq \|x\|$. We next compute
	\begin{align*}
		\|b(x)\|_\varphi^2 &= \<b(x), x\xi\>_\varphi - \<\bar\delta(x), x\otimes (S_\varphi \xi)\>_{\text{HS}}\\
			&=\<b(x), x\xi\>_\varphi - \< x^*\cdot\bar\delta(x), 1\otimes (S_\vphi\xi)\>_{\text{HS}}\\
			&=\<b(x), x\xi\>_\varphi - \< \bar\delta(x^*x), 1\otimes (S_\vphi \xi)\>_{\text{HS}} + \<\bar\delta(x^*)\cdot x,1\otimes (S_\vphi\xi)\>_{\text{HS}}.
	\end{align*}
We focus on the third term above:
	\begin{align*}
		\<\bar\delta(x^*)\cdot x,1\otimes (S_\vphi \xi)\>_{\text{HS}} &= \<\bar{\hat\delta}(x)^\dagger, 1\otimes S_\vphi(x\xi)\>_{\text{HS}}\\
			&= \<(x\xi)\otimes 1, \bar{\hat\delta}(x)\>_{\text{HS}}\\
			&= \<x\xi, b(x)\>_\varphi.
	\end{align*}
Thus we have shown
	\begin{align*}
		\|b(x)\|_\varphi^2 &= 2\Re{\<b(x), x\xi\>_\varphi} - \<\bar\delta(x^*x), 1\otimes S_\vphi(\xi)\>_{\text{HS}},
	\end{align*}
and consequently
	\begin{align}\label{eigenvector_jump_in}
		\| x\xi - b(x)\|_\varphi^2  = \|x\xi\|_\varphi^2 - \<\bar\delta(x^*x), 1\otimes S_\vphi(\xi)\>_{\text{HS}}.
	\end{align}
	
Now, if $x=u$ is a unitary then the above reduces to
	\begin{align*}
		\| u \xi - b(u)\|_\varphi^2 = \| u \xi\|_\varphi^2 = \|\xi\|_\varphi^2.
	\end{align*}
For self-adjoint $x=a$, let $\alpha>1$ and write $\frac{a}{\alpha\|a\|}=\frac{u_1+u_2}{2}$ as a sum of two unitaries $u_1,u_2$. In particular,
	\begin{align*}
		u_1&= \frac{a}{\alpha\|a\|} + i \sqrt{ 1- \frac{a^2}{\alpha^2\|a\|^2}}\\
		u_2&= \frac{a}{\alpha\|a\|} - i \sqrt{ 1- \frac{a^2}{\alpha^2\|a\|^2}}.
	\end{align*}
Then $u_1, u_2\in M^\vphi\cap \overline{\dom{}}(\delta\oplus\hat\delta)$ by Corollary \ref{extC1} and by the first part of the proof we have
	\begin{align*}
		\| a\xi - b(a)\|_\varphi &\leq \frac{\alpha\|a\|}{2}\left( \|u_1\xi - b(u_1)\|_\varphi + \|u_2\xi - b(u_2)\|_\varphi\right)= \alpha\|a\|\|\xi\|_\varphi^2.
	\end{align*}
Letting $\alpha\to 1$ yields the desired inequality. Finally, for generic $x$ simply write it as the sum of its real and imaginary parts, use the triangle inequality, and apply the previous bound.

Towards proving the inequalities without $b$, we recall that the Tomita operator $S_\varphi$ has the polar decomposition $S_\varphi =J_\varphi \Delta_\varphi^{1/2} $, where $J_\varphi$ is an anti-linear isometry and $\Delta_\varphi$ is the modular operator. In particular, $1=S_\varphi S_\varphi = J_\varphi \Delta_\varphi^{1/2} S_\varphi$. Hence for $x\in M^\vphi\cap\overline{\dom{}}(\delta\oplus\hat\delta)$ we have
	\begin{align*}
		\xi x - (\varphi\otimes \sigma_{-i}^\varphi)(\bar{\hat\delta}(x)) &= J_\varphi \Delta_\varphi^{1/2}\left[ x^* \mu \hat\xi - (\sigma_i^\varphi\otimes \varphi)(\bar\delta(x^*))\right]\\
			&= J_\varphi \left[ \sigma_{-i/2}^\varphi(x^*)\mu \Delta_\vphi^{1/2}\hat\xi - (1\otimes \varphi)( \sigma_{i/2}^\varphi\otimes \sigma_{-i/2}^\varphi)(\bar\delta(x^*))\right]\\
			&= J_\varphi \mu^{1/2} \left[ \sigma_{-i/2}^\varphi(x^*)\hat\xi - (1\otimes\varphi)(\bar\delta(\sigma_{i/2}^\varphi(x^*)))\right]\\
			&=J_\varphi \mu^{1/2} \left[ x^*\hat\xi - (1\otimes\varphi)(\bar\delta(x^*))\right]
	\end{align*}
where we have used (\ref{modular_operator_of_conjugate_variables_to_free_difference_quotient}) with $z=-i/2$ in the second-to-last equality. Consequently
	\begin{align}\label{J_trick}
		\| \xi x- (\varphi\otimes \sigma_{-i}^\varphi)(\bar{\hat\delta}(x))\|_\varphi = \mu^{1/2} \|x^*\hat\xi - (1\otimes\vphi)(\bar\delta(x^*))\|_\varphi
	\end{align}
If $x=u$ is unitary, then this and the previously established equality for unitaries yields
	\begin{align*}
		\| \xi u - (\varphi\otimes \sigma_{-i}^\varphi)(\bar{\hat\delta}(u)) \|_\varphi = \mu^{1/2} \|\hat\xi\|_\varphi = \|\xi\|_\varphi
	\end{align*}
where the last equality follows from a simple computation using (\ref{adjoint_of_conjugate_variables_to_free_difference_quotient}) and  (\ref{modular_operator_of_conjugate_variables_to_free_difference_quotient}).
	
The inequalities for $a,x\in M^\vphi\cap \overline{\dom{}}(\delta\oplus\hat\delta)$ self-adjoint and generic, respectively, also follow from (\ref{J_trick}) by using the previously established inequalities and $\mu^{1/2}\|\hat{\xi}\|_\vphi=\|\xi\|_\vphi$. 

Now, fix $\lambda\in \R_+^\times$, and let $p\in \mc{E}_\lambda(\dom(\delta))$. Then (\ref{eigenvector_jump_in}) holds for $p$:
	\begin{align*}
		\|p\xi - b(p)\|_\vphi^2&=\|p\xi\|_\vphi^2 - \< \delta(p^*p), 1\otimes S_\vphi(\xi)\>_{\text{HS}}\\
			&= \|p\xi\|_\vphi^2 - \< \xi\otimes 1, \delta(p^*p)^\dagger\>_{\text{HS}}\\
			&= \|p\xi\|_\vphi^2 - \<\xi, b(p^*p)\>_{\text{HS}}.
	\end{align*}
Then, since $p^*p\in M^\vphi$, we can use (\ref{partial_derivative_bounds}) to obtain
	\begin{align*}
		\|p\xi - b(p)\|_\vphi^2&\leq \|p\xi\|_\vphi^2 + \|\xi\|_\vphi 3\|p^*p\| \|\xi\|_\vphi\\
			&\leq 4\|p\|^2\|\xi\|_\vphi^2.
	\end{align*}
Hence
	\[
		\|b(p)\|_\vphi \leq \|p\xi\|_\vphi + \|p\xi - b(p)\|_\vphi \leq 3 \|p\| \|\xi\|_\vphi
	\]
The final estimate then follows from
	\begin{align*}
		(\vphi\otimes \sigma_{-i}^\vphi)( \bar{\hat{\delta}}(p)) &= J_\vphi \Delta_\vphi^{1/2} (\sigma_i^\vphi\otimes \vphi)(\bar{\delta}(p^*))\\
			&= J_\vphi  (1\otimes\vphi)(\sigma_{i/2}^\vphi\otimes \sigma_{i/2}^\vphi)(\bar\delta(p^*))\\
			&= J_\vphi \mu^{1/2}  (1\otimes \vphi)(\bar\delta( \sigma_{i/2}^\vphi(p^*)))\\
			&= J_\vphi \mu^{1/2}\lambda^{1/2} (1\otimes \vphi)(\bar\delta( p^*)),
	\end{align*}
and the fact that $\mu^{1/2}\|\hat{\xi}\|_\vphi=\|\xi\|_\vphi$.
\end{proof}

\begin{cor}\label{boundedness_of_adjoint_formula}
Let $\delta\colon \dom(\delta)\to M_\infty\otimes M_\infty^{op}$ be a $\mu$-modular derivation for some $\mu>0$, with $\dom(\delta)$ generated by eigenoperators of $\sigma^\vphi$. Assume that the conjugate variable $\xi$ to $\delta$ exists. Then the closures of the densely defined maps $(1\otimes\varphi)\circ\hat\delta$ and $(\varphi\otimes\sigma_{-i}^\varphi)\circ\hat\delta$ on $L^2(M,\varphi)$ (both with domain $\dom(\delta)$) satisfy for any $\lambda\in \R_+^\times$ and $x\in \mc{E}_\lambda(M)$
	\begin{align*}
		\left\|\overline{(1\otimes \varphi)\circ\hat\delta}(x)\right\|_\varphi &\leq 3\|x\|\|\xi\|_\varphi ,\qquad \text{and} \\
		\|\overline{(\varphi\otimes \sigma_{-i}^\varphi)\circ\hat\delta}(x)\|_\varphi &\leq 3\lambda^{1/2}\|x\| \|\xi\|_\varphi.
	\end{align*}
Furthermore, for every $\lambda,\gamma\in \R_+^\times$ we have $\mc{E}_\lambda(M)\otimes \mc{E}_\gamma(M)^{op} \subset \dom{(\delta^*)}$.
\end{cor}
\begin{proof}
Denote the two above maps by $B$ and $C$, respectively. We first show $B$ and $C$ are closable by showing the $L^2$-dense set $\dom(\delta)$ lies in the domain of their adjoints. Given $p,x\in \dom(\delta)$ we have
	\begin{align*}
		|\< p, B(x)\>_\varphi |&= |\varphi(p^* (1\otimes \varphi)(\hat{\delta}(x)))|\\
			&= |\varphi\otimes\varphi^{op}( (p^*\otimes 1)\# \hat{\delta}(x)) |\\
			&= |\<p\otimes 1, \hat{\delta}(x)\>_\varphi|\\
			&\leq \|\hat{\delta}^*(p\otimes 1)\|_\varphi \|x\|_\varphi,
	\end{align*}
and hence $p\in \dom(B^*)$. Similarly, we have
	\begin{align*}
		|\<p, C(x)\>_\varphi | &= |\varphi( p^* (\varphi\otimes \sigma^\vphi_{-i})(\hat{\delta}(x)))|\\
			&= |\varphi( (\varphi\otimes 1)(\hat{\delta}(x))p^*)|\\
			&= |\varphi\otimes\varphi^{op}( (1\otimes p^*)\# \hat{\delta}(x))|\\
			&\leq \| \hat{\delta}^*(1\otimes p)\|_\varphi \|x\|_\varphi,
	\end{align*}
so that $p\in \dom(C^*)$. Thus $B$ and $C$ are closable, and we let $\bar{B}$ and $\bar{C}$ denote their closures.

Let $x\in \mc{E}_\lambda(M)$ for some $\lambda\in \R_+^\times$, then by Kaplansky's density theorem we can find a sequence $(p_n)_{n\in\N}\subset \dom(\delta)$ which converges to $x$ in $L^2(M,\varphi)$ and satisfies $\|p_n\|\leq \|x\|$ for each $n$. By replacing $p_n$ with $\mc{E}_\lambda(p_n)$ for each $n\in \N$, we may assume $p_n\in \mc{E}_\lambda(\dom(\delta))$. For $w\in \dom(B^*)$, using the final bounds in Proposition \ref{one_component_of_derivative_is_bounded} we have
	\begin{align*}
		|\<x, B^*(w)\>_\varphi| &= \lim_{n\to\infty} |\<p_n, B^*(w)\>_\varphi|\\
			&=\lim_{n\to\infty} |\< (1\otimes \varphi)(\hat{\delta}(p_n), w\>_\varphi|\\
			&\leq \limsup_{n\to \infty} 3\|p_n\| \|\xi\|_\varphi \|w\|_\varphi\\
			&\leq 3\|x\| \|\xi\|_\varphi \|w\|_\varphi,
	\end{align*}
which shows that $x$ is in the domain of $(B^*)^*=\bar{B}$ with $\|\bar{B}(x)\|_\varphi\leq 3\|x\|\|\xi\|_\varphi$. Similarly, we have $x\in \dom(\bar{C})$ with $\|\bar{C}(x)\|_\varphi \leq 3\lambda^{1/2} \|x\| \|\xi\|_\varphi$. 

Now, for $\lambda,\gamma\in \R_+^\times$ let $a\in \mc{E}_\lambda(M)$ and $b\in \mc{E}_\gamma(M)$. As above, we let $(p_n)_{n\in\N}\subset \mc{E}_\lambda(\dom(\delta))$ and $(q_n)_{n\in\N}\subset \mc{E}_\gamma(\dom(\delta))$ be sequences converging to $a$ and $b$ in $L^2(M,\vphi)$, respectively, and satisfying $\|p_n\|\leq \|a\|$ and $\|q_n\|\leq \|b\|$ for all $n\in \N$. Since $S_\vphi$ is bounded on $E_\gamma$, we note that $(q_n^*)_{n\in\N}$ converges to $b^*$ in $L^2(M,\vphi)$. Consequently, $(p_n\otimes q_n)_{n\in\N}$ converges to $a\otimes b$ in $L^2(M\bar\otimes M^{op})$. Using the formula in Lemma \ref{adjoint_formula}, we have for any $w\in\dom(\bar\delta)$
	\begin{align*}
		\<a\otimes b, \bar\delta(w)\>_{\text{HS}}&=\lim_{n\to\infty} \< p_n\otimes q_n , \bar\delta(w)\>_{\text{HS}}\\
			&=\lim_{n\to\infty} \< p_n \xi \sigma_{-i}^\vphi(q_n) - B(p_n)\sigma_{-i}^\vphi(q_n) - p_n C(q_n), w\>_\vphi\\
			&\leq \limsup_{n\to\infty} \left[ \|\sigma_{-i/2}^\vphi(q_n)\| \|p_n\xi\|_\vphi + \|\sigma_{-i/2}^\vphi(q_n)\| \|B(p_n)\|_\vphi + \|p_n\| \|C(q_n)\|_\vphi\right] \|w\|_\vphi\\
			&\leq \limsup_{n\to\infty} 7\gamma^{1/2} \|q_n\| \|p_n\| \|\xi\|_\vphi \|w\|_\vphi\leq 7\gamma^{1/2} \|b\| \|a\| \|\xi\|_\vphi \|w\|_\vphi,
	\end{align*}
Thus $a\otimes b\in \dom(\delta^*)$.
\end{proof}
%Now, let $x_1\in M$ and $x_2\in E_\lambda$ for some $\lambda>0$. Since $x_1\in \dom{(\bar{B}_y)}$ and $x_2\in \dom{(\bar{C}_y)}$, there exists sequences $(p_n)_{n\geq 1}\subset \dom{(B_y)}=\P$ and $(q_m)_{m\geq 1}\subset \dom{(C_y)}=\P$ converging to $x_1$ and $x_2$ in $L^2(M,\varphi)$, respectively, such that $B_y(p_n)\to \bar{B}_y(x_2)$ and $C_y(q_m)\to \bar{C}_y(x_2)$ in $L^2(M,\varphi)$. In particular, for $z\in \dom{(\delta_{y^*})}$ and using Lemma \ref{adjoint_formula} we have
%	\begin{align*}
%		\<x_1\otimes x_2, \delta_{y^*}(z)\>_{\text{HS}} &= \lim_{n,m\to\infty} \< p_n\otimes q_m, \delta_{y^*}(z)\>_{\text{HS}}\\
%			&= \lim_{n,m\to\infty} \< p_n\xi_{y^*} \sigma_{-i}^\varphi(q_m) - B_y(p_n)\sigma_{-i}^\varphi(q_m) - p_n C_y(q_m), z\>_\varphi\\
%			&= \< x_1\xi_{y^*} \sigma_{-i}^\varphi(x_2) - \bar{B}_y(x_1)\sigma_{-i}^\varphi(x_2) - x_1 \bar{C}_y(x_2), z\>_\varphi.
%	\end{align*}

%%%%%%%%%%%%%%%%%%%%%%%%%%%%%%%%%%%%%%%%%%%%%%%%%%%%%%%%%%%%%%%%%%%%%%%%%%%%%
\subsection{Contraction resolvent arising as deformations of $\delta^*\bar\delta$}\label{cont_resolv}

Let $\delta\colon L^2(M,\vphi)\to L^2(M\bar\otimes M^{op})$ be a closable $\mu$-modular derivation, for some $\mu>0$, with $\dom(\delta)$ generated by eigenoperators of $\sigma^\vphi$. One of the key steps in the proofs of our main theorems will be to show that certain central elements $z$ lie in $\ker(\bar\delta)$, and so it will be useful to have a way to approximate such $z$ with elements from $\dom(\bar\delta)$. Towards this end, in this subsection we replicate in $M^\vphi$ the analysis of contraction resolvents given in \cite[Section 1]{Dab10}. The lemma we prove shows that we can in fact approximate such $z$ with similarly central elements in $\dom(\bar\delta)$. We refer the reader to \cite[Chapter I]{MR92} for a more general treatment of contraction resolvents.

We consider $ L:=\delta^*\bar{\delta}$, a self-adjoint operator with dense domain. For each $t>0$ define $T_t:=e^{- t L}$. Then $\{T_t\}_{t>0}$ is a \emph{strongly continuous contraction semigroup} with infinitesimal generator $- L$. For each $\alpha>0$ define $\eta_\alpha=\alpha(\alpha+L)^{-1}$. Proposition 1.10 of \cite{MR92} implies
	\begin{align*}
		\eta_\alpha = \alpha\int_0^\infty e^{-\alpha s} T_s\ ds,
	\end{align*}
and that $\{\frac{1}{\alpha}\eta_\alpha\}_{\alpha>0}$ is a \emph{strongly continuous contraction resolvent} (\emph{cf.} \cite[Definition 1.4]{MR92}). In particular:
	\begin{enumerate}
		\item For all $\alpha>0$, $\eta_\alpha$ is a $\|\cdot\|_\varphi$-contraction; and
		
		\item $\eta_\alpha$ converges strongly to the identity as $\alpha\to \infty$ .
	\end{enumerate}
Moreover, $\text{Range}(\eta_\alpha)=\dom(L)\subset \dom(\bar{\delta})$ (\emph{cf.} the proof of \cite[Proposition 1.5]{MR92}).

Define $\zeta_\alpha:=(\eta_\alpha)^{1/2}$. Then we have by \cite[Lemma 3.2]{Pet09}
	\begin{align}\label{zeta_formula}
		\zeta_\alpha = \frac{1}{\pi}\int_0^\infty \frac{1}{\sqrt{t}} \eta_\alpha (t+ \eta_\alpha)^{-1}\ dt = \frac{1}{\pi}\int_0^\infty \frac{1}{\sqrt{t}(1+t)} \eta_{\alpha(1+t)/t}\ dt.
	\end{align}
Furthermore, $\text{Range}(\zeta_\alpha)=\dom(L^{1/2})=\dom(\bar{\delta})$, where the latter equality follows from $\|L^{1/2}(x)\|_\varphi = \|\bar{\delta}(x)\|_\varphi$ for all $x\in \dom(L^{1/2})$. Consequently, $\bar{\delta}\circ\zeta_\alpha$ defines a bounded operator. For $x\in L^2(M,\varphi)$ we have
	\begin{align*}
		\lim_{\alpha\to\infty} \| x - \zeta_\alpha(x)\|_\varphi &= \lim_{\alpha\to\infty}\left\| \frac{1}{\pi} \int_0^\infty \frac{1}{\sqrt{t}(1+t)}\left( x - \eta_{\alpha(1+t)/t}(x)\right)\ dt\right\|_\varphi\\
			&\leq \lim_{\alpha\to\infty}\frac{1}{\pi} \int_0^\infty \frac{1}{\sqrt{t}(1+t)} \left\| x - \eta_{\alpha(1+t)/t}(x)\right\|_\varphi\ dt.
	\end{align*}
Using property (1) above, we see that the integrand is dominated by $\frac{1}{\sqrt{t}(1+t)} 2\|x\|_\varphi$ and so the dominated convergence theorem implies
	\begin{align}\label{zeta_to_identity}
		\lim_{\alpha\to\infty} \| x - \zeta_\alpha(x)\|_\varphi=0.
	\end{align}

Recalling that $L^2(M,\vphi)$ admits bounded left \emph{and right} actions of $M^\vphi$, we have the following lemma.

\begin{lem}\label{comm_for_cont_res}
For $x\in \ker(\delta)\cap\ker(\hat\delta)\cap M^\vphi$ and $\xi\in L^2(M,\vphi)$
	\[
		\zeta_\alpha(\xi\cdot x)=\zeta_\alpha(\xi)\cdot x\qquad\text{ and }\qquad \zeta_\alpha(x\cdot \xi)=x\cdot\zeta_\alpha(\xi).
	\]
\end{lem}
\begin{proof}
First note that $\delta(x)=0$ implies $ L(x)=0$, and so using $\hat\delta(x)=0$, Lemma \ref{adjoint_formula}, and Lemma \ref{extension_of_Leibniz_rule} we have $ L(\xi\cdot x)= L(\xi)\cdot x$ and $ L(x\cdot \xi) =x\cdot L(\xi)$ for all $\xi\in \dom(L)$. We next claim $\eta_\alpha(\xi\cdot x)=\eta_\alpha(\xi)\cdot x$ and $\eta_\alpha(x\cdot\xi)=x\cdot \eta_\alpha(\xi)$ for all $\xi\in L^2(M,\vphi)$. Suppose $\eta= \eta_\alpha(\xi\cdot x)$ and $\eta'=\eta_\alpha(\xi)$, which we note are contained in $\dom(L)\subset\dom(\bar\delta)$. Then
	\begin{align*}
		\frac{1}{\alpha}(\alpha+ L)(\eta - \eta'\cdot x) = \xi\cdot x  - \eta'\cdot x - \frac{1}{\alpha} L(\eta')\cdot x = \xi\cdot x - \frac{1}{\alpha} (\alpha +  L)(\eta')\cdot x = \xi\cdot x - \xi\cdot x =0,
	\end{align*}
hence $\eta_\alpha(\xi\cdot x)=\eta=\eta'\cdot x=\eta_\alpha(\xi)\cdot x$. Similarly, $\eta_\alpha(x\cdot\xi) = x\cdot \eta_\alpha(\xi)$. Finally, using (\ref{zeta_formula}) we have $\zeta_\alpha(\xi\cdot x)= \zeta_\alpha(\xi)\cdot x$ and $\zeta_\alpha(x\cdot\xi) = x\cdot \zeta_\alpha(\xi)$.
\end{proof}

%%%%%%%%%%%%%%%%%%%%%%%%%%%%%%%%%%%%%%%%%%%%%%%%%%%%%%%%%%%%%%%%%%%%%%%%%%%%%%%%%%%%
%                      Diffuse elements iin the centralizer                         %
%%%%%%%%%%%%%%%%%%%%%%%%%%%%%%%%%%%%%%%%%%%%%%%%%%%%%%%%%%%%%%%%%%%%%%%%%%%%%%%%%%%%

\section{Diffuse Elements in the Centralizer}\label{diffuse}

In this section we will show that there is an abundance of diffuse elements in the centralizer $M^\vphi$. Recall our notation from Section \ref{Generators}:
	\[
		\P=\C\<G\>=\C\<x_1,\ldots, x_n\>,
	\]
where $G=G^*$ consists of eigenoperators of $\sigma^\vphi$ and $x_1,\ldots, x_n$ are generators of the form in Proposition \ref{equivalent_forms_of_generators}.(ii). Then, more precisely, we will give a condition for when elements in $\C\<G\>\cap M^\vphi$ are diffuse. We begin by replicating \cite[Theorem 4.4]{CS05} in our non-tracial context.

%%%%%%%%%%%%%%%%%%%%%%%%%%%%%%%%%%%%%%%%%%%%%%%%%%%%%%%%%%%%%%%%%%%%%%%%%%%%%%%%%%%%
\subsection{An $L^2$-homology estimate}\label{technical_estimate_section}

Let $P_1\in \mc{B}(L^2(M,\varphi))$ denote the projection onto the cyclic vector. We let $\Psi\colon M\otimes M^{op}\to \text{FR}(L^2(M,\varphi))$ be the isometry into the finite-rank operators defined by
	\begin{align*}
		\Psi(a\otimes b)\xi = aP_1b\xi,\qquad a,b\in M,\ \xi\in L^2(M,\varphi).
	\end{align*}

For the matrix $A\in M_n(\C)$ as in Proposition \ref{equivalent_forms_of_generators}, let $\Gamma(\R^n, A^{it})''$ be the free Araki-Woods factor corresponding to the orthogonal group $\{A^{it}\}_{t\in\R}$ (\emph{cf.} \cite{Shl97}). It is generated by quasi-free semicircular elements $s_1,\ldots, s_n$ and admits a free quasi-free state $\varphi_A$ satisfying for each $j=1,\ldots,n$
	\begin{align*}
		\sigma_{z}^\varphi(s_j)=\sum_{k=1}^n [A^{iz}]_{jk}s_k.
	\end{align*}
The covariance of the system is given by $\varphi_A(s_js_k)=\left[\frac{2}{1+A}\right]_{kj}$. 

Let $\H_A=L^2(\Gamma(\R^n,A^{it})'', \varphi_A)$, then $\H_A$ can be identified with a Fock space on which each $s_j=\l(e_j)+\l(e_j)^*$ is a sum of left creation and left annihilation operators for an orthonormal basis $\{e_1,\ldots,e_n\}$ of $\R^n$ obliquely embedded in $\C^n$. Letting $r_j:=r(e_j)$, $j=1,\ldots,n$, be the corresponding right creation operators, we have that
	\begin{align*}
		[s_j,r_k] = \<e_j,e_k\>_{H_A} P_1=\<s_j,r_k\>_{\varphi_A} P_1= \left[\frac{2}{1+A}\right]_{kj} P_1.
	\end{align*}
We say that $\{r_1,\ldots, r_n\}$ is a \emph{quasi-dual system to $\{s_1,\ldots,s_n\}$ with covariance $\frac{2}{1+A}$}.

We consider the free product $(\M,\theta)=(M,\vphi)*(\Gamma(\R^n, A^{it})'',\vphi_A)$. Then, by using the right regular representation for $r_1,\ldots, r_n$ on $L^2(\M,\theta)=(L^2(M,\varphi),1)* (\H_A,\Omega)$, we can realize these operators in $\mc{B}(L^2(\M,\theta))$, where they satisfy $[x,r_k]=0$ for all $x\in M$.

\begin{rem}\label{conjugate_variables_are_same_as_dima}
Observe that $\Gamma(\R^n,A^{it})'' \cong \Gamma( M_{s.a.}\subset M)''$ since $s_1,\ldots, s_n$ have the same covariance as the generators $x_1,\ldots, x_n$ and also vary under the modular operator in the same way. Consequently, the maps $\P\ni p\mapsto \delta_j(p)\# s_j$, $j=1,\ldots, n$, are exactly the derivations considered in \cite{Shl03}, in which  Shlyakhtenko defined the conjugate variable to this derivation as an element $\xi\in L^2(M,\varphi)$ satisfying
	\begin{align*}
		\<\xi, p\>_\varphi = \<s_j, \delta_j(p)\# s_j\>_\theta\qquad \forall p\in \P,
	\end{align*}
provided it exists. The freeness condition implies that for $a,b\in M$ we have
	\begin{align*}
		\<s_j, as_jb\>_\theta=\theta(s_jas_jb) = \varphi(a)\varphi(b)\varphi(s_j^2) = \<1\otimes 1,a\otimes b\>_{\text{HS}}.
	\end{align*}
Thus if $J_\varphi(x_j\colon \C[x_k\colon k\neq j])$ exists then
	\begin{align*}
		\<J_\varphi(x_j\colon \C[x_k\colon k\neq j]), p\>_\varphi = \<1\otimes 1,\delta_j(p)\>_{\text{HS}} = \<s_j, \delta_j(p)\# s_j\>_\theta.
	\end{align*}
Hence $J_\varphi(x_j\colon \C[x_k\colon k\neq j])$ is also a conjugate variable in the sense of \cite{Shl03}. Moreover, since $\|x_j\|_\vphi=1$ for each $j=1,\ldots, n$, this also implies $\Phi_\varphi^*(x_1,\ldots, x_n)$ equals the free Fisher information from \cite[Definition 2.5]{Shl03}.
\end{rem}

For each $\epsilon>0$ and $j=1,\ldots, n$ define $x_j(\epsilon)=x_j+\sqrt{\epsilon} s_j$ and let $M_\epsilon= W^*(x_j(\epsilon)\colon j=1,\ldots, n)\subset \M$. Since $x_j$ and $s_j$ vary linearly in the same way under the action of $\sigma^\theta$, it is easy to see that $\forall \epsilon>0$ $M_\epsilon$ is globally invariant under $\sigma^\theta$. Therefore, by \cite[Theorem IX.4.2]{Tak03}, there is a conditional expectation $\mc{E}_\epsilon\colon \M\to M_\epsilon$. For each $\epsilon>0$ let $p_\epsilon$ denote the orthogonal projection from $L^2(\M,\theta)$ to $L^2(M_\epsilon, \theta)$ so that for $x\in \M$ we have $p_\epsilon xp_\epsilon = \mc{E}_\epsilon(x) p_\epsilon$. Define for each $j=1,\ldots, n$ $r_j(\epsilon)= p_\epsilon\frac{1}{\sqrt{\epsilon}} r_j p_\epsilon$ so that
	\begin{align*}
		[x_j(\epsilon),r_k(\epsilon)] = p_\epsilon \frac{1}{\sqrt{\epsilon}} [x_j+\sqrt{\epsilon} s_j,r_k]p_\epsilon = p_\epsilon[s_j,r_k]p_\epsilon = \left[\frac{2}{1+A}\right]_{kj} P_1,
	\end{align*}
where we have used the fact that $M_\epsilon$ is unital to conclude $P_1=p_\epsilon P_1 p_\epsilon$. Hence $\{r_1(\epsilon),\ldots, r_n(\epsilon)\}$ is a quasi-dual system to $\{x_1(\epsilon),\ldots, x_n(\epsilon)\}$ with covariance $\frac{2}{1+A}$.

We will associate the following quantity to $x_1,\ldots, x_n$, which can be thought of as a type of \emph{free entropy dimension}:
	\[
		d_A^\star(x_1,\ldots, x_n):= n - \liminf_{\epsilon\to 0} \epsilon \sum_{j=1}^n \left\|J_\theta^A(x_j(\epsilon)\colon \C[x_k(\epsilon)\colon k\neq j])\right\|_\theta^2.
	\]
This is the non-tracial analogue of the quantity $\delta^\star$ that appears in the tracial case as the result of formally applying L'H\^opital's rule to the free entropy dimension $\delta^*$ (\emph{cf.} \cite[Section 4.1]{CS05}). We note that we present this quantity merely as a convenient notation. Since our only examples are those where the $\liminf$ in the definition of $d_A^\star$ is zero, and since we do not have a corresponding non-tracial non-microstates free entropy to compare this quantity with, it is unclear if this is really the correct definition for a non-tracial free entropy dimension. In any case, the following lemma, which is proved using arguments from \cite[Propositions 3.6 and 3.7]{VoiV} adapted to the present context, tells us that $d_A^\star(x_1,\ldots, x_n)=n$ in the situations we will consider.

\begin{lem}\label{conjugate_variables_for_reguaralized_generators}
For each $\epsilon>0$ and $j\in\{1,\ldots, n\}$,
	\begin{align*}
		J_\theta^A(x_j(\epsilon)\colon \C[x_k(\epsilon)\colon k\neq j]) = \frac{1}{\sqrt{\epsilon}} \mc{E}_\epsilon(s_j),
	\end{align*}
and if $J_\varphi^A(x_j\colon \C[x_k\colon k\neq j])$ exists then
	\begin{align*}
		J_\theta^A(x_j(\epsilon)\colon \C[x_k(\epsilon)\colon k\neq j])=p_\epsilon J_\varphi^A(x_j\colon \C[x_k\colon k\neq j]).
	\end{align*}
In particular, when $J_\varphi^A(x_j\colon \C[x_k\colon k\neq j])$ exists for each $j=1,\ldots, n$, we have $d_A^\star(x_1,\ldots, x_n)=n$.
\end{lem}
\begin{proof} We first note that for any $p\in \C\<x_1(\epsilon),\ldots, x_n(\epsilon)\>$ and any $j\in\{1,\ldots, n\}$
	\begin{align*}
		\partial_{x_j(\epsilon)}(p) = \frac{1}{\sqrt{\epsilon}} \partial_{s_j}(p)
	\end{align*}
Consequently,
	\begin{align*}
		J_\theta^A(x_j(\epsilon)\colon \C[x_k(\epsilon)\colon k\neq j]) = \frac{1}{\sqrt{\epsilon}} p_\epsilon J_\theta^A (s_j \colon \C[x_k(\epsilon)\colon k\neq j]).
	\end{align*}
Next, we claim that
	\begin{align*}
		J_\theta^A (s_j\colon \C[x_k(\epsilon)\colon k\neq j]) = J_{\varphi_A}^A(s_j\colon \C[s_k\colon k\neq j]) = s_j.
	\end{align*}
The last equality follows by \cite[Proposition 2.2]{Nel15}. Thus to show the first equality we must demonstrate
	\begin{align*}
		\<s_j, p\>_\theta = \< 1\otimes 1, \partial_{s_j} p \>_{\theta\otimes\theta^{op}}\qquad \forall p\in \C\<x_1(\epsilon),\ldots, x_k(\epsilon)\>.
	\end{align*}
By expanding each $x_k(\epsilon)=x_k+ \sqrt{\epsilon} s_k$, it suffices to show
	\begin{align}\label{induction_step}
		\<s_j, c_0p_1c_1\cdots p_mc_m\>_\theta = \sum_{k=1}^m \theta\otimes\theta^{op}( c_0p_1c_1\cdots c_	{k-1} \partial_{s_j}(p_k) c_k\cdots p_mc_m),
	\end{align}
where $c_0,\ldots, c_m\in \P$ and $p_1,\ldots, p_m\in \C\<s_1,\ldots, s_n\>$. We proceed by induction on $m$. Recalling that $\theta(s_j)=\varphi_A(s_j)=0$ and using free independence (twice) we have
	\begin{align*}
		\<s_j, c_0p_1c_1\>_\theta &= \varphi(c_0)\varphi(c_1)\varphi_A(s_j p_1)\\
						&= \varphi(c_0)\varphi(c_1) \varphi_A\otimes\varphi_A^{op}(\partial_{s_j}(p_1))\\
						&= \theta\otimes\theta^{op}( c_0\partial_{s_j}(p_1) c_1).
	\end{align*}
Let $m\geq 2$ and suppose (\ref{induction_step}) holds for elements of the form $c_0p_1c_1\cdots p_\l c_\l$ with $\l< m$. For $a\in \M$, write $\mathring{a}=a-\theta(a)$. Upon writing $p_k=\mathring{p}_k+\theta(p_k)$ for each $k=1,\ldots, m$ and $c_k=\mathring{c}_k+\theta(c_k)$ for each $k=1,\ldots,m-1$ and expanding, it suffices by the induction hypothesis to consider the case when $p_1,\ldots, p_m$ and $c_1,\ldots, c_{m-1}$ are centered with respect to $\theta$. We then have by free independence
	\begin{align*}
		\<s_j,c_0p_1\cdots p_mc_m\>_\theta = \theta(s_j\mathring{c}_0p_1\cdots p_mc_m) + \theta(c_0)\theta(s_jp_1\cdots p_mc_m) = 0.
	\end{align*}
On the other hand, if we write $\partial_{s_j}(p_k) = \sum_{\l} a_\l^k\otimes b_\l^k$ for each $k=1,\ldots, m$ then
	\begin{align*}
		 \sum_{k=1}^m\sum_\l \theta( c_0p_1c_1\cdots c_{k-1}a_\l^k) \theta(b_\l^k c_k\cdots p_mc_m).
	\end{align*}
The factors $\theta( c_0p_1c_1\cdots c_{k-1}a_\l^k)$ vanish if $k\geq 2$, while the factors $\theta(b_\l^k c_k\cdots p_mc_m)$ vanish if $k\leq m-1$. Since $m\geq 2$, at least one of these conditions always holds and we have proved the claim. 

That
	\begin{align*}
		J_\theta^A(x_j(\epsilon)\colon \C[x_k(\epsilon)\colon k\neq j])=p_\epsilon J_\varphi^A(x_j\colon \C[x_k\colon k\neq j]),
	\end{align*}
holds via a similar argument. Hence, when $J_\varphi^A(x_j\colon \C[x_k\colon k\neq j])$ exists for each $j=1,\ldots, n$, we have
	\[
		\liminf_{\epsilon\to 0} \epsilon \sum_{j=1}^n \left\| J_\theta^A(x_j(\epsilon)\colon \C[x_k(\epsilon)\colon k\neq j])\right\|_\theta^2 \leq \lim_{\epsilon\to0} \epsilon \sum_{j=1}^n \left\|J_\varphi^A(x_j\colon \C[x_k\colon k\neq j])\right\|_\vphi^2=0.
	\]
Thus $d_A^\star(x_1,\ldots, x_n)=n$.
\end{proof}

We let $1$ denote the cyclic vector in $L^2(\M,\theta)$. Define $S_\theta$ to be the closure of the operator $y1\mapsto y^*1$, densely defined on $\M 1\subset L^2(\M,\theta)$, with polar decomposition $S_\theta = J_\theta \Delta_\theta^{1/2}$. The restrictions of $S_\theta$ to the subspaces $L^2(M,\varphi)$ and $\H_A$ yield $S_\varphi=J_\varphi \Delta_\varphi^{1/2}$ and $S_{\varphi_A} = J_{\varphi_A} \Delta_{\varphi_A}^{1/2}$, respectively.

For $T\in \B(L^2(\M,\theta))$ we define
	\begin{align*}
		\rho(T):=J_\theta T^* J_\theta,
	\end{align*}
which we recall from Remark \ref{right_action} is the usual right action of $\mc{M}$ on $L^2(\mc{M},\theta)$, but differs from the one we have been considering thus far. An easy computation shows that $\rho$ is an isometry on $\Psi( M\otimes M^{op})$ with respect to the Hilbert-Schmidt norm. 
%Proof too simple for inclusion
%	\begin{align*}
%		\< \rho(aP_1b), \rho(cP_1d)\>_{\text{HS}} &=\sum_{j=1}^\infty \<J_\varphi b^* P_1 a^* J_\varphi e_j, J_\varphi d^* P_1 c^* J_\varphi e_j\>_\varphi\\
%			&= \sum_{j=1}^\infty \< d^* P_1 c^* J_\varphi e_j, b^* P_1 a^* J_\varphi e_j\>_\varphi\\
%			&= \sum_{j=1}^\infty \<c^* J_\varphi e_j, 1\>_\varphi \<1, a^* J_\varphi e_j\>_\varphi \< d^*, b^*\>_\varphi\\
%			&=\sum_{j=1}^\infty \< J_\varphi c, e_j\>_\varphi \< e_j, J_\varphi a\>_\varphi \<d^*, b^*\>_\varphi\\
%			&= \< J_\varphi c, J_\varphi a\>_\varphi \<d^*, b^*\>_\varphi\\
%			&= \< aP_1b, cP_1d\>_{\text{HS}}
%	\end{align*}
We also note that for $a,b\in M$,  $f,g\in M_\infty$, and $\xi\in L^2(M,\vphi)$
	\begin{align*}
		\rho(f)aP_1b\rho(g)\xi &= \<1, b\rho(g)\xi\>_\varphi \rho(f)a 1\\
			&= \<g^*J_\varphi b \xi, 1\>_\varphi a J_\varphi f^* 1\\
			&= \< J_\varphi g1, b\xi\>_\varphi a \sigma_{-i/2}^\varphi(f)1\\
			&= \<\sigma_{-i/2}^\varphi(g^*)1, b\xi\>_\varphi a\sigma_{-i/2}^\varphi(f) 1\\
			&= a\sigma_{-i/2}^\varphi(f) P_1 \sigma_{i/2}^\varphi(g) b \xi,
	\end{align*}
thus we define the actions
	\begin{align}\label{rho_action}
		\rho(f)\cdot (a\otimes b)\cdot \rho(g):= a\sigma_{-i/2}^\varphi(f)\otimes \sigma_{i/2}^\varphi(g) b.
	\end{align}

\begin{lem}\label{theta_on_conjugations_of_r}
For $x\in M$ and $r_j, s_j$ as above, we have $\rho(r_j)x1=0$. If $a,b\in M_\infty$ we have
	\begin{align*}
		\rho(a\rho(r_j) b)1 =\sigma_{-i/2}^\varphi(a)s_j\sigma_{-i/2}^\varphi(b) 1.
	\end{align*}
\end{lem}
\begin{proof}
For $x\in M$, we have $J_\theta x 1 = J_\vphi x 1 = \Delta_\vphi^{1/2} x^* 1 \in L^2(M,\vphi)$. Therefore
	\[
		\rho(r_j) x1 = J_\theta r_j^* (\Delta_\vphi^{1/2} x^* 1) = 0.
	\]
Next, we note the following identity for $c\in M_\infty$:
	\begin{align*}
		J_\theta c 1 = \Delta_\theta^{1/2} S_\theta c1 = \Delta_\varphi^{1/2} c^*1 = \sigma_{-i/2}^\varphi(c^*) 1.
	\end{align*}
We compute
	\begin{align*}
		\rho( a \rho(r_j) b)1 &= J_\theta b^* J_\theta r_j J_\theta a^* J_\theta 1 = J_\theta b^* J_\theta r_j \sigma_{-i/2}^\varphi(a) 1\\
			&= J_\theta b^* J_\theta \left[\sigma_{-i/2}^\varphi(a)\otimes e_j\right]= J_\theta b^* \left[\Delta_{\varphi_A}^{1/2}e_j\otimes a^*\right]\\
			&= J_\theta \left[b^*\otimes \Delta_{\varphi_A}^{1/2}e_j\otimes a^*\right]= \sigma_{-i/2}^\varphi(a) \otimes e_j\otimes \sigma_{-i/2}^\varphi(b)\\
			&= \sigma_{-i/2}^\varphi(a)s_j\sigma_{-i/2}^\varphi(b) 1,
	\end{align*}
as claimed.
\end{proof}

\begin{lem}\label{P1_on_commutator}
For all $x\in \mc{M}$ and $T\in \B(L^2(\M,\theta))$ we have
	\begin{align*}
		\text{Tr}(P_1[T_1,\rho(x)]) = \<\rho(x^*)1, \left(\Delta_\theta^{-1/2} \rho(T) - T\right)1\>_\theta
	\end{align*}
\end{lem}
\begin{proof}
We simply compute
	\begin{align*}
	 	\text{Tr}(P_1[T,\rho(x)]) &= \<1, T\rho(x)1 - \rho(x)T1\>_\theta \\
	 		&= \<T^*1,\rho(x)1\>_\theta - \<\rho(x)^*1, T1\>_\theta\\
	 		&= \<x^* 1, J_\theta T^* 1\>_\theta - \< \rho(x^*)1, T1\>_\theta\\
	 		&= \<\Delta_\theta^{-1/2} J_\theta x 1, J_\theta T^* 1\>_\theta - \< \rho(x^*)1, T1\>_\theta\\
	 		&= \<\rho(x^*)1, (\Delta_\theta^{-1/2} \rho(T) - T)1\>_\theta.
	\end{align*}
\end{proof}

\begin{lem}\label{approx_in_M_epsilon}
Let $\delta>0$. Given $T=\sum_\ell a_\ell P_1b_\ell \in \Psi(M\otimes M^{op})$, there exists $\epsilon_0>0$ so that for each $\epsilon\in (0,\epsilon_0)$ we can find $x_\ell, y_\ell\in M_\epsilon$ such that if
	\begin{align*}
		T(\epsilon):=\sum_\ell x_\ell P_1 y_\ell,
	\end{align*}
then
	\begin{align*}
		\| T - T(\epsilon)\|_{\text{HS}} \leq \| T - T(\epsilon)\|_1 <\delta.
	\end{align*}
\end{lem}
\begin{proof}
This follows from exactly the same argument as in \cite[Lemma 4.2]{CS05}, except we must note that
	\begin{align*}
		\| aP_1 b\|_1 &= \sup_{\|T\|_\infty =1} |\<T, aP_1b\>_{Tr}|\\
				&= \sup_{\|T\|_\infty=1} | Tr(P_1 bT^*aP_1)|\\
				&= \sup_{\|T\|_\infty=1} |\<Tb^*1,a1\>_\varphi| = \|a\|_\varphi \|b^*\|_\varphi.
	\end{align*}
So we must approximate $b^*$ (rather than $b$) by polynomials in the $\|\cdot\|_\varphi$-norm.
\end{proof}

Note that for $a,b\in M_\epsilon$ and $\xi\in L^2(\M,\theta)$ we have
	\begin{align*}
		aP_1bp_\epsilon \xi = \<b^*1, p_\epsilon \xi\>_\theta a1 = \<b^*1, \xi\>_\theta p_\epsilon a 1 = p_\epsilon a P_1b \xi,
	\end{align*}
so that $[T(\epsilon),p_\epsilon]=0$ for $T(\epsilon)\in \Psi(M_\epsilon\otimes M_\epsilon^{op})$ as in the previous lemma.

\begin{prop}\label{orthogonal_to_P_1}
With the notation as above, suppose that $d_A^\star(x_1,\ldots, x_n)=n$. Then any set $\{T_{jk}\}_{j,k=1}^{n}\subset \Psi(M_\infty\otimes M_\infty^{op})$ such that $\sum_{j=1}^n [T_{jk},\rho(x_j)]=0$ for each $k=1,\ldots, n$ satisfies
	\begin{align*}
		\sum_{j,k=1}^n \< P_1, T_{jk} \left[\frac{2}{1+A}\right]_{kj}\>_{\text{HS}}=0.
	\end{align*}
\end{prop}
\begin{proof}
For each $1\leq j,k\leq n$ we can write
	\begin{align*}
		T_{jk} = \sum_\l a_\l^{jk} P_1 b_\l^{jk},
	\end{align*}
for some $a_\l^{ij}, b_\l^{ij}\in M_\infty$. We compute
	\begin{align*}
		\sum_{j,k=1}^n \<P_1, \left[\frac{2}{1+A}\right]_{kj} T_{jk}\>_{\text{HS}} &= \sum_{j,k=1}^n \Tr\left( T_{jk} \rho\left(\left[\frac{2}{1+A}\right]_{kj} P_1\right)\right)\\
			&= \sum_{j,k=1}^n \Tr( T_{jk} \rho([x_j(\epsilon),r_k(\epsilon)]))\\
			&= \sum_{j,k=1}^n \Tr( [\rho(x_j(\epsilon)), T_{jk}] \rho(r_k(\epsilon)) )\\
			&= \sum_{j,k=1}^n \Tr( [\rho(\sqrt{\epsilon} s_j), T_{jk}] \rho(r_k(\epsilon)) )\\
			&= \sum_{j,k=1}^n \Tr( T_{jk}[\sqrt{\epsilon}\rho(r_k(\epsilon)), \rho(s_j)]).
	\end{align*}
Let $\delta>0$, then by Lemma \ref{approx_in_M_epsilon} there exists $\epsilon_0>0$ such that if $\epsilon<\epsilon_0$ then for each $T_{ij}$ we can find $T_{ij}(\epsilon)\in \Psi(M_\epsilon\otimes M_\epsilon^{op})$ satisfying $\|T_{ij} - T_{ij}(\epsilon)\|_1\leq \delta$. Then since 
	\begin{align*}
		\|[\sqrt{\epsilon}\rho(r_k(\epsilon)), \rho(s_j)]\| = \left\| \sqrt{\epsilon} J_\theta [p_\epsilon \frac{1}{\sqrt{\epsilon}} r_j^* p_\epsilon, s_j] J_\theta \right\| \leq 2\|r_j^*\| \|s_j\|\leq 4,
	\end{align*}
we have 
	\begin{align*}
		\left| \sum_{j,k=1}^n \<P_1, \left[\frac{2}{1+A}\right]_{kj} T_{jk}\>_{\text{HS}}\right| \leq 4n^2\delta + \left|\sum_{j,k=1}^n  \Tr(T_{jk}(\epsilon)[\sqrt{\epsilon}\rho(r_k(\epsilon)), \rho(s_j)])\right|.
	\end{align*}
We focus on the latter term. Note that $[p_\epsilon, J_\theta]=0$ and denote $\xi_j(\epsilon)=\frac{1}{\sqrt{\epsilon}} p_\epsilon s_j p_\epsilon$. We have
	\begin{align*}
		\sum_{j,k=1}^n  \Tr(T_{jk}(\epsilon)[\sqrt{\epsilon}\rho(r_k(\epsilon)), \rho(s_j)]) &= \sum_{j,k=1}^n  \Tr(T_{jk}(\epsilon)[\rho(r_k), \rho(\sqrt{\epsilon}\xi_j(\epsilon))]).
	\end{align*}
Using $\|\sqrt{\epsilon}\xi_j(\epsilon)\|\leq \|s_j\|\leq 2$ and switching back to $T_{ij}$ from $T_{ij}(\epsilon)$ we have
	\begin{align*}
		\left| \sum_{j,k=1}^n \<P_1, \left[\frac{2}{1+A}\right]_{kj} T_{jk}\>_{\text{HS}}\right| \leq 8n^2\delta + \left|\sum_{j,k=1}^n  \Tr(T_{jk}[\rho(r_k), \rho(\sqrt{\epsilon}\xi_j(\epsilon))])\right|.
	\end{align*}
Now, we have that $J_\theta x J_\theta \in \M'\cap \B(L^2(\mc{M},\theta))$ for $x\in \M$ so $a_\l^{jk}$ and $b_\l^{jk}$ commute with $\rho(\sqrt{\epsilon}\xi_j(\epsilon))$. Hence considering the last term in the above inequality we have
	\begin{align*}
		\sum_{j,k=1}^n  \Tr(T_{jk}[\rho(r_k), \rho(\sqrt{\epsilon}\xi_j(\epsilon))]) &= \sum_{j,k=1}^n\sum_{\l} \Tr( P_1 [b_\l^{jk}\rho(r_k)a_\l^{jk}, \rho(\sqrt{\epsilon}\xi_j(\epsilon))])\\
			&= \sum_{j,k=1}^n\sum_{\l} \< \rho(\sqrt{\epsilon}\xi_j(\epsilon)^*) 1, \left(\Delta^{-1/2}_\theta \rho(b_\l^{jk}\rho(r_k) a_\l^{jk}) - b_\l^{jk}\rho(r_k)a_\l^{jk}\right)1\>_\theta\\
			&= \sum_{j,k}^n\sum_{\l} \<\rho(\sqrt{\epsilon}\xi_j(\epsilon)^*) 1, \Delta_\theta^{-1/2} \sigma_{-i/2}^\varphi(b_\l^{jk}) s_k \sigma_{-i/2}^\varphi(a_\l^{jk})1\>_\theta\\
			&= \sum_{j,k}^n\sum_{\l} \< \Delta_\theta^{-1/2} J_\theta \sqrt{\epsilon}\xi_j(\epsilon)  1, \sigma_{-i/2}^\varphi(b_\l^{jk}) s_k \sigma_{-i/2}^\varphi(a_\l^{jk})1\>_\theta\\
			&= \sum_{j,k}^n\sum_{\l} \< \sqrt{\epsilon}\xi_j(\epsilon)  1, \sigma_{-i/2}^\varphi(b_\l^{jk}) s_k \sigma_{-i/2}^\varphi(a_\l^{jk})1\>_\theta
	\end{align*}
where we have used Lemmas \ref{P1_on_commutator} and \ref{theta_on_conjugations_of_r}, and that $\Delta_\theta^{-1/2}J_\theta \xi_j(\epsilon)1=S_\theta \xi_j(\epsilon) 1=\xi_j(\epsilon) 1$ since $s_j$ is self-adjoint. So by applying the Cauchy--Schwarz inequality to our previous computation we obtain
	\begin{align*}
		\left|\sum_{j,k=1}^n  \Tr(T_{jk}[\rho(r_k), \rho(\sqrt{\epsilon}\xi_j(\epsilon))])\right| &\leq \left|\sum_{j=1}^n \< \sqrt{\epsilon}\xi_j(\epsilon)1,\sum_{k=1}^n \sum_{\l} \sigma_{-i/2}^\varphi(b_\l^{jk}) s_k \sigma_{-i/2}^\varphi(a_\l^{jk})1\>_\theta\right|\\
			&\leq \left(\epsilon \sum_{j=1}^n \left\|\xi_j(\epsilon)\right\|_\theta^2\right)^{1/2} \left( \sum_{j=1}^n \left\|\sum_{k=1}^n \sum_{\l} \sigma_{-i/2}^\varphi(b_\l^{jk}) s_k \sigma_{-i/2}^\varphi(a_\l^{jk})1\right\|_{\theta}^2\right)^{1/2}
	\end{align*}
Thus we have shown so far that
	\begin{align*}
		\left| \sum_{j,k=1}^n \<P_1, \left[\frac{2}{1+A}\right]_{kj} T_{jk}\>_{\text{HS}}\right| \leq 8n^2\delta + \left(\epsilon \sum_{j=1}^n \left\|\xi_j(\epsilon)\right\|_\theta^2\right)^{1/2}  \left( \sum_{j=1}^n \left\|\sum_{k=1}^n \sum_{\l} \sigma_{-i/2}^\varphi(b_\l^{jk}) s_k \sigma_{-i/2}^\varphi(a_\l^{jk})1\right\|_{\theta}^2\right)^{1/2}.
	\end{align*}
Since $d_A^\star(x_1,\ldots, x_n)=n$, the first factor in the second term above tends to zero as $\epsilon\to 0$. Thus we have the desired equality after first letting $\epsilon$ tend to zero and then $\delta$.
\end{proof}

\begin{cor}\label{commutation_kills}
Let $M$ be a von Neumann algebra with faithful normal state $\varphi$. Suppose $M$ is generated by a finite set $G=G^*$ of eigenoperators of $\sigma^\varphi$ with finite free Fisher information. If $T_1,\ldots, T_n\in \Psi(M_\infty\otimes M_\infty^{op})$ satisfy $\sum_{j=1}^n [T_j, \rho(x_j)]=0$, then $T_1=\cdots = T_n=0$.
\end{cor}
\begin{proof}
Let $\M$, $\theta$, and $M_\epsilon$ be as above. Lemma \ref{conjugate_variables_for_reguaralized_generators} implies $d_A^\star(x_1,\ldots, x_n)=n$. We compute
	\begin{align*}
		\sum_{j=1}^n \|T_j\|_{\text{HS}}^2 &= \sum_{j,m=1}^n [1]_{mj} \< T_m, T_j\>_{\text{HS}}\\
			&= \sum_{j,k,m=1}^n \left[\frac{1+A}{2}\right]_{mk} \left[\frac{2}{1+A}\right]_{kj} \<T_m, T_j\>_{\text{HS}}\\
			&= \sum_{j,k,m=1}^n \< \left[\frac{1+A}{2}\right]_{km} T_m , \left[\frac{2}{1+A}\right]_{kj} T_j\>_{\text{HS}}
	\end{align*}
We can write for each $k$ and $m$
	\begin{align*}
		\left[\frac{1+A}{2}\right]_{km}T_m = \sum_{\l} a^{km}_\l P_1 b^{km}_\l,
	\end{align*}
for some $a^{km}_\l, b^{km}_\l\in M$, so that
	\begin{align*}
		\sum_{j=1}^n \|T_j\|_{\text{HS}}^2 = \sum_{j,k}^n \<P_1, \left[\frac{2}{1+A}\right]_{kj}\sum_{m=1}^n\sum_{\l} (a^{km}_\l)^* T_j (b^{km}_\l)^* \>_{\text{HS}}.
	\end{align*}
We note that since $\rho(x_j)\in M'\cap\mc{B}(L^2(M,\vphi))$, we have for each $k=1,\ldots, n$
	\begin{align*}
		0 = \left(\sum_{m=1}^n\sum_{\l} (a_\l^{km})^* \otimes (b_\l^{km})^*\right)\# \sum_{j=1}^n [T_j, \rho(x_j)] = \sum_{j=1}^n \left[ \sum_{m=1}^n\sum_\l (a_\l^{km})^* T_j (b_\l^{km})^*, \rho(x_j)\right].
	\end{align*}
Thus if we set $T_{jk}:= \sum_{m=1}^n\sum_\l (a_\l^{km})^* T_j (b_\l^{km})^*$ for each $k=1,\ldots, n$ then $\{T_{jk}\}_{j,k=1}^n$ satisfies the hypothesis of Proposition \ref{orthogonal_to_P_1} and hence
	\begin{align*}
		\sum_{j=1}^n \|T_j\|_{\text{HS}}^2 = \sum_{j,k}^n \<P_1, \left[\frac{2}{1+A}\right]_{kj}\sum_{m=1}^n\sum_{\l} (a^{km}_\l)^* T_j (b^{km}_\l)^* \>_{\text{HS}}=0,
	\end{align*}
implying $T_j=0$ for each $j=1,\ldots,n$.
\end{proof}

%%%%%%%%%%%%%%%%%%%%%%%%%%%%%%%%%%%%%%%%%%%%%%%%%%%%%%%%%%%%%%%%%%%%%%%%%%%%%%%%%%%%%%%%%%%%%%%%%%%

\subsection{The existence of diffuse elements in the centralizer}\label{Diffuse}

Whenever we refer to a monomial in this section, we mean a monomial in $\C\<G\>$. To avoid confusion, we will temporarily return to using the notation $\C\<G\>$ in place of $\P$. We also remind the reader that the notations $E_\lambda$,  $E_I$, and $\mc{E}_\lambda$ were defined in Subsection \ref{Implications}.

We define
	\begin{align*}
		E_{\leq 1}^\infty =\{x\in \C\<G\>\colon [(\varphi\otimes 1)\circ\delta_{y_1}]\cdots [(\varphi\otimes 1)\circ\delta_{y_m}](x)\in E_{(0,1]},\ \forall n\geq0\text{ and }y_1,\ldots, y_m\in G\}.
	\end{align*}
Observe that for $x=y_1\cdots y_m\in\C\<G\>$ a monomial, $x\in E_{\leq 1}^\infty$ is equivalent to $y_j\cdots y_m\in \mc{E}_{\lambda_j}(M)$ with $\lambda_j\leq 1$ for each $j=1,\ldots, m$. %... being able to read $x$ from right to left and never having more $c_j^*$ than $c_j$ for any $j=1,\ldots, n$. For, example if $x\in \C\<c_1,c_1^*\>\cap E_{\leq 1}^\infty$ is a monomial we can write $x=(c_1^*)^{j_1}c_1^{k_1}\cdots (c_1^*)^{j_m}c_1^{k_m}$ where the exponents $j_1,k_1,\ldots, j_m,k_m\geq 0$ satisfy
%	\begin{align*}
%		j_\l \leq k_\l + (k_{\l-1} - j_{\l -1}) +\cdots +( k_m - j_m)\qquad \forall \l\in\{1,\ldots, m\}.
%	\end{align*}
If $w_j\in \C\<c_j,c_j^*\>\cap E_{\leq 1}^\infty$ are monomials for each $j=1,\ldots, n$ then a monomial formed by any interleaving of the factors in $w_1,\ldots, w_n$ is also in $E_{\leq 1}^\infty$.

\begin{lem}\label{non-tracial_zero_divisor_left_right_transfer}
For $x\in \mc{E}_\lambda(M)$, $\lambda\leq 1$, if $xp=0$ for some non-zero projection $p\in M$, then there exists a non-zero projection $q\in M^\varphi$ such that $qx=0$.
\end{lem}
\begin{proof}
If $x=v|x|$ is the polar decomposition, then $v\in \mc{E}_\lambda(M)$ by Proposition \ref{prop:eigenop_polar_decomp}. Now, $xp=0$ implies $p_{\ker(x)}$, the projection onto the kernel of $x$, is non-zero. Moreover, if $p_{\overline{\text{ran}(x)}}$ and $p_{\overline{\text{ran}(x^*)}}$ are the projections onto the closures of the ranges of $x$ and $x^*$, respectively, then we have
	\begin{align*}
		p_{\ker(x)}&=1-p_{\overline{\text{ran}(x^*)}} = 1- v^*v \\
		p_{\ker(x^*)}&= 1- p_{\overline{\text{ran}(x)}} = 1 -vv^*.
	\end{align*}
Hence
	\begin{align*}
		\varphi(p_{\ker(x^*)}) = 1 - \varphi(vv^*) = 1- \varphi( v^*\sigma_{-i}^\varphi(v))= 1- \lambda \varphi(v^*v) = (1-\lambda) + \lambda \varphi(p_{\ker(x)}) >0.
	\end{align*}
So letting $q=p_{\ker(x^*)}\neq 0$, we have $x^*q=0$ or $qx=0$, and $q=1-vv^*\in M^\varphi$.
\end{proof}

\begin{thm}\label{no_zero_divisors_from_the_centralizer}
Assume $\Phi_\vphi^*(G)<\infty$. Suppose for $x\in \C\<G\>$ there is a monomial $x_0$ of highest degree with non-zero coefficient such that $x_0\in E_{\leq 1}^\infty$. Then there is no non-zero projection $p\in M^\varphi$ such that $xp=0$.
\end{thm}
\begin{proof}
Suppose, towards a contradiction, $p\in M^\varphi$ is a non-zero projection such that $xp=0$. Let $\lambda_1\in (0,1]$ be such that $x_0\in \mathcal{E}_{\lambda_1}(M)$. Then
	\begin{align*}
		0= \|xp\|_\varphi^2 = \| \pi_{\lambda_1}xp\|_\varphi^2 + \| (1-\pi_{\lambda_1})xp\|_\varphi^2,
	\end{align*}
and hence $0=\pi_{\lambda_1}xp= \mc{E}_{\lambda_1}(xp) =\mc{E}_{\lambda_1}(x)p$. Lemma \ref{non-tracial_zero_divisor_left_right_transfer} implies there is a non-zero projection $q_1\in M^\varphi$ such that $q_1\mc{E}_{\lambda_1}(x)=0$. Thus
	\begin{align*}
		0 = (q_1\otimes p)\# (\mc{E}_{\lambda_1}(x)\otimes 1 - 1\otimes \mc{E}_{\lambda_1}(x)) = (q_1\otimes p)\#\sum_{y\in G} \delta_y(\mc{E}_{\lambda_1}(x))\#(y\otimes 1 - 1\otimes y).
	\end{align*}
Setting $T_y:= (\sigma_{-i/2}^\varphi\otimes \sigma_{i/2}^\varphi)\left[(q_1\otimes p)\# \delta_y(\mc{E}_{\lambda_1}(x))\right]$ for each $y\in G$ and applying $\sigma_{-i/2}^\varphi\otimes \sigma_{i/2}^\varphi$ to each side of the above equality yields
	\begin{align*}
		0 = \sum_{y\in G} T_y\#(\sigma_{-i/2}^\varphi(y)\otimes 1 - 1\otimes \sigma_{i/2}^\varphi(y)) = \sum_{y\in G} [\rho(y), T_y],
	\end{align*}
by (\ref{rho_action}). So by Corollary \ref{commutation_kills}, $T_y=0$ for each $y\in G$ and hence $(q_1\otimes p) \# \delta_y(\mc{E}_{\lambda_1}(x))=0$ for each $y\in G$. 

Let $x_0=y_1\cdots y_m$ for $y_1,\ldots, y_m\in G$. Then in particular,
	\begin{align*}
		(\varphi\otimes 1)( q_1\cdot \delta_{y_1}(\mc{E}_{\lambda_1}(x))) p=0.
	\end{align*}
The condition on $x_0$ implies that $y_\l\cdots y_m\in E_{(0,1]}$ for each $\l\in\{2,\ldots, m\}$, and so iterating the above argument we can find $\lambda_2,\ldots, \lambda_m \in (0,1]$ and non-zero projections $q_2,\ldots, q_m$ so that for each $\l\in\{2,\ldots, m\}$ if $(\vphi q)(\cdot) = \vphi(q\cdot)$ then
	\begin{align*}
		[(\varphi q_m)\otimes 1]\circ\delta_{y_m}\circ\mc{E}_{\lambda_m}\circ\cdots\circ [(\vphi q_1)\otimes 1]\circ\delta_{y_1}\circ\mc{E}_{\lambda_1}(x) p =0.
	\end{align*}
But since $x_0$ is of highest degree in $x$, say with coefficient $\alpha\neq 0$, we have
	\begin{align*}
		0= [(\varphi q_m)\otimes 1]\circ\delta_{y_m}\circ\mc{E}_{\lambda_m}\circ\cdots\circ [(\vphi q_1)\otimes 1]\circ\delta_{y_1}\circ\mc{E}_{\lambda_1}(x) p = \alpha \varphi(q_m)\cdots\varphi(q_1)p,
	\end{align*}
a contradiction.
\end{proof}

\begin{cor}\label{diffuse_centralizer_elements}
Assume $\Phi_\vphi^*(G)<\infty$. Suppose for $x\in \C\<G\>$ there is a monomial $x_0$ of highest degree with non-zero coefficient such that $x_0\in E_{\leq 1}^\infty$.
	\begin{enumerate}
	\item If $x\in \mc{E}_\lambda(M)$ for $\lambda\leq 1$, then $x$ has no atoms at zero.
	
	\item If $x\in M^\varphi$, then $x$ is diffuse.
	\end{enumerate}	
\end{cor}
\begin{proof}
If $x\in \mc{E}_\lambda(M)$, $\lambda\leq 1$, had an atom at zero, then there would be a non-zero projection $p$ such that $xp=0$. Hence we may apply Theorem \ref{no_zero_divisors_from_the_centralizer} with $p_{\ker(x)}=1-v^*v\in M^\varphi$, $v$ coming from the polar decomposition of $x$, to obtain a contradiction. For $x\in M^\varphi$, we simply apply (1) to each translation $x-\alpha\in \mc{E}_1(M)=M^\vphi$, $\alpha\in \C$.
\end{proof}

\begin{cor}\label{atom_size}
Assume $\Phi_\vphi^*(G)<\infty$. For $y_1,\ldots, y_m\in G$, let $\lambda_j$ be the eigenvalue of $y_j\cdots y_m$ for each $j=1,\ldots, m$, and suppose
	\[
		\lambda_1=\max_{1\leq j\leq m} \lambda_j.
	\]
Then the element $x=y_m^*\cdots y_1^*y_1\cdots y_m\in (M^\vphi)_+$ is diffuse when $\lambda_1\leq 1$ and otherwise has exactly one atom, which is at zero and of size $1-\frac{1}{\lambda_1}$.
\end{cor}
\begin{proof}
Define $\lambda_0=\lambda_{m+1}=1$. First suppose $\lambda_1\leq 1$. For each $j=1,\ldots,m$, the eigenvalues of $y_j\cdots y_m$ and $y_{j}^*\cdots y_1^*y_1\cdots y_m$ are easily seen to be $\lambda_j$ and $\lambda_{j+1}$, respectively. Each of these is less than one by assumption, so $x\in E_{\leq 1}^\infty$ and is diffuse by Corollary \ref{diffuse_centralizer_elements}.

Now suppose $\lambda_1>1$ and consider $\tilde{x}:=y_1\cdots y_my_m^*\cdots y_1^*\in (M^\vphi)_+$. For each $j=1,\ldots, m$, the eigenvalues of $y_j^*\cdots y_1^*$ and $y_j\cdots y_my_m^*\cdots y_1^*$ are $\frac{\lambda_{j+1}}{\lambda_1}$ and $\frac{\lambda_j}{\lambda_1}$, respectively. By assumption, each of these is less than $\frac{\lambda_1}{\lambda_1}=1$, so $\tilde{x}\in E_{\leq 1}^\infty$ and hence is diffuse. Now, for any $d\geq 1$ we have
	\[
		\vphi( x^k)=\vphi( \sigma_i^\vphi(y_1\cdots y_m) x^{k-1} y_m^*\cdots y_1^*) = \frac{1}{\lambda_1} \vphi( \tilde{x}^k).
	\]
Consequently, for any polynomial $p$ with $p(0)=0$, we have $\vphi(p(x))=\lambda_1^{-1}\vphi(p(\tilde{x}))$. By approximating a Dirac mass $\delta_t$ for $t\in (0,\infty)$ by such polynomials, we see that $x$ has no atoms away from zero since $\tilde{x}$ is diffuse. On the other hand, by approximating $1-\delta_0$ by such polynomials we see that $x$ must have an atom of size $1-\frac{1}{\lambda_1}$ at zero.
\end{proof}

\begin{ex}
Let $y\in G$ have eigenvalue $\lambda\leq 1$. Then by Corollary \ref{atom_size}, $y^*y$ is diffuse and $yy^*$ has an atom of size $1-\lambda$ at zero.
\end{ex}

\begin{rem}
One application of Corollary \ref{diffuse_centralizer_elements} is to operators arising from loops on a weighted graph which begin and end on a vertex of minimal weight. Following \cite[Section 4]{GJS10}, given an oriented bipartite graph $\Gamma=(V,E)$, one can associate to each edge $e\in E$ a generalized circular operator $c(e)$ acting on a Fock space $\mc{F}_\Gamma$ associated to the graph. Let $\varphi_\Gamma$ denote the vacuum state on $\mc{B}(\mc{F}_\Gamma)$, then the action of its associated modular automorphism group $\sigma^{\varphi_\Gamma}$ on the operators $c(e)$ is known:
	\begin{align*}
		\sigma^{\varphi_\Gamma}_z(c(e)) = \left(\frac{\mu(t(e))}{\mu(s(e))}\right)^{iz},\qquad z\in \C,
	\end{align*}
where $s(e),t(e)\in V$ are the source and target vertices for $e$, respectively, and $\mu$ is the weighting given by the Perron--Frobenius eigenvector of the adjacency matrix of $\Gamma$. It then follows that $c(e_1)c(e_2)\cdots c(e_m)$ is always an eigenoperator of $\sigma^{\varphi_\Gamma}$, and in particular is in the centralizer with respect to $\varphi_\Gamma$ whenever $e_1e_2\cdots e_m$ is a loop in $\Gamma$. Fix a loop $e_1\cdots e_m$ in $\Gamma$ and observe that the eigenvalue of $c(e_j)\cdots c(e_m)$, $1\leq j\leq m$, is
	\begin{align*}
		\frac{\mu(t(e_j)))}{\mu(s(e_j)))}\frac{\mu(t(e_{j+1}))}{\mu(s(e_{j+1}))}\cdots \frac{\mu(t(e_m))}{\mu(s(e_m))} = \frac{\mu(t(e_m))}{\mu(s(e_j))},
	\end{align*}
since $t(e_k)=s(e_{k+1})$. Thus if $t(e_m)=s(e_1)$ has minimal $\mu$-weight amongst all vertices traversed by the loop, then $c(e_1)\cdots c(e_m)$ is diffuse by Corollary \ref{diffuse_centralizer_elements}. A more direct proof of this result as well as a broader consideration of graphs and weightings $\mu$ can be found in \cite{Har15}.
\end{rem}

%%%%%%%%%%%%%%%%%%%%%%%%%%%%%%%%%%%%%%%%%%%%%%%%%%%%%%%%%%%%%%%%%
%		   Derivations From The Polar Decomposition		%
%%%%%%%%%%%%%%%%%%%%%%%%%%%%%%%%%%%%%%%%%%%%%%%%%%%%%%%%%%%%%%%%%

\section{Derivations From The Polar Decomposition}\label{polar_deriv}

Recall that for $y\in G$, $\delta_y$ is the free difference quotient with respect to $y$, which is $\lambda_y$-modular. When it exists, we denote by $\xi_y$ the conjugate variable to $\delta_y$. From these derivations and the polar decomposition of $y$ we will construct new derivations. In particular, we will construct a derivation associated to $|y|$ which can be restricted to the tracial context of $(M^\vphi,\vphi)$.

%%%%%%%%%%%%%%%%%%%%%%%%%%%%%%%%%%%%%%%%%%%%%%%%%%%%%%%%%%%%%%%%%%%%%%%%%%%%%

\subsection{The derivations $\delta_{|y|}$ and $\delta_v$}\label{new_derivations}

For $y\in G$ with $y\neq y^*$ and polar decomposition $y=v|y|$, we now consider two new derivations $\delta_{|y|}$ and $\delta_v$. Set $\dom{(\delta_{|y|})}=\dom(\delta_v)=\P$ and for $p\in \P$ define these derivations by
	\begin{align*}
		\delta_{|y|}(p)&:=\delta_y(p)\# (v\otimes 1) + \delta_{y^*}(p)\# (1\otimes v^*),\\
		\delta_v(p)& := \delta_y(p)\# (v\otimes |y|) - \delta_{y^*}(p)\# (|y|\otimes v^*).
	\end{align*}
Observe that $\delta_{|y|}$ and $\delta_v$ are $1$-modular, since $v\in \mc{E}_{\lambda_y}(M)$. Also $\hat\delta_{|y|}=\delta_{|y|}$ and $\hat\delta_v=-\delta_v$.

\begin{lem}\label{abs_value_derivation_closable_in_L2}
Fix $y\in G$ with polar decomposition $y=v|y|$, and suppose $y\neq y^*$ and that $\xi_y$ exists. Then $1\otimes 1\in \dom(\delta_{|y|}^*)\cap \dom(\delta_v^*)$ so that $\delta_{|y|}$ and $\delta_v$ are closable as densely defined maps
	\[
		\delta_{|y|},\delta_v\colon L^2(M,\varphi)\to L^2(M\bar{\otimes}M^{op}).
	\]
Furthermore, $|y|,|y^*|\in \dom{(\bar{\delta}_{|y|})}$ with 
	\[
		\bar{\delta}_{|y|}(|y|)=v^*v\otimes 1 + 1\otimes v^*v - v^*v\otimes v^*v \qquad{\and}\qquad \bar{\delta}_{|y|}(|y^*|) = v\otimes v^*,
	\]
and $|y|,|y^*|,v,v^*\in \dom(\bar{\delta}_v)$ with
	\[
		\bar{\delta}_v(|y|)=0,\qquad \bar{\delta}_v(|y^*|)=[v\otimes v^*,|y^*|], \qquad \bar{\delta}_v(v)=v\otimes v^*v,\qquad\text{and}\qquad\bar{\delta}_v(v^*)=-v^*v\otimes v.
	\]
\end{lem}
\begin{proof}
We compute for $p\in \P$
	\begin{align*}
		\<1\otimes 1, \delta_{|y|}(p)\>_{\text{HS}} &= \<1\otimes 1, \delta_y(p)\# v\otimes 1\>_{\text{HS}} + \<1\otimes 1, \delta_{y^*}(p)\# 1\otimes v^*\>_{\text{HS}}\\
			&=  \< \lambda_y^{-1} v^*\otimes 1, \delta_y(p)\>_{\text{HS}} + \<\lambda_y 1\otimes v, \delta_{y^*}(p)\>_{\text{HS}}.
	\end{align*}
Now, Corollary \ref{boundedness_of_adjoint_formula} implies $v^*\otimes 1\in \dom{(\delta_y^*)}$ and $1\otimes v\in \dom{(\delta_{y^*}^*)}$. Thus
	\begin{align*}
		\<1\otimes 1, \delta_{|y|}(p)\>_{\text{HS}} = \< \lambda_y^{-1}\delta_y^*(v^*\otimes 1) + \lambda_y \delta_{y^*}^*(1\otimes v), p\>_{\vphi},
	\end{align*}
which implies $1\otimes 1\in \dom{(\delta_{|y|}^*)}$. Hence $\delta_{|y|}$ is closable by Lemma \ref{adjoint_formula}. The proof for $\delta_v$ is similar.

We note that
	\begin{align*}
		\delta_{|y|}(y^*y)&= (1\otimes y)\# (1\otimes v^*) + (y^*\otimes 1)\# (v\otimes 1)\\
			&= 1\otimes (v^*y) + (y^*v)\otimes 1\\
			&= 1\otimes |y| + |y|\otimes 1,
	\end{align*}
since $v^*v=1-p_{\ker(y)}=1-p_{\ker(|y|)}$. Let $I\subset [0,\infty)$ be a compact interval containing  the spectrum of $y^*y$, and let $\mu$ be the spectral measure of $y^*y$. For each $\alpha>0$, consider the function
	\begin{align*}
		f_\alpha(t)=\sqrt{t+\alpha^2} \in C^1(I).
	\end{align*}
By Lemma \ref{C^1_functional_calculus}, $f_\alpha(y^*y)\in \dom{(\bar{\delta}_{|y|})}$ and $\bar{\delta}_{|y|}(f_\alpha(y^*y))$ is identified (via the functional calculus) with
	\begin{align*}
		\frac{f_\alpha(t) - f_\alpha(s)}{t-s}(\sqrt{t}+\sqrt{s}) = \frac{\sqrt{t} + \sqrt{s}}{\sqrt{t+\alpha^2} + \sqrt{s+\alpha^2}},
	\end{align*}
which by Lebesgue's dominated convergence theorem (DCT) converges in $L^2(I\times I,\mu\times \mu)$ to the function $1-\chi_{\{(0,0)\}}(t,s)$, where $\chi_E$ denotes the characteristic function on a set $E$. Since $v^*v=1-p_{\ker(y^*y)}=\chi_{(0,\infty)}(y^*y)$ and
	\begin{align*}
		1-\chi_{\{(0,0)\}}(t,s) = \chi_{(0,\infty)}(t) + \chi_{(0,\infty)}(s) - \chi_{(0,\infty)}(t)\chi_{(0,\infty)}(s),
	\end{align*}
this limit is identified via the Borel functional calculus for $y^*y$ with
	\[
		v^*v\otimes 1 + 1\otimes v^*v - v^*v\otimes v^*v.
	\]
Thus $\bar{\delta}_{|y|}(f_\alpha(y^*y))\to v^*v\otimes 1 + 1\otimes v^*v - v^*v\otimes v^*v$ in $L^2(M\bar{\otimes}M^{op})$. Since $f_\alpha(t)$ converges to $\sqrt{t}$ uniformly on $I$, it follows that $f_\alpha(y^*y)\to |y|$ in $L^2(M,\varphi)$. Hence, $|y|\in \dom{(\bar{\delta}_{|y|})}$ with $\bar{\delta}_{|y|}(|y|)=v^*v\otimes 1 + 1\otimes v^*v - v^*v\otimes v^*v$.

Now let $I\subset [0,\infty)$ be a compact interval containing the spectrum of $|y|$, and let $\mu$ be the spectral measure of $|y|$. For each $\alpha>0$, consider the function
	\[
		g_\alpha(t)=\frac{t}{(t+\alpha)^2}\in C^1(I).
	\]
Since $|y^*|=v|y|v^*$ and $v$ is an eigenoperator with eigenvalue $\lambda_y$, we have
	\[
		\| yg_\alpha(|y|)y^* -|y^*|\|_\vphi = \lambda_y\| |y|g_\alpha(|y|)|y|-|y|\|_\vphi.
	\]
So $yg_\alpha(|y|)y^*\to |y^*|$ in $L^2(M,\vphi)$ since $t^2g_\alpha(t)\to t$ in $L^2(I,\mu)$ by Lebesgue's DCT. By Corollary \ref{extC1} we have $g_\alpha(|y|)\in \dom(\bar{\delta}_{|y|})$. Then, since $g_\alpha(|y|)\in M^\vphi$ and $y,y^*\in M_0$ we have by Proposition \ref{closure_is_der} that $yg_\alpha(|y|)y^*\in \dom(\bar{\delta}_{|y|})$ with
	\begin{align*}
		\bar{\delta}_{|y|}(yg_\alpha(|y|)y^*)= v\otimes g_\alpha(|y|)y^* + y\cdot LR_{|y|}(\tilde{g}_\alpha)\cdot y^* + yg_\alpha(|y|)\otimes v^*.
	\end{align*}
We claim that this converges to $v\otimes v^*$ in $L^2(M\bar\otimes M^{op})$. Let $p=v^*v=1-p_{\ker(|y|)}$, then a straightforward computation shows
	\begin{align*}
		\| v\otimes g_\alpha(|y|)y^* + &y\cdot LR_{|y|}(\tilde{g}_\alpha)\cdot y^* + yg_\alpha(|y|)\otimes v^* - v\otimes v^*\|_{\text{HS}}\\
			&= \| p\otimes g_\alpha(|y|)|y| + |y|\cdot LR_{|y|}(\tilde{g}_\alpha)\cdot |y| + |y|g_\alpha(|y|)\otimes p - p\otimes p\|_{\text{HS}}.
	\end{align*}
It is readily seen (after computing $\tilde{g}_\alpha$) that $p\otimes g_\alpha(|y|)|y| + |y|\cdot LR_{|y|}(\tilde{g}_\alpha)\cdot |y| + |y|g_\alpha(|y|)\otimes p - p\otimes p$ corresponds via the Borel functional calculus for $|y|$ to
	\begin{align*}
		\chi_{(0,\infty)}(t) s g_\alpha(s)+&(\alpha^2-ts) g_\alpha(t) g_\alpha(s) +tg_\alpha(t)\chi_{(0,\infty)}(s) - \chi_{(0,\infty)}(t)\chi_{(0,\infty)}(s)\\
			&=\left[\chi_{(0,\infty)}(t) - \frac{t^2}{(t+\alpha)^2}\right]\left[\frac{s^2}{(s+\alpha)^2}-\chi_{(0,\infty)}(s)\right] + \frac{t\alpha}{(t+\alpha)^2}\frac{s\alpha}{(s+\alpha)^2}.
	\end{align*}
This converges to zero in $L^2(I\times I, \mu\times \mu)$ by Lebesgue's DCT. Hence $|y^*|\in dom(\bar{\delta}_{|y|})$ with $\bar{\delta}_{|y|}(|y^*|)=v\otimes v^*$.

Seeing that $|y|\in \dom(\bar{\delta}_v)$ is much easier. Indeed,
	\begin{align*}
		\delta_v(y^*y) &= -(|y|\otimes v^*)\cdot y + y^*\cdot (v\otimes |y|)\\
				&= -|y|\otimes (v^*v|y|) + (|y|v^*v)\otimes |y| =0,
	\end{align*}
so by Lemma \ref{C^1_functional_calculus} $f_\alpha(y^*y)\in \dom(\bar{\delta}_v)$ with $\bar{\delta}_v(f_\alpha(y^*y))=0$ for all $\alpha>0$. This, in turn, implies $|y|\in \dom(\bar{\delta}_v)$ with $\bar\delta_v(|y|)=0$.

Towards showing $v\in \dom(\bar{\delta}_v)$, we consider the elements $y(\alpha+|y|)^{-1}$, $\alpha>0$, which we claim converge to $v$ in $L^2(M,\varphi)$ as $\alpha\to 0$. Once again let $I\subset [0,\infty)$ be a compact interval containing the spectrum of $|y|$, and let $\mu$ be the spectral measure of $|y|$. Noting that
	\begin{align*}
		\|y(\alpha+|y|)^{-1} - v\|_\varphi = \||y|(\alpha+|y|)^{-1} - v^*v\|_\vphi,
	\end{align*}
we see that the convergence follows because $ |y|(\alpha+|y|)^{-1} - v^*v$ corresponds via the Borel functional calculus of $|y|$ to the function $t(\alpha+t)^{-1} - \chi_{(0,\infty)}(t)$, which converges to zero on $L^2(I,\mu)$ by Lebesgue's DCT. Furthermore, $(\alpha+|y|)^{-1}\in \dom(\bar{\delta}_v)$ with $\bar\delta_v( (\alpha+|y|)^{-1})=0$ by Lemma \ref{C^1_functional_calculus}. So by the first part of Lemma \ref{extension_of_Leibniz_rule} we have
	\begin{align*}
		\bar\delta_v( y(\alpha+|y|)^{-1}) = (v\otimes |y|)\cdot (\alpha+|y|)^{-1},
	\end{align*}
which converges to $v\otimes v^*v$ as $\alpha\to 0$. Thus $v\in \dom(\bar{\delta}_v)$ with $\bar\delta_v(v)=v\otimes v^*v$. The proof for $v^*$ is similar using $(\alpha+|y|)^{-1}y^*$ to approximate $v^*$.

Finally, observe that since $\hat{\delta}_v=-\delta_v$ we have $\overline{\dom{}}(\delta_v\oplus \hat{\delta}_v)=\dom(\bar{\delta}_v)$. Thus Proposition \ref{closure_is_der} implies $|y^*|=v|y|v^*\in \dom(\bar{\delta}_v)$ with
	\[
		\bar{\delta}_v(|y^*|)=v\otimes v^*v|y|v^* - v|y|v^*v\otimes v^*= v\otimes v^*|y^*| - |y^*|v\otimes v^* = [v\otimes v^*, |y^*|]
	\]
as claimed.
\end{proof}

\begin{rem}
Consider the derivation
	\begin{align*}
		\tilde{\delta}_v:=\delta_y\# (vv^*\otimes v|y|) - \delta_{y^*}\# (|y|v^*\otimes vv^*) = \delta_v\# (v^*\otimes v).
	\end{align*}
In light of Lemma \ref{abs_value_derivation_closable_in_L2}, we can see
	\begin{align*}
		\tilde{\delta}_v(v) &= vv^*\otimes v\\
		\tilde{\delta}_v(v^*) &= - v^*\otimes vv^*,
	\end{align*}
which is exactly the derivation defined in \cite[Section 2.3]{Shl01}---the inspiration for $\delta_v$.
\end{rem}

The primary benefit of the derivation $\bar{\delta}_{|y|}$, is that it gives us a derivation on the tracial von Neumann algebra $M^\vphi$. By composing this with $\mc{E}_\vphi\otimes \mc{E}_\vphi^{op}$, we can consider an exclusively tracial context. In the following lemma, we check a few details for these derivations that will be relevant for the proofs of Theorems A and B.

\begin{lem}\label{tracial_ultrapower_result}
For $y\in G$ satisfying $y\neq y^*$ and $\lambda_y\leq 1$, assume $\xi_y$ exists. Let $Y_1=|y|$, $Y_2=|y^*|$, and $D=\P[Y_1,Y_2]\cap M^\vphi$ (a dense subset of $L^2(M^\vphi,\vphi)$). Then
	\begin{align*}
		\delta_j:=(\mc{E}_\vphi\otimes \mc{E}_\vphi^{op})\circ\bar\delta_{Y_j}\mid_D\qquad j=1,2
	\end{align*}
are $1$-modular derivations valued in $M^\vphi\otimes (M^\vphi)^{op}$ that satisfy for $j=1,2$:
	\begin{itemize}
		\item[(i)] $1\otimes 1\in \dom(\delta_j^*)$ with $\delta_j^*(1\otimes 1)\in L^2(M^\varphi,\varphi)$;
		
		\item[(ii)] $\delta_j$ is closable;
		
		\item[(iii)] $(1\otimes \varphi)\circ \delta_j$ extends to a bounded map from $M^\varphi$ to $L^2(M^\varphi,\varphi)$.

	\end{itemize}
Moreover,
			\[
				\delta_j(Y_k)=\delta_{j=k} U_j
			\]
for $j,k=1,2$, where $U_1:=1\otimes 1$ and $U_2:=vv^*\otimes 1 + 1\otimes vv^* - vv^*\otimes vv^*$. Also, $\delta_j\# U_j=\delta_j$ for $j=1,2$.
\end{lem}
\begin{proof}
Note that $\P\cap M^\vphi=\mc{E}_\vphi(\P)$ is dense in $L^2(M^\vphi,\vphi)$: $\P$ is dense in $L^2(M,\vphi)$ and
	\[
		\|\xi - p\|_\vphi^2 = \| \xi- \mc{E}_\vphi(p)\|_\vphi^2+\|(1-\mc{E}_\vphi)(p)\|_\vphi^2\qquad \forall \xi\in L^2(M^\vphi,\vphi),\ p\in \P.
	\]
Consequently, $D$ is dense.

Lemma \ref{abs_value_derivation_closable_in_L2} and Proposition \ref{closure_is_der} imply that $\bar{\delta}_{|y|}$ and $\bar{\delta}_{|y^*|}$ are $1$-modular derivations on $D$. Since $D\subset M^\vphi$ we see that $\delta_1$ and $\delta_2$ satisfy the Leibniz rule and are therefore  $1$-modular derivations valued in $M^\vphi\otimes (M^\vphi)^{op}$.

Lemma \ref{abs_value_derivation_closable_in_L2} implies that $1\otimes1\in \dom(\delta_{|y|}^*)$ with $\xi_{|y|}:=\delta_{|y|}^*(1\otimes 1)\in L^2(M,\varphi)$. Then Lemma \ref{conj_var_entire} implies $\xi_{|y|}$ is invariant under $\Delta_\vphi$ and therefore contained in $L^2(M^\vphi, \vphi)$. Since $\mc{E}_\vphi\otimes\mc{E}_{\vphi}^{op}(1\otimes 1)=1\otimes 1$, it follows that $1\otimes 1 \in \dom(\delta_1^*)$ with $\delta_1^*(1\otimes 1)=\delta_{|y|}^*(1\otimes 1)=\xi_{|y|}$. Similarly for $|y^*|$ and $\delta_2$. This establishes (i). Since the conjugate variables exist, Lemma \ref{adjoint_formula} implies $\delta_1$ and $\delta_2$ are closable. Thus (ii) is established, and (iii) follows from Corollary \ref{boundedness_of_adjoint_formula}.

Note that $\lambda_y\leq 1$ implies $y^*y$ is diffuse by Corollary \ref{diffuse_centralizer_elements}. Hence $v^*v=1$ and the formula for $\delta_j(U_k)$, $j,k=1,2$, follows from Lemma \ref{abs_value_derivation_closable_in_L2}. Now, $\delta_1\# U_1=\delta_1 \# (1\otimes 1) = \delta_1$ is clear, and towards showing $\delta_2\# U_2 =\delta_2$ observe that
	\[
		\delta_2\# U_2 = (\mc{E}_\vphi\otimes \mc{E}_\vphi^{op})\left( \bar\delta_{Y_2}\mid_D \# U_2\right).
	\]
So it suffices to show $\bar\delta_{|y^*|}\# U_2 = U_2$. Recall that
	\[
		\delta_{|y^*|} = \delta_{y^*}\# (v^*\otimes 1) + \delta_y \# (1\otimes v).
	\]
Since $(v^*\otimes 1)\# U_2= v^*\otimes 1$, and similarly for $1\otimes v$, we have $\delta_{|y^*|}\# U_2=U_2$. Let $\xi\in \dom(\bar\delta_{|y^*|})$ be approximated by $\{p_n\}_{n\in\N}\subset \P$ in the $\|\cdot\|_{\bar\delta_{|y^*|}}$-norm. Then for any $\eta\in L^2(M\bar\otimes M^{op})$ we have
	\begin{align*}
		\< \bar\delta_{|y^*|}(\xi)\# U_2, \eta\>_{\text{HS}} & = \< \bar\delta_{|y^*|}(\xi), \eta\# U_2^*\>_{\text{HS}}\\
			&= \lim_{n\to\infty} \< \delta_{|y^*|}(p_n), \eta\# U_2^*\>_{\text{HS}}\\
			&= \lim_{n\to\infty} \< \delta_{|y^*|}(p_n)\# U_2, \eta\>_{\text{HS}}\\
			&= \lim_{n\to\infty} \< \delta_{|y^*|}(p_n), \eta\>_{\text{HS}}\\
			&= \<\bar\delta_{|y^*|}(\xi),\eta\>_{\text{HS}},
	\end{align*}
where we have used $U_2\in M^\vphi\otimes (M^\vphi)^{op}$. Thus $\bar\delta_{|y^*|}\# U_2 = \bar\delta_{|y^*|}$ and the desired formula holds.
\end{proof}

We are very grateful to Yoann Dabrowski who suggested to us the following lemma, which is a modified version of \cite[Proposition 3.21]{CDS14}.

\begin{lem}\label{lem:closable_big_kernel}
For $y\in G$ satisfying $y\neq y^*$ and $\lambda_y\leq 1$, assume $\xi_y$ exists. With the same notations as in the previous lemma, consider $\delta_0\colon D\to M^\vphi\otimes (M^\vphi)^{op}$ defined for $p\in D$ by
	\[
		\delta_0(p):= \sum_{j=1}^2 \delta_j(p) \#(Y_j\otimes 1-1\otimes Y_j) - [p,1\otimes 1].
	\]
Then $1\otimes 1\in \dom(\delta_0^*)$ with
	\[
		\delta_0^*(1\otimes 1) = \sum_{j=1}^2 [Y_j,\delta_j^*(1\otimes 1)].
	\]
In particular, $\delta_0$ is a closable operator, with closure $\bar\delta_0$ satisfying $L^2(\C[Y_1,Y_2],\vphi)\subset \ker(\bar\delta_0)$. Moreover, for any $x\in \dom(\bar\delta_0)\cap M^\vphi$ we have
	\begin{align}\label{eqn:vanishing_derivation}
		\| [x,U_2]\|_{\text{HS}}^2 =&  - \<\bar\delta_0(x), [x,1\otimes 1]\>_{\text{HS}}\nonumber\\
			 &+\sum_{j=1}^2 \left(\< [x,Y_j],[x, \delta_j^*(U_j)]\>_\vphi + 2\Re\< (\vphi\otimes 1 - 1\otimes \vphi^{op})\circ \bar\delta_j(x), [x,Y_j]\>_\vphi \right) .
	\end{align}
\end{lem}
\begin{proof}
For $p\in D$ note that $\vphi\otimes\vphi^{op}([p,U_j])=0$ for $j=1,2$. Thus, we have
	\begin{align*}
		\< \sum_{j=1}^2 [Y_j,\delta_j^*(1\otimes 1)], p\>_\vphi &= \sum_{j=1}^2 \<1\otimes 1, \delta_j([Y_j,p])\>_{\text{HS}}\\
			&= \sum_{j=1}^2 \<1\otimes 1, [U_j, p]\>_{\text{HS}}+ \< Y_j\otimes 1 - 1\otimes Y_j, \delta_j(p)\>_{\text{HS}}\\
			&= \sum_{j=1}^2 \< 1\otimes 1, \delta_j(p)\#(Y_j\otimes 1 - 1\otimes Y_j)\>_{\text{HS}}\\
			&= \<1\otimes 1, \delta_0(p)\>_{\text{HS}}.
	\end{align*}
So $1\otimes 1\in \dom(\delta_0^*)$ with the claimed image. Hence $\delta_0$ is closable by Lemma \ref{adjoint_formula}.

For $p\in \C[Y_1,Y_2]$, $\delta_0(p)=0$ since $U_2\# (Y_2\otimes 1 - 1\otimes Y_2) = Y_2\otimes 1- 1\otimes Y_2$. It follows that $L^2(\C[Y_1,Y_2],\vphi)\subset \dom(\bar\delta_0)$ with $\delta_0$ identically zero on this subspace.

We now establish (\ref{eqn:vanishing_derivation}). First note that $U_j\in \dom(\delta_k^*)$ for $j,k=1,2$ by Corollary \ref{boundedness_of_adjoint_formula}, and so the right-hand side is well-defined. Also, by Theorem \ref{Kaplansky} it suffices to establish this formula for $p\in D$. We have $\delta_j([p,Y_j]) = [\delta_j(p), Y_j] + [p,U_j]$. So
	\begin{align*}
			\|[p,U_j]\|_{\text{HS}}^2 = \< \delta_j([p,Y_j]),[p,U_j]\>_{\text{HS}} - \<[\delta_j(p), Y_j], [p,U_j]\>_{\text{HS}}=: {\I}_j - {\II}_j.
	\end{align*}
We compute
	\begin{align*}
		{\I}_j&= \< [p^*,\delta_j([p,Y_j])], U_j\>_{\text{HS}}\\
			&= \< \delta_j([p^*,[p,Y_j]]), U_j\>_{\text{HS}} - 	\<[\delta_j(p^*), [p,Y_j]], U_j\>_{\text{HS}}\\
			&= \< [p,Y_j], [p,\delta_j^*(U_j)]\>_{\text{HS}} +  \< \delta_j(p^*), [U_j, [p^*,Y_j]]\>_{\text{HS}}
	\end{align*}
Focusing on the second term, we use $\hat{\delta}_j=\delta_j$ and $U_j^\dagger=U_j$ to observe that
	\begin{align*}
		\< \delta_j(p^*), [U_j, [p^*,Y_j]]\>_{\text{HS}} &= \overline{ \< [U_j, [p^*,Y_j]],\delta_j(p)^\dagger\>_{\text{HS}}}\\
			&= \overline{\vphi\otimes\vphi^{op}( [U_j, [p^*,Y_j]]^*\# \delta_j(p)^\dagger)}\\
			&= \vphi\otimes\vphi^{op}( ([U_j, [p^*,Y_j]]^\dagger)^* \# \delta_j(p))\\
			&= \< [U_j, [p^*,Y_j]]^\dagger, \delta_j(p)\>_{\text{HS}}\\
			&= \< [ [Y_j, p], U_j], \delta_j(p)\>_{\text{HS}}\\
			&= \< [U_j, [p,Y_j]], \delta_j(p)\>_{\text{HS}}.
	\end{align*}
Thus we have shown
	\[
		{\I}_j = \< [p,Y_j], [p,\delta_j^*(U_j)]\>_{\text{HS}} + \< [U_j, [p,Y_j]], \delta_j(p)\>_{\text{HS}}.
	\]

Next, using that $\vphi$ is a trace on $M^\vphi$, we have
	\begin{align*}
		{\II}_j = \<\delta_j(p), [ [ p,U_j],Y_j]\>_{\text{HS}} = \<\delta_j(p), [p, [U_j, Y_j]]\>_{\text{HS}} - \<\delta_j(p), [U_j, [p,Y_j]]\>_{\text{HS}}.
	\end{align*}
Note that $[U_j,Y_j]=1\otimes Y_j - Y_j\otimes 1$. So considering $[p, [U_j, Y_j]]$ appearing in the first term above, we have
	\begin{align*}
		[p, [U_j, Y_j]] &= (p\otimes 1-1\otimes p)\# [U_j,Y_j]\\
			&= [p, 1\otimes 1]\# [U_j,Y_j]\\
			&= [p, 1\otimes 1]\# (1\otimes Y_j - Y_j\otimes 1).
	\end{align*}
Since $\vphi\otimes\vphi^{op}$ is tracial on $M^\vphi\otimes(M^\vphi)^{op}$ we have
	\begin{align*}
		\sum_{j=1}^2 \<\delta_j(p), [p,[U_j,Y_j]]\>_{\text{HS}} &= \sum_{j=1}^2 \< \delta_j(p)\# (1\otimes Y_j - Y_j\otimes 1), [p,1\otimes 1]\>_{\text{HS}}\\
			&= \< -\delta_0(p) - [p,1\otimes 1], [p,1\otimes 1]\>_{\text{HS}}\\
			&= -\< \delta_0(p), [p,1\otimes 1]\>_{\text{HS}} - \|[p,1\otimes 1]\|_{\text{HS}}^2.
	\end{align*}
Thus we have shown
	\[
		\sum_{j=1}^2 -{\II}_j = \<\delta_0(p), [p,1\otimes 1]\>_{\text{HS}} + \| [p,1\otimes 1]\|_{\text{HS}}^2 + \sum_{j=1}^2 \<\delta_j(p), [U_j,[p,Y_j]]\>_{\text{HS}}
	\]
Combining this with out previous observations about ${\I}_j$, we have
	\begin{align*}
		\sum_{j=1}^2 \| [p,U_j]\|_{\text{HS}}^2 = \sum_{j=1}^2 &\left( \< [p,Y_j], [p,\delta_j^*(U_j)]\>_{\text{HS}}+2\Re \<\delta_j(p),[U_j,[p,Y_j]]\>_{\text{HS}}\right)\\
			& + \<\delta_0(p),[p,1\otimes 1]\>_{\text{HS}} + \| [p,1\otimes 1]\|_{\text{HS}}^2.
	\end{align*}
Subtracting $\|[p,U_1]\|_{\text{HS}}^2 = \|[p,1\otimes 1]\|_{\text{HS}}^2$ from each side gives the desired formula, provided we show
	\[
		\<\delta_j(p),[U_j,[p,Y_j]]\>_{\text{HS}}=\< (\vphi\otimes 1 - 1\otimes \vphi^{op})\circ \delta_j(p), [p,Y_j]\>_\vphi .
	\]
We compute
	\begin{align*}
		\<\delta_j(p),[U_j,[p,Y_j]]\>_{\text{HS}} &= (\vphi\otimes\vphi^{op})( \delta_j(p)^*\# (1\otimes [p,Y_j] - [p,Y_j]\otimes 1) \# U_j) \\
			&= (\vphi\otimes\vphi^{op})(U_j\# \delta_j(p)^*\# (1\otimes [p,Y_j] - [p,Y_j]\otimes 1))\\
			&= (\vphi\otimes\vphi^{op})(\delta_j(p)^*\# (1\otimes [p,Y_j] - [p,Y_j]\otimes 1)),
	\end{align*}
where we have used Lemma \ref{tracial_ultrapower_result}  to assert $\delta_j\# U_j=\delta_j$ in the last equality. Continuing, we obtain
	\begin{align*}
		\<\delta_j(p),[U_j,[p,Y_j]]\>_{\text{HS}} &=(\vphi\otimes\vphi^{op})(\delta_j(p)^*\# (1\otimes [p,Y_j] - [p,Y_j]\otimes 1))\\
			&= \vphi\left(\vphantom{\sum} [p,Y_j] (\vphi\otimes 1)(\delta_j(p)^*) - (1\otimes\vphi^{op})(\delta_j(p)^*) [p,Y_j]\right)\\
			&= \vphi\left(\vphantom{\sum} (\vphi\otimes 1)(\delta_j(p))^* [p,Y_j] - (1\otimes\vphi^{op})(\delta_j(p))^* [p,Y_j]\right)\\
			&=\< (\vphi\otimes 1 - 1\otimes \vphi^{op})\circ \delta_j(p), [p,Y_j]\>_\vphi,
	\end{align*}
as claimed.
\end{proof}

%%%%%%%%%%%%%%%%%%%%%%%%%%%%%%%%%%%%%%%%%%%%%%%%%%%%%%%%%%%%%%%%%%%%%%%%%%%%%%%%%%%%

\subsection{Extending $\bar\delta_{|y|}$ via the predual $(M\bar\otimes M^{op})_*$}

Fix $y\in G$ with $\lambda_y\leq 1$, so that $|y|$ is diffuse by Corollary \ref{atom_size}. Though we won't use it to prove either of our main theorems, there does exist a closed extension of $\bar{\delta}_{|y|}$ that is defined on $v$ and $v^*$ and sends both to zero. To see this extension, though, one must first expand the codomain of the derivations to the predual $(M\bar\otimes M^{op})_*$.

Let $\M$ be a von Neumann algebra with faithful normal state $\psi$. Recall that $L^2(\M,\psi)$ can be embedded into $\M_*$ the predual of $\M$ via
	\begin{align*}
		L^2(\M,\psi)\ni \xi\mapsto \omega_\xi\in \M_*,
	\end{align*}
where $\omega_\xi(x)=\<1,x\xi\>_\psi$ for $x\in \M$.

\begin{lem}\label{predual_bounded_by_L1_norm}
Suppose $a\in \M$ satisfies $|a|\in \M^\psi$. Then $\|\omega_a\|\leq \psi(|a|)$.
\end{lem}
\begin{proof}
If $\psi$ is a trace, then this follows by \cite[Equation V.2.(2)]{Tak02}. The proof presented here is identical modulo the additional non-tracial hypothesis.

Suppose $a$ has polar decomposition $a=v|a|$, and let $x\in M$ have polar decomposition $x=w|x|$. Then 
	\begin{align*}
		|\omega_a(x)|^2 = |\psi(xa)|^2 &= |\psi( |a|^\frac{1}{2} w|x|^{\frac12} |x|^{\frac12} v |a|^\frac{1}{2})|^2\\
			&\leq \psi( |a|^\frac{1}{2} w|x|w^* |a|^\frac{1}{2}) \psi( |a|^\frac{1}{2} v^*|x| v|a|^\frac{1}{2}).
	\end{align*}
But
	\begin{align*}
		w|x|w^* &= |x^*| \leq \|x\| 1\qquad \text{, and}\\
		v^*|x|v &\leq \|x\| v^*v \leq \|x\| 1,
	\end{align*}
so that continuing our previous estimate we have $|\omega_a(x)|^2 \leq \|x\|^2 \psi(|a|)^2$.
\end{proof}

Using the above embedding, we think of $\delta_{|y|}$ as a densely defined map
	\begin{align*}
		\delta_{|y|}\colon L^2(M,\varphi)\to (M\bar{\otimes}M^{op})_*
	\end{align*}
with adjoint
	\begin{align*}
		\delta_{|y|}^\star\colon M\bar{\otimes}M^{op}\to L^2(M,\varphi).
	\end{align*}

\begin{prop}
Fix $y\in G$ with $\lambda_y\leq 1$ and polar decomposition $y=v|y|$, and suppose $y\neq y^*$ and that $\xi_y$ exists. Then $\delta_{|y|}$ is closable as a densely defined map
	\begin{align*}
		\delta_{|y|}\colon L^2(M,\varphi)\to (M\bar{\otimes}M^{op})_*,
	\end{align*}
say with closure $\overline{\delta}^1_{|y|}$, and $v,v^*\in \dom(\overline{\delta}^1_{|y|})$ with $\overline{\delta}^1_{|y|}(v)=\overline{\delta}^1_{|y|}(v^*)=0$.
\end{prop}
\begin{proof}
The same proof as in Lemma \ref{abs_value_derivation_closable_in_L2} shows that $\delta_{|y|}^\star(1\otimes 1)\in L^2(M,\varphi)$ exists and that $\P\otimes \P^{op}\subset \dom(\delta_{|y|}^\star)$. We claim that the weak density of $\P\otimes \P^{op}$ in $M\bar{\otimes}M^{op}$ suffices to establish closability of $\delta_{|y|}$. Indeed, suppose $\dom{(\delta_{|y|})}\ni x_n\to 0$ in $L^2(M,\varphi)$ and $\delta_{|y|}(x_n)\to \omega\in (M\bar{\otimes}M^{op})_*$. Given $\epsilon>0$, let $\zeta\in \P\otimes\P^{op}$ be such that
	\begin{align*}
		\|\omega\| &\leq |\omega(\zeta)|+\epsilon.
	\end{align*}
Then, we have
	\begin{align*}
		\|\omega\| & \leq \lim_{n\to\infty} |\<1\otimes 1, \zeta \#\delta_{|y|}(x_n)\>_{\text{HS}}|+\epsilon\\
				& = \lim_{n\to\infty} |\<\delta_{|y|}^\star(\zeta^*),x_n\>_\varphi| + \epsilon = \epsilon.
	\end{align*}
Thus $\omega=0$ and $\delta_{|y|}$ is closable as a map into $(M\bar\otimes M^{op})_*$. Denote this closure by $\overline{\delta}^1_{|y|}$ and note that it is an extension of $\bar\delta_{|y|}$ by the aforementioned embedding.

Recall from the proof of Lemma \ref{abs_value_derivation_closable_in_L2}, that $y(\alpha+|y|)^{-1}$ converges in $L^2(M,\vphi)$ to $v$ as $\alpha\to 0$. Furthermore, $(\alpha+|y|)^{-1}\in \dom(\bar\delta_{|y|})\subset \dom{(\overline{\delta}^1_{|y|})}$ by Corollary \ref{extC1}. The computation
	\begin{align*}
		\frac{\frac{1}{\alpha+t} - \frac{1}{\alpha+s}}{t-s} = \frac{-1}{(\alpha+t)(\alpha+s)}
	\end{align*}
along with the formulas in Lemma \ref{abs_value_derivation_closable_in_L2} imply
	\begin{align*}
		\overline{\delta}^1_{|y|}( (\alpha+|y|)^{-1}) = -\left[(\alpha+|y|)^{-1}\otimes (\alpha+|y|)^{-1}\right]\# \bar\delta_{|y|}(|y|)=-(\alpha+|y|)^{-1}\otimes (\alpha+|y|)^{-1}
	\end{align*}
(note $v^*v=1$ since $\lambda_y\leq 1$). By Lemma \ref{extension_of_Leibniz_rule}, $y(\alpha+|y|)^{-1}\in \dom(\bar\delta_{|y|})\subset \dom(\overline{\delta}^1_{|y|})$ with
	\begin{align*}
		\overline{\delta}^1_{|y|}( y(\alpha+|y|)^{-1}) &= v\otimes (\alpha+|y|)^{-1} - y(\alpha+|y|)^{-1}\otimes (\alpha+|y|)^{-1}\\ 
			& = v\left[ 1 - |y|(\alpha+|y|)^{-1}\right]\otimes (\alpha+|y|)^{-1}\\
			& = v \alpha(\alpha+|y|)^{-1}\otimes (\alpha+|y|)^{-1}\\
			& = v\sqrt{\alpha}(\alpha+|y|)^{-1}\otimes \sqrt{\alpha}(\alpha+|y|)^{-1},
	\end{align*}
It is easy to see that
	\begin{align*}
		|v\sqrt{\alpha}(\alpha+|y|)^{-1}\otimes \sqrt{\alpha}(\alpha+|y|)^{-1}| = \sqrt{\alpha}(\alpha+|y|)^{-1}\otimes \sqrt{\alpha}(\alpha+|y|)^{-1} \in (M\bar{\otimes}M^{op})^{\varphi\otimes\varphi^{op}},
	\end{align*}
which can be identified with the function
	\begin{align*}
		g_\alpha(t,s)= \frac{\sqrt{\alpha}}{\alpha+t}\frac{\sqrt{\alpha}}{\alpha+s}.
	\end{align*}
Now, let $I\subset [0,\infty)$ be a compact interval containing the spectrum of $|y|$, let $\mu$ be the spectral measure of $|y|$, and let $m$ be the Lebesgue measure. Since $\bar{\delta}_{|y|}(|y|)=1\otimes 1$, $\bar{\delta}_{|y|}$ is the free difference quotient for $|y|$, and hence $1\otimes 1 \in \dom(\delta_{|y|}^*)$ implies that $|y|$ has finite free Fisher information. Consequently, \cite{VoiI} and \cite{VoiV} imply that $p:=d\mu/dm \in L^3(\R,m)$. Thus, by H\"older's inequality we have
	\[
		\|g_\alpha\|_{L^1(I\times I, \mu\times \mu)} \leq \| g_\alpha\|_{L^\frac{3}{2}(I\times I,m\times m)} \|p\|^2_{L^3(I,m)}
	\]
An easy computation shows $\|g_\alpha\|_{L^{3/2}(I\times I,m\times m)}\to 0$ as $\alpha\to 0$. Lemma \ref{predual_bounded_by_L1_norm} then implies $\overline{\delta}^1_{|y|}(y(\alpha+|y|)^{-1})\to 0$ in $(M\bar{\otimes}M^{op})_*$ as $\alpha\to 0$, and hence $v\in \dom(\overline{\delta}^1_{|y|})$ with $\overline{\delta}^1_{|y|}(v)=0$. The proof for $v^*$ is similar, but does use that $y$ and $v$ belong to the same eigenspace $E_\lambda$.
\end{proof}

\begin{rem}
Note that unless the spectrum of $|y|$ is bounded away from zero, $g_\alpha(s,t)^2$ does not converge to zero in $L^\frac{3}{2}(I\times I,m\times m)$. This means the above argument cannot show that $g_\alpha(s,t)\to 0$ in $L^2(I\times I,\mu\times \mu)$.
\end{rem}

Let $\bar{\delta}_{|y|}$ denote the closure of $\delta_{|y|}$ as a map into $L^2(M\bar{\otimes}M^{op})$. By the above proposition, we can consider an extension $d_{|y|}$ of $\bar{\delta}_{|y|}$ defined on $\text{span}\{\C\< v, v^*\>\cup \dom(\bar{\delta}_{|y|})\}$ so that $d_{|y|}(v)=d_{|y|}(v^*)=0$, the rest of its definition being determined by the Leibniz rule and $\bar{\delta}_{|y|}$. As a map into $(M\bar{\otimes}M^{op})_*$, it is clear that $d_{|y|}$ is closable with closure $\overline{\delta}^1_{|y|}$; however, this also implies that $d_{|y|}$ is closable as a map into $L^2(M\bar{\otimes}M^{op})$. Indeed, if $(x_n)_{n\in\N}\subset \dom(d_{|y|})$ converges to zero in $L^2(M,\varphi)$ and $(d_{|y|}(x_n))_{n\in\N}$ converges to some $\eta\in L^2(M\bar{\otimes}M^{op})$, then (since $\|\omega_\zeta\|\leq \|\zeta\|_{\text{HS}}$ for $\zeta\in L^2(M\bar{\otimes}M^{op})$) it follows that $(d_{|y|}(x_n))_{n\in\N}$ converges to $\omega_\eta$ as elements of $(M\bar{\otimes}M^{op})_*$. Hence $\omega_\eta=0$, since $\overline{\delta}^1_{|y|}$ is closed, which then implies $\eta=0$ because
	\begin{align*}
		\<a,\eta\>_{\text{HS}} = \omega_\eta(a^*)=0\qquad \forall a\in M\bar{\otimes}M^{op}.
	\end{align*}
Thus, as maps into $L^2(M\bar{\otimes}M^{op})$, the closure of $d_{|y|}$ is a closed extension of $\bar{\delta}_{|y|}$ with the claimed properties regarding $v,v^*$.

%%%%%%%%%%%%%%%%%%%%%%%%%%%%%%%%%%%%%%%%%%%%%%%%%%%%%%%%%%%%%%%%%
%					Main Results					%
%%%%%%%%%%%%%%%%%%%%%%%%%%%%%%%%%%%%%%%%%%%%%%%%%%%%%%%%%%%%%%%%%

\section{Main Results}\label{main_results}

We conclude with the proofs of our main results. We also include a few easy corollaries that minimize the hypotheses.

\begin{cthm}{A}
Let $M$ be a von Neumann algebra with a faithful normal state $\varphi$. Suppose $M$ is generated by a finite set $G=G^*$, $|G|\geq 2$, of eigenoperators of $\sigma^\varphi$ with finite free Fisher information. Then $(M^\vphi)'\cap M=\C$. In particular, $M^\varphi$ is a $\II_1$ factor and if $H< \R_+^\times$ is the closed subgroup generated by the eigenvalues of $G$ then
	\begin{align*}
		\text{$M$ is a factor of type} \left\{\begin{array}{cl} \rm{III}_1 & \text{if  }\ H=\R_+^\times\\
													\rm{III}_\lambda	& \text{if }\ H=\lambda^\Z,\ 0<\lambda<1\\
													\rm{II}_1 & \text{if }\ H=\{1\}. \end{array}\right.
	\end{align*}
\end{cthm}
\begin{proof}
%First note that $M^\varphi \cap \C\<G\>$ is $\|\cdot\|_\varphi$-dense in $M^\varphi$. Indeed, if $\C\<G\>$ is operator norm dense in $M$, so given $z\in M^\varphi$ we can find $\{p_n\}\in \C\<G\>$ which converge to $z$ in the $\|\cdot\|_\varphi$-norm. Noting that
%	\begin{align*}
%		\| z-p_n\|_\varphi^2 = \| \mc{E}_\varphi( z- p_n)\|_\varphi^2 + \| (1- \mc{E}_\varphi)(p_n)\|_\varphi^2,
%	\end{align*}
%we see that $\{\mc{E}_\varphi(p_n)\}$ also converges to $z$ in the $\|\cdot\|_\varphi$-norm. But every monomial $q\in \C\<G\>$ is an eigenvector of $\Delta_\varphi$ and so either $\mc{E}_\varphi(q)=q$ or $\mc{E}_\varphi(q)=0$. Consequently $\mc{E}_\varphi(p_n)\in M^\varphi\cap \C\<G\>$.
  
By Corollary \ref{cor:bicentralizer}, $(M^\vphi)'\cap M = (M^\vphi)'\cap M^\vphi$. Also $M'\cap M \subset M'\cap M^\vphi\subset (M^\vphi)'\cap M^\vphi $, so it suffices to show $M^\vphi$ is a factor. Fix $z\in (M^\varphi)'\cap M^\vphi$. Pick $y\in G$ such that $y\neq y^*$ (if no such $y$ exists then $M^\vphi=M$ and we simply appeal to the tracial result \cite[Theorem 1]{Dab10}). By replacing $y$ with $y^*$ if necessary, we may assume $\lambda_y\leq 1$ so that $y^*y$ is diffuse by Corollary \ref{diffuse_centralizer_elements}. Let $\delta:=\delta_2$ be as in Lemma \ref{tracial_ultrapower_result} with $Y_2=|y^*|$. Then 		
	\[
		\delta(y^*y)=\mc{E}_\vphi\otimes\mc{E}_\vphi^{op}(v^*\otimes y + y^*\otimes v)=0
	\]
where $y^*=v^*|y^*|$ is the polar decomposition. Let $\{\zeta_\alpha\}_{\alpha>0}$ be the maps derived from the contraction resolvent associated to $L=\delta^*\bar{\delta}$. Then by Lemma \ref{comm_for_cont_res} we have
	\[
		0=\zeta_\alpha([z,y^*y])=[\zeta_\alpha(z),y^*y],
	\]
Then Proposition \ref{closure_is_der} implies
	\[
		0=\bar{\delta}([\zeta_\alpha(z),y^*y])=[\bar{\delta}\circ\zeta_\alpha(z), y^*y].
	\]
Hence $\bar{\delta}\circ\zeta_\alpha(z)$ is a Hilbert--Schmidt operator commuting with a diffuse operator, implying $\bar{\delta}\circ\zeta_\alpha(z)=0$ for all $\alpha>0$. Since $\lim_{\alpha\to\infty} \zeta_\alpha(z)= z$ by (\ref{zeta_to_identity}) we see that $z\in \dom(\bar{\delta})$ with $\bar{\delta}(z)=0$. But then, invoking Proposition \ref{closure_is_der} again as well as the computations in Lemma \ref{abs_value_derivation_closable_in_L2}, we have
	\[
		0=\bar{\delta}([z,|y^*|]) = [z,p\otimes 1+1\otimes p - p\otimes p],
	\]
where $p=vv^*$. Applying $1\otimes \vphi$ yields
	\[
		zp + z(1-p)\vphi(p) = p\vphi(z) + (1-p)\vphi(pz).
	\]
Multiplying both sides by $p$ and applying $\vphi$ reveals that $\vphi(zp)=\vphi(z)\vphi(p)$. So we can rewrite the above equation as:
	\[
		zp+z(1-p)\vphi(p)= p \vphi(z) + (1-p)\vphi(p)\vphi(z).
	\]
Then multiplying by $p+\frac{1}{\vphi(p)}(1-p)$ yields
	\[
		zp+z(1-p)=\vphi(z) p + \vphi(z) (1-p),
	\]
or $z=\vphi(z)\in \C$ as claimed.

Now, $M^\varphi$ is a factor containing the diffuse element $y^*y$; that is, $M^\varphi$ is a $\II_1$ factor. As for the type classification of $M$, first note that if $H=\{1\}$ then $G\subset M^\varphi$ and hence $M=M^\varphi$ is a $\II_1$ factor. Otherwise, we appeal to the modular spectrum $S(M)$.

Recall the notation established in Subsection \ref{arveson_connes}. Since $M$ and $M^\varphi$ are factors, $S(M)\cap \R_+^\times = \Gamma(\sigma^\varphi)$ by Lemma \ref{Connes_first_lemma} and $\Gamma(\sigma^\varphi)=\text{Sp}(\sigma^\varphi)$ by Lemma \ref{Connes_second_lemma}. Thus, by Connes's classification it suffices to show $H=\text{Sp}(\sigma^\varphi)$. Let $\lambda$ be an eigenvalue for some non-zero monomial $p\in \C\<G\>$ (hence $\lambda\in H$). Then for any $f\in \bigcap_{x\in M} I(x)$ we have
	\begin{align*}
		0 = \sigma_f^\varphi(p) = \int_\R \check{f}(-t)\sigma_t^\varphi(p)\ dt = \int_\R \check{f}(-t)  \lambda^{it}\ dt\cdot p = f(\lambda) p.
	\end{align*}
Hence $f(\lambda)=0$ and therefore $\lambda\in \text{Sp}(\sigma^\varphi)$. Since the Arveson spectrum is closed, this implies $H\subset \text{Sp}(\sigma^\varphi)$. If $H=\R_+^\times$, then this forces equality. Otherwise, given $\mu\in \R_+^\times\setminus H$ there is an open set $U\subset \R_+^\times\setminus H$ containing $\mu$. Suppose $f\in A(\R_+^\times)$ satisfies $\text{supp}(f)\subset U$. Then any monomial $p\in \C\<G\>$ has an eigenvalue $\lambda\not\in U$ and hence
	\begin{align*}
		\sigma_f^\varphi(p) = f(\lambda) p = 0.
	\end{align*}
Consequently $M^{\sigma^\varphi}_0(U)=\{0\}$, and so $\mu\not\in \text{Sp}(\sigma^\varphi)$ by Lemma \ref{M_0_Takesaki_Lemma}.
\end{proof}

Other than in the type classification of $M$, the above proof only required that the other generators $G\setminus\{y,y^*\}$ were annihilated by $\delta_y$ and $\delta_{y^*}$. Consequently, we have the following corollary.

\begin{cor}\label{minimalist_corollary}
Suppose $M$ is a von Neumann algebra with a faithful normal state $\varphi$. Let $y\in M$ be an eigenoperator of $\sigma^\varphi$ and $B\subset M$ a unital $*$-subalgebra which is globally invariant under $\sigma^\varphi$. Letting $N$ denote the von Neumann algebra generated by $B$ and $y$, if $\Phi^*_\vphi(y,y^*\colon B)<\infty$ then $N^{\varphi\mid_N}$ is a $\rm{II}_1$ factor and $N$ is a factor.
\end{cor}

We note that $\Phi^*_\vphi(y,y^*\colon B)<\infty$ in particular implies that $y$ and $y^*$ are distinct elements which are algebraically free over $B$.

\begin{cthm}{B}
Let $M$ be a von Neumann algebra with a faithful normal state $\varphi$. Suppose $M$ generated by a finite set $G=G^*$, $|G|\geq 2$, of eigenoperators of $\sigma^\varphi$ with finite free Fisher information. Then $M^\varphi$ does not have property $\Gamma$. Furthermore, if $M$ is a type $\III_\lambda$ factor, $0<\lambda<1$, then $M$ is full.
\end{cthm}
\begin{proof}
Pick $y\in G$ such that $y\neq y^*$ (if no such $y$ exists, then $M^\vphi=M$ and we simply appeal to the tracial result \cite[Theorem 13]{Dab10}). By replacing $y$ with $y^*$ if necessary, we may assume $\lambda_y\leq 1$ so that $y^*y$ is diffuse by Corollary \ref{diffuse_centralizer_elements}. Letting $y=v|y|$ be the polar decomposition, it follows that $v^*v=1$.

Let $Y_1=|y|$, $Y_2=|y^*|$, $\delta_1$ and $\delta_2$ as in Lemma \ref{tracial_ultrapower_result}, and $\delta_0$ as in Lemma \ref{lem:closable_big_kernel}. We claim that $W^*(Y_1,Y_2)$ does not have property $\Gamma$ and consequently is non-amenable. Suppose $(z_n)_{n\in \N}\subset W^*(Y_1,Y_2)$ is a uniformly bounded, asymptotically central sequence. Without loss of generality, we may assume $\vphi(z_n)=0$ for all $n\in\N$. By Lemma \ref{lem:closable_big_kernel}, $\bar\delta_0(z_n)=0$ for all $n\in \N$ and so (\ref{eqn:vanishing_derivation}) implies
	\[
		\| [z_n, U_2]\|_{\text{HS}}^2 = \sum_{j=1}^2 \left( \<[z_n,Y_j], [z_n,\delta_j^*(U_j)]\>_{\vphi} + 2\Re\< (\vphi\otimes 1 - 1\otimes \vphi)\circ \bar\delta_j(z_n), [z_n, Y_j]\>_{\vphi}\right).
	\]
Thus by the uniform boundedness of $(z_n)$ and Corollary \ref{boundedness_of_adjoint_formula}, we have $\|[z_n,U_2]\|_{\text{HS}}^2\to 0$. Let $p=vv^*$ then one easily checks that
	\begin{align}\label{eqn:expanded_comm_U2}
		\|[z_n,U_2]\|_{\text{HS}}^2 = \| z_np + \sqrt{\vphi(p)} z_n(1-p) \|^2_\vphi + \| pz_n + \sqrt{\vphi(p)}(1-p) z_n\|_\vphi^2 + 2|\<p,z_n\>_\vphi|^2.
	\end{align}
Therefore
	\begin{align*}
		\|z_n\|_\vphi^2 &= \left\| \left( p + \frac{1}{\sqrt{\vphi(p)}}(1-p)\right)\left(pz_n + \sqrt{\vphi(p)}(1-p) z_n\right)\right\|_\vphi^2\\
			&\leq \left\|p + \frac{1}{\sqrt{\vphi(p)}}(1-p)\right\|^2 \| [z_n,U_2]\|_{\text{HS}}^2,
	\end{align*}
which tends to zero. Thus $W^*(Y_1,Y_2)$ does not have property $\Gamma$, and consequently is non-amenable.

Now, as $W^*(Y_1,Y_2)$ is non-amenable, it admits a non-amenability set (see \cite[Definition 2.4]{DI16}). That is, there exists $S\subset W^*(Y_1,Y_2)$ and constant $K>0$ such that
	\begin{align}\label{eqn:non-amen_set}
		\|\eta\|_{\text{HS}} \leq K \sum_{x\in S} \| [x, \eta] \|_{\text{HS}} \qquad \forall \eta\in L^2(W^*(Y_1,Y_2)\bar\otimes W^*(Y_1,Y_2)).
	\end{align}
Let $\{\zeta_\alpha\}_{\alpha>0}$ be the maps derived from the contraction resolvent associated to $L_0:=\delta_0^*\bar\delta_0$. For $\omega$ a free ultra-filter, let $(z_n)_{n\in \N}\in W^*(Y_1,Y_2)'\cap (M^\vphi)^\omega$, which we can assume satisfies $\vphi(z_n)=0$ for all $n\in\N$. Fix $\alpha>0$, then for each $n\in \N$ (\ref{eqn:vanishing_derivation}), Corollary \ref{boundedness_of_adjoint_formula}, and (\ref{eqn:non-amen_set})  imply
	\begin{align*}
		\|[\zeta_\alpha(z_n),U_2]\|_{\text{HS}}^2 \leq& \sum_{j=1}^2 \| [\zeta_\alpha(z_n), Y_j]\|_{\vphi} \|z_n\| \left( 2\|\delta_j^*(U_j)\|_\vphi + 12 \|\delta_j^*(1\otimes 1)\|_\vphi\right)\\
			& + K\|[\zeta_\alpha(z_n),1\otimes 1]\|_{\text{HS}} \sum_{x\in S} \| [x, \bar\delta_0\circ\zeta_\alpha(z_n)]\|_{\text{HS}}.
	\end{align*}
Lemma \ref{lem:closable_big_kernel} implies that $W^*(Y_1,Y_2)\subset \ker(\bar\delta_0)$. So by Lemma \ref{comm_for_cont_res} we have $[\zeta_\alpha(z_n), Y_j] = \zeta_\alpha( [z_n, Y_j])$ for $j=1,2$, and $[x,\bar\delta_0\circ \zeta_\alpha(z_n)] = \bar\delta_0\circ\zeta_\alpha( [x,z_n])$ for all $x\in S$. Thus the above estimate implies $\lim_n \|[\zeta_\alpha(z_n),U_2]\|_{\text{HS}}^2=0$. By \cite[Theorem 4.3]{Pet09}, $\{\eta_\alpha\}_{\alpha>0}$ converges uniformly on $( W^*(Y_1,Y_2)'\cap (M^\vphi)^\omega)_1$. By (\ref{zeta_formula}), $\{\zeta_\alpha\}_{\alpha>0}$ also converges uniformly. Using (\ref{eqn:expanded_comm_U2}) we therefore have
	\[
		\lim_{n\to\infty} \|[z_n, U_2]\|_{\text{HS}}^2 = \lim_{n\to\infty} \lim_{\alpha\to\infty} \| [\zeta_\alpha(z_n),U_2]\|_{\text{HS}}^2 = \lim_{\alpha\to\infty} \lim_{n\to\infty} \| [\zeta_\alpha(z_n),U_2]\|_{\text{HS}}^2=0.
	\]
From this and the same computation which showed $W^*(Y_1,Y_2)$ does not have property $\Gamma$, we have $\| z_n\|_\vphi^2\to 0$. Thus $W^*(Y_1,Y_2)'\cap (M^\vphi)^\omega=\C$, and in particular $M^\vphi$ does not have property $\Gamma$.

As discussed in Subsection \ref{full}, $M$ is full if and only if any uniformly bounded sequence $(x_n)_{n\in\N}$ of $\varphi$-centered elements of $M$ in the asymptotic centralizer with respect to $\phi$ for all $\phi\in M_*$ converges $*$-strongly to zero. Recall that on uniformly bounded subsets of $M$ the $*$-strong topology coincides with the topology defined by the norm $\|x\|_\varphi^\#:=\varphi(x^*x+xx^*)^{1/2}$, $x\in M$. 

Let $(x_n)_{n\in\N}\subset M$ be a sequence satisfying the above hypothesis. From the proof of Theorem \ref{A}, we know that if $M$ is of type $\III_\lambda$, $0<\lambda<1$, then $S(M)=\{0\}\cup\lambda^\Z$. So by Lemma \ref{Connes_first_lemma}, $\text{spectrum}(\Delta_\varphi)=\{0\}\cup\lambda^\Z$ and hence 1 is isolated in the spectrum of $\Delta_\varphi$. Thus we may apply \cite[Proposition 2.3.(2)]{Con74} to assert that $\|x_n - \mc{E}_\varphi(x_n)\|^\#_\varphi\to 0$. So, it suffices to show $\|\mc{E}_\varphi(x_n)\|_\varphi^\# = \sqrt{2}\|\mc{E}_\varphi(x_n)\|_\varphi\to 0$. We will show that $(\mc{E}_\varphi(x_n))_{n\in\N}$ is asymptotically central and therefore converges to zero in $L^2(M^\varphi,\varphi)$ since $M^\varphi$ does not have property $\Gamma$.

Fix $z\in M^\varphi$ and set $a_n:=[\mc{E}_\varphi(x_n),z]$. Then
	\begin{align*}
		\| [\mc{E}_\varphi(x_n), z] \|_\varphi^2 &= \varphi( a_n^* \mc{E}_\varphi(x_n)z - a_n^* z  \mc{E}_\varphi(x_n))\\
			& = \varphi([a_n^*, \mc{E}_\varphi(x_n)] z)\\
			& = \varphi( [a_n^*, x_n] z) + \varphi( [a_n^*, \mc{E}_\varphi(x_n) - x_n]z).
	\end{align*}
Letting $\phi(\cdot):= \varphi(\cdot z) \in M_*$, we see that the first term in the last expression above is bounded by $\|a_n^*\| \| [x_n, \phi]\|$, which tends to zero since $(x_n)_{n\in\N}$ is uniformly bounded and in the asymptotic centralizer with respect to $\phi$. Thus it suffices to bound the second term:
	\begin{align*}
		\varphi( [a_n^*, \mc{E}_\varphi(x_n) - x_n]z) &= \varphi( (\mc{E}_\varphi(x_n) - x_n) [z,a_n^*])\\
										&\leq \| \mc{E}_\varphi(x_n) - x_n \|_\varphi^\# \| [z, a_n^*]\|_\varphi\\
										&\leq \| \mc{E}_\varphi(x_n) - x_n \|_\varphi^\# 4 \|\mc{E}_\varphi(x_n)\| \|z\|^2.
	\end{align*}
This tends to zero and completes the proof.
\end{proof}

We are again grateful to Yoann Dabrowski who suggested the above proof showing that $M^\vphi$ does not have property $\Gamma$. In particular, this allowed us to replace the hypothesis $|G|\geq 3$ with $|G|\geq 2$.

As with Corollary \ref{minimalist_corollary}, we used only that $G\setminus \{y_,y^*\}$ is annihilated by $\delta_y$ and $\delta_{y^*}$. Hence we have the following corollary.

\begin{cor}
Suppose $M$ is a von Neumann algebra with a faithful normal state $\varphi$. Let $y\in M$ be an eigenoperator of $\sigma^\varphi$ and $B\subset M$ a unital $*$-subalgebra which is globally invariant under $\sigma^\varphi$. Letting $N$ denote the von Neumann algebra generated by $B$ and $y$, then $N^{\varphi\mid_N}$ does not have property $\Gamma$, and if $1$ is isolated in the spectrum of $\Delta_{\vphi\mid_N}$ then $N$ is full.
\end{cor}

%%%%%%%%%%%%%%%%%%%%%%%%
%                        	Bibiliography                                %
%%%%%%%%%%%%%%%%%%%%%%%%

\bibliography{references}

\end{document}